\documentclass[11pt,reqno,english,empty]{amsart}
\usepackage{array}
\usepackage{frcursive}
\usepackage[hypertex,
  colorlinks=true,
 linkcolor=blue,
 citecolor=blue,
]{hyperref}
\usepackage[T1]{fontenc}
\usepackage{mathrsfs}
\usepackage{amsmath,amsthm,amssymb}
\usepackage{latexsym}
\usepackage{enumerate}
\usepackage{mathrsfs}
\usepackage{stmaryrd}
\usepackage{amsopn}
\usepackage{amsmath}
\usepackage{amssymb}
\usepackage{amsfonts}
\usepackage{amsbsy}
\usepackage{amscd,indentfirst,epsfig}
\usepackage{amsfonts,amsmath,latexsym,amssymb,verbatim,amsbsy}
\usepackage{amsthm}
\usepackage{colordvi}
\usepackage{pstricks}

\usepackage{ulem}
\def\al{\alpha}

\def\R{\mathbb R}
\def\S{\mathbb S}

\def\N{\mathbb N}

\def\C{\mathcal C}

\def\D{\mathscr D}

\def\M{\mathcal M}

\def\s{\sigma}
\def \fe {\varphi_{\alpha}}
\def \gl {\psi_{\alpha}}
\def\a{\alpha}

\def\l{\ell}

\def\k{\kappa}

\def\d{\mathrm{d}}

\def\Q{\mathcal{Q}}

\newtheorem{theo}{Theorem}[section]
\newtheorem{prop}[theo]{Proposition}
\newtheorem{cor}[theo]{Corollary}
\newtheorem{lem}[theo]{Lemma}

\newtheorem{rmq}[theo]{Remark}
\def \leq {\leqslant}
\def \geq {\geqslant}

\def \RF {\mathcal{RF}}
\def \F {\mathcal{F}}
\numberwithin{equation}{section}

\def\beq{\begin{equation}}
\def\eeq{\end{equation}}
\def\beqn{\begin{equation*}}
\def\eeqn{\end{equation*}}
\def\bea{\begin{eqnarray}}
\def\eea{\end{eqnarray}}
\def\bean{\begin{eqnarray*}}
\def\eean{\end{eqnarray*}}
\def\bary{\begin{array}}
\def\eary{\end{array}}

\title[On ballistic annihilation]
{\textbf{Uniqueness of the self-similar profile for a kinetic annihilation model
}}

\author{Véronique Bagland \& Bertrand Lods }

\address{\textbf{Véronique Bagland}, Clermont Universit\'e, Universit\'{e} Blaise Pascal, Laboratoire de Math\'{e}matiques, CNRS UMR 6620,  BP 10448, F-63000 Clermont-Ferrand,
France.}\email{Veronique.Bagland@math.univ-bpclermont.fr}

\address{\textbf{Bertrand Lods}, Universit\`{a} degli
Studi di Torino \& Collegio Carlo Alberto, Department of Economics and Statistics, Corso Unione Sovietica, 218/bis, 10134 Torino, Italy.}\email{lods@econ.unito.it}

\begin{document}

\maketitle

\begin{abstract}
We prove the uniqueness of the self-similar profile solution for a modified Boltzmann equation describing probabilistic ballistic annihilation. Such a model describes a system of hard spheres such that, whenever two particles meet, they either annihilate with probability $\al \in (0,1)$ or they undergo an elastic collision with probability $1-\al$. The existence of a self-similar profile for $\al$ smaller than an explicit threshold value $\underline{\alpha}_1$ has been obtained in our previous contribution \cite{jde}. We complement here our analysis of such a model by showing that, for some $\alpha^{\sharp} $ explicit, the self-similar profile is unique for $\al \in (0,\alpha^{\sharp})$. \\

\end{abstract}

\tableofcontents

\section{Introduction}

We investigate in the present paper a kinetic model, recently introduced in  \cite{Ben-Naim,coppex04,coppex05,Maynar1,Maynar2,Trizac}, which describes the so-called probabilistic ballistic annihilation of hard-spheres. In  such a description, a system of (elastic) hard spheres interact according to the following mechanism: they freely move between collisions while, whenever two particles meet, they either annihilate with probability $\alpha \in (0, 1)$ or they undergo an elastic collision with probability $1-\alpha$. In the spatially homogeneous situation, the  velocity distribution $f(t,v)$  of particles with velocity $v \in \R^d$ $(d \geq 2)$ at time $t \geq 0$ satisfies the following
\begin{equation}\label{BE}
\partial_t f(t,v)=(1-\alpha)\Q(f,f)(t,v) -\alpha \Q^-(f,f)(t,v)
\end{equation}
where $\Q$ is the quadratic Boltzmann collision operator defined by
 \begin{equation*}\label{bilin}
 \Q(g,f)(v) =\int_{\R^d \times \S^{d-1}} |v-v_*|
         \Big( g(v') f(v_*') - g(v) f(v_*) \Big) \,\frac{\d v_*\,\d\sigma}{|\mathbb{S}^{d-1}|},
 \end{equation*}
 where the post-collisional velocities $v'$ and $v'_*$ are parametrized by
 \begin{equation*}\label{eq:rel:vit}
 v'  =  \frac{v+v_*}{2} + \frac{|v-v_*|}{2}\;\s,   \qquad
v'_* =  \frac{v+v_*}{2} - \frac{|v-v_*|}{2}\;\s,   \qquad \s  \in  {\S}^{d-1}.
 \end{equation*}
The above collision operator  $\Q(g,f)$ splits as $\Q(g,f)=\Q^+(g,f)-\Q^-(g,f)$ where the gain part $\Q^+$ is given by
 $$ \Q^+(g,f)(v) =  C_d \int _{\R^d \times {\S}^{d-1}}|v-v_*| f(v'_*) g(v')\, \d v_* \,  \d\sigma$$
with $C_d=1/|\mathbb{S}^{d-1}|$  while the loss part $\Q^-$ is defined as
 $$\Q^-(g,f)(v)=g(v)L(f)(v), \qquad \text{ with } \qquad L(f)(v)=\int_{\R^d}|v-v_*|f(v_*)\d v_*.$$
For the sequel of the paper, we shall define the annihilation operator
$$\mathbb{B}_\al(f,f):=(1-\alpha)\Q(f,f) -\alpha \Q^-(f,f)=(1-\alpha)\Q^+(f,f) -\Q^-(f,f).$$

We 	refer to \cite{jde} and the references therein for a more detailed description of the above annihilation model. Throughout the paper we shall use the notation
$\langle \cdot \rangle = \sqrt{1+|\cdot|^2}$. We denote, for any
$\eta \in \R$, the Banach space
\[
     L^1_\eta(\R^d) = \left\{f: \R^d \to \R \hbox{ measurable} \, ; \; \;
     \| f \|_{L^1_\eta} := \int_{\R^d} | f (v) | \, \langle v \rangle^\eta \d v
     < + \infty \right\}.
\]

\subsection{Self-similar solutions}
From the mathematical viewpoint, the well-posedness of Equation \eqref{BE} has been studied in our previous contribution \cite{jde} where it is  proved that if $f_0\in L^1_{3}(\R^d)$ is a nonnegative distribution function, then,  there exists a unique nonnegative weak solution 
$$f\in\C([0,\infty);L_2^1(\R^d)) \cap L^1_{\mathrm{loc}}((0,\infty);L^1_{3}(\R^d))$$
to \eqref{BE} such that $f(0,\cdot)=f_0$. Moreover, multiplying \eqref{BE} by $1$ or $|v|^2$ and integrating with respect to $v$, one  obtains
$$\dfrac{\d}{\d t}\int_{\R^d}f(t,v)\d v=-\alpha \int_{\R^d}\Q^-(f,f)(t,v)\d v \leq 0$$
and
$$\dfrac{\d}{\d t}\int_{\R^d}|v|^2\,f(t,v)\d v=-\alpha \int_{\R^d} |v|^2\Q^-(f,f)(t,v)\d v \leq 0.$$
It is clear therefore that \eqref{BE} does not admit any non trivial steady
solution and, still formally, $f(t,v) \to 0$ as $t \to \infty.$
According to physicists, solutions to \eqref{BE}
should approach for large times a self-similar solution $F=F_{\alpha}$ (depending \textit{a priori} on the parameter $\alpha$) to \eqref{BE} of
the form
\begin{equation}\label{scalingfpsi}F_{\alpha}(t,v)=\lambda(t)\,\psi_\alpha(\beta(t)v)\end{equation}
for some suitable scaled functions $\lambda(t), \beta(t) \geq 0$  
and some nonnegative function  $\psi_\a=\psi_\a(\xi)$ such that
\beq\label{bas}
 \psi_\a \ \equiv \hspace{-3.5mm} / \hspace{2mm} 0 \qquad \mbox{ and } \qquad
\int_{\R^d} \psi_\alpha(\xi) \, (1+|\xi|^2)\, \d\xi <\infty.
\eeq
Notice that, as observed in \cite{jde}, $F_{\alpha}(t,v)$ is a solution to \eqref{BE} if and
only if $\psi_\alpha(\xi)$ is a solution to the rescaled problem
{$$\dfrac{\dot{\lambda}(t)\beta^{d+1}(t)}{\lambda^2(t)}\psi_\alpha(\xi)
+\dfrac{\dot{\beta}(t)\beta^{d}(t)}{\lambda(t)}\xi \cdot \nabla_\xi \psi_\alpha(\xi)
=\mathbb{B}_{\alpha}(\psi_\alpha,\psi_\alpha)(\xi)$$}
where the dot symbol stands for the time derivative. The profile $\psi_\alpha$ being independent of time $t$, there should exist some constants $\mathbf{A}=\mathbf{A}_{\psi_\alpha}$ and $\mathbf{B}=\mathbf{B}_{\psi_{\alpha}}$  such that
$\mathbf{A}=\frac{\dot{\lambda}(t)\beta^{d+\gamma}(t)}{\lambda^2(t)},$ and $\mathbf{B}=\frac{\dot{\beta}(t)\beta^{d+\gamma-1}(t)}{\lambda(t)}$.
Thereby, $\psi_\alpha$ is a solution to
\begin{equation*}
\mathbf{A}\psi_\alpha(\xi)+\mathbf{B}\xi \cdot \nabla_\xi \psi_\alpha(\xi)
=\mathbb{B}_{\alpha}(\psi_\alpha,\psi_\alpha)(\xi).
\end{equation*}
The coefficients $\mathbf{A}$ and $\mathbf{B}$ can explicitly be expressed in terms of the profile $\psi_{\alpha}$. Indeed, integrating first the above stationary problem  with respect to $\xi$ and then multiplying  it by $|\xi|^2$ and integrating again with respect to $\xi$ one sees that \eqref{bas} implies that
$$\mathbf{A}=-\frac{\alpha}{2} \int_{\R^d}
\left(\frac{d+2}{\int_{\R^d} \psi_\alpha(\xi_*)\, \d\xi_*}
-\frac{d\,|\xi|^2}{\int_{\R^d} \psi_\alpha(\xi_*)\, |\xi_*|^2\, \d\xi_*}\right)
\Q^-(\psi_\alpha ,\psi_\alpha )(\xi)\d \xi$$
and
$$\mathbf{B} =-\frac{\alpha}{2}\int_{\R^d}
\left(\frac{1}{\int_{\R^d} \psi_\alpha(\xi_*)\, \d\xi_*}
-\frac{|\xi|^2}{\int_{\R^d} \psi_\alpha(\xi_*)\, |\xi_*|^2\, \d\xi_*}\right)
\Q^-(\psi_\alpha ,\psi_\alpha )(\xi)\d \xi.$$

It was the main purpose of our previous contribution \cite{jde}, to prove the existence of an explicit range of parameters for which such a profile exists. Namely, we have

\begin{theo}\label{existence}{\cite{jde}}There exists some explicit threshold value $\underline{\alpha}_1\in (0,1)$ such that, for any $\alpha \in (0,\underline{\alpha}_1)$ the steady problem
\begin{equation}
\label{tauT}
\mathbf{A}_{\psi_\al} \psi_\al(\xi)+\mathbf{B}_{\psi_\al} \xi \cdot \nabla_\xi \psi_\al(\xi)
=\mathbb{B}_\al(\psi_\al,\psi_\al)(\xi)
\end{equation}
admits a nontrivial nonnegative solution $\psi_\al \in L^1_2(\R^d) \cap L^2(\R^d)$ with
\begin{equation}\label{init}\int_{\R^d} \psi_\al(\xi)\left(\begin{array}{c}
                                   1 \\
                                   \xi \\
                                   |\xi|^2
                                 \end{array}\right)\d\xi=\left(\begin{array}{c}
   1 \\
    0 \\
    \frac{d}{2}
\end{array}\right)\end{equation}
and where
\begin{equation}\begin{cases}\label{ABalpha}
\mathbf{A}_{\psi_\al} &=\displaystyle -\dfrac{\alpha}{2} \int_{\R^d}
\left( d+2
-2\,|\xi|^2\right)
\Q^-(\psi_\al ,\psi_\al )(\xi)\d \xi\\
& \\
\mathbf{B}_{\psi_\al} &=\displaystyle -\dfrac{\alpha}{2d}\int_{\R^d}
\left(d
-2|\xi|^2\right)
\Q^-(\psi_\al ,\psi_\al )(\xi)\d \xi.\end{cases}\end{equation}
Moreover, for any $\alpha_\star<\underline{\alpha}_1$ there exists $K > 0$  such that
$$\sup_{\alpha \in (0,\alpha_\star)} \left(\|\psi_\al\|_{L^1_{3}} + \|\psi_\al\|_{L^{2}}\right) \leq K.$$
\end{theo}
\begin{rmq} We give in the Appendix a sketchy proof of Theorem \ref{existence}, referring to \cite{jde} for details. Notice that the stationary solutions constructed in \cite{jde} are radially symmetric and therefore satisfy the above zero momentum assumption. Note that   we consider here $\alpha  \in (0,\alpha_\star)$ with $\alpha_\star <\underline{\alpha}_1$ in order to get uniform estimates with respect to $\alpha$. In the physical dimension $d=3$, one sees that $\underline{\alpha}_1\leq {2/7}$. 
\end{rmq}

\begin{rmq}
Let us note that here as in \cite{jde}, we only consider profiles satisfying \eqref{init}. Indeed, once we have shown some result of existence or uniqueness for such profiles, we readily get the same result for profiles with arbitrary positive mass and energy and {zero momentum} by a simple rescaling. 
\end{rmq}

Our goal in the present paper is to prove the uniqueness of such a self-similar profile (for a smaller range of the parameters $\alpha$). More precisely, our main result can be formulated as
\begin{theo}\label{theo:main}
There exists some explicit $\alpha^{\sharp} \in (0,\underline{\alpha}_1)$  such that, for any $\alpha \in (0,\alpha^{\sharp})$, the solution $\psi_{\alpha}$ to \eqref{tauT} satisfying \eqref{init} is {\it unique}.
\end{theo}

\subsection{Strategy of proof and organization of the paper} In all the sequel, for any $\al \in (0,1)$, we denote by $\mathscr{E}_\al$  the set of {\it all  nonnegative solution} $\psi_\al$ to \eqref{tauT} with \eqref{init}. Theorem \ref{existence} asserts that provided the parameter $\alpha$ belongs to $(0,\underline{\alpha}_1)$, the set $\mathscr{E}_\al$ is non empty while our main result, Theorem \ref{theo:main}, states that $\mathscr{E}_{\al}$ reduces to a singleton as soon as $\al$ is small enough.

Our strategy of proof is inspired by a strategy adopted in
\cite{BCL, MiMo2, MiMo3, AloLo3} for the study of driven granular gases associated to different kinds of forcing
terms. The approach is based upon the knowledge of some specific limit problem corresponding to $\alpha \to 0$. 

To be more precise, since $\mathbb{B}_0=\Q$ is the classical Boltzmann operator and because one expects  $\mathbf{A}_{\psi_0}=\mathbf{B}_{\psi_0}=0$, one formally notices that for $\alpha=0$, the set  $\mathscr{E}_0$ reduces to the set of distributions $\psi_0$ satisfying \eqref{init} and such that
$$\Q(\psi_0,\psi_0)=0.$$
It is well-known that the steady solution $\psi_{0}$ is therefore a unique Maxwellian distribution; in other words, one expects  $\mathscr{E}_0$ to reduce  to a singleton: $\mathscr{E}_0=\{\M\}$ where $\M$ is the normalized Maxwellian distribution
\begin{equation}\label{M}
\M(\xi)=\pi^{-\frac{d}{2}} \exp\left(-|\xi|^2\right), \qquad \xi \in \mathbb{R}^d.
\end{equation}
The particular case $\alpha=0$ will be referred to as the ``{\it Boltzmann limit}'' in the sequel. 

As in \cite{BCL, MiMo2, MiMo3, AloLo3}, our strategy is based upon the knowledge of such a ``Boltzmann limit'' and on
quantitative estimates of the difference between solutions $\psi_{\alpha}$ to our
original problem \eqref{tauT} and the equilibrium state in the Boltzmann limit. More precisely, our approach is based upon the following three main steps:

\begin{enumerate}
\item First, we prove that \textit{any solution} $\gl$ to \eqref{tauT} satisfies
$$\lim_{\al \to 0}\gl=\mathcal{M}$$
in some suitable sense. This step consists in finding  a suitable Banach space $ \mathcal{X}$ such that $\gl \in \mathcal{X}$ for any $\al>0$ and $\lim_{\al \to 0}\|\gl-\mathcal{M}\|_{\mathcal{X}}=0$. Notice that the above limit will be deduced from a compactness argument.
\item Using the linearized Boltzmann operator around 
 the limiting Maxwellian $\mathcal{M}$
\begin{equation*}
  \mathscr{L} (h)=\Q(\mathcal{M},h)+\Q(h,\mathcal{M})
\end{equation*}
we prove that, if $\gl$ and $\fe$ are two functions in $\mathscr{E}_{\alpha}$, then, for any $\varepsilon >0$, there exists some threshold value $\widetilde{\alpha}$ such that
\begin{equation}\label{Ld}\|\mathscr{L}(\gl-\fe)\|_{{\mathcal{X}}} \leq \varepsilon\|\gl-\fe\|_{{\mathcal{Y}}} \qquad \forall \alpha \in (0,\widetilde{\alpha})\end{equation}
for  some suitable subspace ${\mathcal{Y}} \subset {\mathcal{X}}.$ The proof of such a step comes from precise {\it a posteriori} estimates on the difference of solutions to \eqref{tauT} and the first step. Using then the spectral properties of the linearized operator $\mathscr{L}$ in ${\mathcal{X}}$, one can deduce that the above inequality \eqref{Ld} implies the existence of some positive constant $C_{2} > 0$ such that
$$\|\fe-\gl\|_{{\mathcal{Y}}} \leq C_{2}\al \|\fe-\gl\|_{{\mathcal{Y}}}, \qquad \al \in (0,\widetilde{\al})$$
from which we deduce directly that $\fe=\gl$ provided $\al$ is small enough. This proves Theorem \ref{theo:main}. However, since the first step of this strategy is based upon a compactness argument, the approach as described is non quantitative and no indication on the parameter $\alpha^{\sharp}$ is available at this stage.
\item The final step in our strategy is to provide a quantitative version of the first step. This will be achieved, as in \cite{AloLo3}, by providing a suitable nonlinear relation involving the norm $\|\psi_{\al}-\M\|_{\mathcal{Y}}$ for any $\al.$ 
\end{enumerate}

To prove the first step of the above strategy, one has first to identify a Banach space $\mathcal{X}$ on which uniform estimates are available for any solution  $\psi_{\al}$ to \eqref{tauT}. We can already anticipate that $\mathcal{X}$ is a weighted $L^{1}$-space with exponential weight:
$$\mathcal{X}=L^{1}(\R^{d}, \exp(a|v|)\,\d v), \qquad a > 0$$
and the determination of such a space will be deduced from uniform {\it a posteriori} estimates on elements of $\mathscr{E}_{\al}.$ Such estimates are described in Section \ref{sec:apost} and rely on a careful study of the moments of solutions to \eqref{tauT}. Concerning the convergence of any solution $\psi_{\al} \in \mathscr{E}_{\al}$ towards the Maxwellian $\M$, we first prove that
$$\lim_{\al \to 0}\|\psi_{\al}-\M\|_{\mathbb{H}^{m}_{k}}=0$$
where $\mathbb{H}^{m}_{k}$ is a suitable (weighted) Sobolev space (see Notations hereafter). The proof of such a convergence result is, as already mentioned, based upon a compactness argument and requires a careful investigation of the regularity properties of the solution to \eqref{tauT}.   Our approach for the study of regularity of solutions to \eqref{tauT} is similar to that introduced in \cite{AloLo3} for granular gases and differs from the related contributions on the matter \cite{MiMo3,MiMo2} where the regularity of steady solutions is deduced from the properties of the time-dependent problem. In contrast with these results, our methodology is direct and \textit{relies only on the steady equation} \eqref{tauT} and the crucial estimate is a regularity result for $\Q^{+}(f,g)$ (see Theorem \ref{regularite}). By using standard interpolation inequalities, we complete the first step of our program. 
Concerning the second step (2), it uses in a crucial way some control of the difference of two solutions $\gl,\fe \in \mathscr{E}_{\alpha}$ in some Sobolev norms (see Proposition \ref{prop:estimDiff}). Again, such estimates rely on the regularity properties of the collision operator $\Q^{+}$. Notice also that the spectral properties of the linearized operator $\mathscr{L}$ have been investigated recently in \cite{Mo} where it is shown that the linearized operator shares the same spectral features in the weighted $L^{1}$-space $\mathcal{X}$ than in the more classical space $L^{2}(\R^{d},\M^{-1}(v)\d v)$ (in which the self-adjointness of $\mathscr{L}$ allows easier computations of its spectrum).  The proof of this second step is achieved in Section \ref{sec:sec3} and, more precisely, in Sections \ref{sec:BoLi}, which deals with the Boltzmann limit, and \ref{sec:unique} which addresses the non-quantitative proof of the uniqueness result.

Finally, the proof of of the third above step, as already mentioned is simply based upon a nonlinear estimate on $\|\psi_{\al}-\M\|_{\mathcal{Y}}$ of the form
\begin{equation*} 
\|\psi_\al-\M \|_{\mathcal{Y}} \leq c_1 \|\psi_\al -\M\|_{\mathcal{Y}}^2 + c_2 \alpha  \end{equation*}
for some positive constants $c_{i} > 0$, $i=1,2$ and for $\alpha$ small enough. Such a nonlinear estimate is provided in Section \ref{sec:quanti}. Notice that, as in \cite{AloLo3} and in  contrast with the reference \cite{MiMo3} on granular gases, our approach on the quantitative estimates comes {\it a posteriori} (in the sense that quantitative estimates are deduced from the uniqueness result whereas, in \cite{MiMo3}, the uniqueness result is already proved through quantitative estimates). The main difference between these two approaches is that the present one does not rely on any entropy estimates. 

The paper is ended by three appendices. In Appendix \ref{sec:reguQ+} we give a detailed proof of the regularity properties of the gain part operator $\Q^{+}(g,f)$ which, as already said, play a crucial role in our analysis of the regularity of the solution $\psi_{\al}$ to \eqref{tauT}. Then, in Appendix B, we briefly recall some of the main steps of the proof of the existence of the self-similar profile $\psi_{\al}$. This gives us also the opportunity to sharpen slightly the constants appearing in Theorem \ref{existence} with respect to \cite{jde}. In this appendix, we also investigate the regularity properties of the solution to the time-dependent version of \eqref{tauT} introduced in \cite{jde}. These results have their own interest and illustrates the robustness of the method developed in Section \ref{sec:regularity} to investigate the Sobolev regularity of the steady solution $\psi_{\al}$. Finally, Appendix C recalls some useful interpolations inequalities that are repeatedly used in the paper.

\subsection{Notations} Let us introduce the notations we shall use
in the sequel. 
More generally we define the weighted Lebesgue space $L^p_\eta
(\R^d)$ ($p \in [1,+\infty)$, $\eta \in \R$) by the norm
$$\| f \|_{L^p_\eta} = \left[ \int_{\R^d} |f (v)|^p \, \langle v
        \rangle^{p\eta} \d  v \right]^{1/p} \qquad 1 \leq p < \infty$$
        while $\| f \|_{L^\infty_\eta} =\mathrm{ess-sup}_{v \in \R^d} |f(v)|\langle v\rangle^\eta$ for $p=\infty.$

We shall also use weighted Sobolev spaces $\mathbb{H}^s_\eta(\R^d)$ ($s \in \R_+$, $\eta \in \R$). When $s \in \N$, they  are defined by the norm
$$\| f \|_{\mathbb{H}^s_\eta} = \left( \sum_{|\ell|\leq s} \| \partial^\ell f \|^2_{L^2_\eta} \right)^{1/2}$$
where for $\ell\in\N^d$,  $\partial^\ell=\partial_{\xi_1}^{\ell_1} \ldots \partial_{\xi_d}^{\ell_d}$ and $|\ell|=\ell_1+\ldots +\ell_d$. Then, the definition is extended to real positive values of $s$ by interpolation. For negative value of $s$, one defines $\mathbb{H}^{s}_{-\eta}(\R^{d})$ as the dual space of $\mathbb{H}^{-s}_{\eta}(\R^{d}),$ $\eta \in \R$.


\section{A posteriori estimates on $\psi_\al$}\label{sec:apost}
In all the sequel, for any $\alpha \in (0,\underline{\al}_{1})$, $\psi_{\alpha} \in \mathscr{E}_{\al}$ denotes a solution to
\begin{equation*}
\mathbf{A}_{\psi_\al} \psi_\al(\xi)+\mathbf{B}_{\psi_\al} \xi \cdot \nabla_\xi \psi_\al(\xi)
=\mathbb{B}_\al(\psi_\al,\psi_\al)(\xi)
\end{equation*}
which satisfies \eqref{init} where $\mathbf{A}_{\psi_\al}$ and $\mathbf{B}_{\psi_{\al}}$ are given by \eqref{ABalpha}. Let us note that $\mathbf{A}_{\psi_\al}$ and $\mathbf{B}_{\psi_\al}$ have no sign. However,

$$0 < d\mathbf{B}_{\psi_\al}-\mathbf{A}_{\psi_\al}=\alpha
  \displaystyle \int_{\R^d} \Q^-(\psi_\al,\psi_\al)(\xi)\, \d \xi\,=:\alpha\,\mathbf{a}_{\psi_\al},$$
and
$$0 < (d+2)\mathbf{B}_{\psi_\al}-\mathbf{A}_{\psi_\al}= \dfrac{2\alpha}{d} \displaystyle \int_{\R^d} |\xi|^2\Q^-(\psi_\al,\psi_\al)(\xi)\, \d \xi\,=:\alpha\,\mathbf{b}_{\psi_\al}.$$
Notice that
\begin{equation}\label{ABab}
\mathbf{A}_{\psi_\al}=-\dfrac{\alpha}{2}(d+2)\mathbf{a}_{\psi_\al}+\dfrac{\alpha}{2}d\,\mathbf{b}_{\psi_\al} \qquad \text{ and } \qquad \mathbf{B}_{\psi_\al}=-\dfrac{\alpha}{2} \mathbf{a}_{\psi_\al}+\frac{\alpha}{2}\mathbf{b}_{\psi_\al}.\end{equation}

\subsection{Uniform moments estimates}

We establish several a posteriori estimates on $\psi_\al \in \mathscr{E}_\al$, uniform with respect to the parameter $\alpha$.  We first introduce several notations. For any $k \geq 0$, let  us introduce the moment of order $2k$ as
\begin{equation}\label{moments}
 M_k(\psi_\alpha)=\int_{\R^d} \psi_\alpha(\xi) |\xi|^{2k}\, \d\xi, \qquad \psi_\al \in \mathscr{E}_\al.
\end{equation}
One has first the obvious uniform estimates
\begin{lem}\label{estimab}
For any $\alpha \in (0, \underline{\alpha_1})$ and any $\psi_\al \in \mathscr{E}_\al$ one has
$$\sqrt{2 d} \geq \mathbf{a}_{\psi_\al} \geq \frac{d^2}{4}\; (M_{3/2}(\psi_\al))^{-1} , \qquad \sqrt{d/2} \geq M_{1/2}(\psi_\alpha) \geq \frac{d^2}{4}\; (M_{3/2}(\psi_\al))^{-1} ,$$
while
$$\frac{2}{d} M_{3/2}(\psi_\al)+ \left(\frac{d}{2}\right)^{\frac{1}{2}}\geq \mathbf{b}_{\psi_\al} \geq  \left(\frac{d}{2}\right)^{\frac{1}{2}}, \qquad M_{3/2}(\psi_\al) \geq \left(\frac{d}{2}\right)^{\frac{3}{2}}.$$
Finally, $\mathbf{a}_{\psi_\al} \leq \sqrt{d} \leq \sqrt{2}\mathbf{b}_{\psi_\al}$.
\end{lem}

 \begin{proof} From the definition of $\mathbf{a}_{\psi_\al}$ and \eqref{init}, one has first
 $$\mathbf{a}_{\psi_\al} \leq \int_{\R^d \times \R^d} \psi_\al(\xi)\psi_\al(\xi_*)\left(|\xi|+|\xi_*|\right)\d\xi\d\xi_* =2  M_{1/2}(\psi_\alpha)$$
 while, from Jensen's inequality, 
 $$\int_{\R^d}\psi_\al(\xi_*)|\xi-\xi_*|\d \xi_* \geq  |\xi| \qquad \text{and} \qquad \mathbf{a}_{\psi_\al} \geq   M_{1/2}(\psi_\alpha).$$ Moreover, by H\"older's inequality, $M_{1/2}(\psi_\alpha) \leq \sqrt{M_1(\psi_\alpha)}=\sqrt{d/2}$  and
$$\frac{d}{2}=M_1(\psi_\alpha) \leq \sqrt{M_{1/2}(\psi_\alpha)  M_{3/2}(\psi_\al)}.$$
In the same way, one obtains
 $$ \frac{2}{d} M_{3/2}(\psi_\al) \leq \mathbf{b}_{\psi_\al} \leq  \frac{2}{d} M_{3/2}(\psi_\al)+ M_{1/2}(\psi_\al)\,$$
 and we conclude by noticing that, by Jensen's inequality, $M_{3/2}(\psi_\al) \geq \left(\frac{d}{2}\right)^{\frac{3}{2}}$.

Finally, by Cauchy-Schwarz' inequality, one has
\begin{equation*}
\mathbf{a}_{\psi_\al}  \leq \left(\int_{\R^{2d}} \psi_\al(\xi)\psi_\al(\xi_*)\left|\xi-\xi_*\right|^2\d\xi\d\xi_*\right)^{1/2}= \left(2\int_{\R^d}\psi_\al(\xi)\,|\xi|^2\,\d\xi\right)^{1/2}=\sqrt{d}.
\end{equation*} Since we already saw that $\mathbf{b}_{\psi_\al} \geq \sqrt{d/2}$,  this gives the last estimate.
 \end{proof}

Then, one can reformulate \cite[Proposition 3.4]{jde} to get
\begin{prop}\label{propmom}
For any $\alpha_\star<\underline{\alpha}_1$, one has ,
$$\sup_{\alpha \in (0,\alpha_\star)} \sup_{\psi_\al \in \mathscr{E}_\al} M_k(\psi_\alpha) =\mathbf{M}_k < \infty \qquad \forall k \geq 0$$
where $\mathbf{M}_{k}$ depends only on $\alpha_{\star}$  and $k \geq 0$.
\end{prop}
\begin{proof} The proof follows from the computations made in \cite[Proposition 3.4]{jde}; we sketch only the main steps for the sake of completeness. Let { $\alpha\leq \alpha_\star$} and $\psi_\al \in \mathscr{E}_\al$. One has, for any $k \geq 0$,
\beq\label{eqmom}
\al(k-1)\mathbf{a}_{\psi_\al}\,M_k(\psi_\al)= \al\,k\,\mathbf{b}_{\psi_\al}\,M_k(\psi_\al) + \int_{\R^d} \mathbb{B}_\al\left(\psi_\al,\psi_\al\right)(\xi)\,|\xi|^{2k}\d\xi
\eeq
and, in particular,
$$\al\,\mathbf{a}_{\psi_\al}\,M_{3/2}(\psi_\al)=3\al\,\mathbf{b}_{\psi_\al}\,M_{3/2}(\psi_\al) + 2\int_{\R^d}\mathbb{B}_\al(\psi_\al,\psi_\al)(\xi)\,|\xi|^3\d\xi.$$
The last integral can be estimated thanks to Povzner's estimates \cite[Lemma 3.2]{jde} \footnote{Notice that here we used an improved version of \cite[Lemma 3.2]{jde} where, thanks to Jensen's inequality, we bound $\Q^-(\psi_\al,\psi_\al)$ from below in a more accurate way.} and
$$
\int_{\R^d}\mathbb{B}_\al(\psi_\al,\psi_\al)(\xi)\,|\xi|^3\d\xi \leq -(1-\beta_{3/2}(\al))M_2(\psi_\al)
+\frac{3}{2}\beta_{3/2}(\al)\left(M_{3/2}(\psi_\al)\,M_{1/2}(\psi_\al)+ \left(\frac{d}{2}\right)^2\right).
$$
Thus, since $\mathbf{a}_{\psi_\al} \geq 0$, we get
$$2 (1-\beta_{3/2}(\al))M_2(\psi_\al)\leq
3\beta_{3/2}(\al)\left(M_{3/2}(\psi_\al)\,M_{1/2}(\psi_\al)+\left(\frac{d}{2}\right)^2\right) + 3\al\,\mathbf{b}_{\psi_\al}\,M_{3/2}(\psi_\al)$$
Since $M_2(\psi_\alpha) \geq \frac{2}{d}M_{3/2}(\psi_\al)^2$, using Lemma \ref{estimab} we obtain
\begin{multline*}
\left(\frac{2}{d}(1-\beta_{3/2}(\al))-\frac{3\al}{d}\right)M_{3/2}^2(\psi_\al)\leq
\frac{3}{2}\beta_{3/2}(\al)\sqrt{d/2}\,M_{3/2}(\psi_\al) + \frac{3}{2}\beta_{3/2}(\al)\left(\frac{d}{2}\right)^{2}\\ + \frac{3\al}{d}\, \,\left(\frac{d}{2}\right)^{\frac{3}{2}}\,M_{3/2}(\psi_\al)
\end{multline*}
so that
\begin{equation}\label{M32}
\left(1- \beta_{3/2}(\al)-\frac{3\al}{2}\,\right)M_{3/2}^2(\psi_\al)  \leq C_1 \,M_{3/2}(\psi_\al)+C_2
\end{equation}
where $C_1, C_2$ are positive constants  independent of $\al$ (notice that we used the fact that $\beta_{3/2}(\al) \leq \beta_{3/2}(0)$). Since $\beta_{3/2}(\al)=(1-\alpha)\varrho_{3/2}$ (see \eqref{varrhoK} for the definition of $\varrho_{3/2}$),
setting
$$\alpha_0=\frac{1-\varrho_{3/2}}{\frac{3}{2}-\varrho_{3/2}}$$
we have for $\alpha_\star < \min \{\alpha_0,\underline{\alpha}_1\} = \underline{\alpha}_1$ (by \cite{jde})
$$\sup_{\alpha \in (0,\alpha_\star)}\sup_{\psi_\al \in \mathscr{E}_\al} M_{3/2}(\psi_\al) < \infty$$
thanks to \eqref{M32}. Then, reproducing the computations of \cite[Proposition 3.4]{jde}, one sees that, once $M_{3/2}(\psi_\al)$ is uniformly bounded, the same is true for all $M_k(\psi_\al)$ with $k \geq 3/2$.
\end{proof}

\begin{rmq}\label{ABbound}
Notice that, thanks to the above estimate, one sees that there exists some universal constant $C > 0$ such that
$$\left|\mathbf{A}_{\psi_\al}\right|+\left|\mathbf{B}_{\psi_\al}\right|   \leq C \alpha \qquad \forall \psi_\al \in \mathscr{E}_\al, \quad \forall \alpha \in (0,\alpha_\star).$$
\end{rmq}

\begin{rmq}
As in \cite[Lemma 2.3]{MiMo3}, one may deduce from Proposition \ref{propmom} that
 there exists $C_0 > 0$ such that, for any $\alpha\in(0,\alpha_\star)$,
\begin{equation}\label{concentration}
\int_{\R^d}\psi_\al (\xi_*)|\xi-\xi_*|\d\xi_* \geq C_0\langle \xi\rangle \qquad\forall \xi \in \R^d, \qquad \psi_\al \in \mathscr{E}_\al.\end{equation}
This inequality shall be useful in next section.
\end{rmq}

\subsection{High-energy tails for the steady solution}\label{sect:high}

We are interested here in estimating the high-energy tails of the solution $\psi_\alpha$. In all the sequel, $\alpha_{\star} < \underline{\al}_{1}$ is fixed. Namely, we have the following
\begin{prop}\label{prop:tails}  There exist some constants $A >0$ and $M >0$ such that, for any  $\alpha \in (0,\alpha_\star]$ and any $\psi_\alpha \in \mathscr{E}_\al$, one has
$$\int_{\R^d} \psi_\alpha(\xi) \exp\left(A |\xi|\right) \d\xi \leq M.$$
\end{prop}

\begin{proof}
We adapt here the strategy of \cite{BoGaPa,BCL}. Formally, we have
$$\int_{\R^d} \psi_\alpha(\xi) \exp\left(r |\xi|\right) \d\xi
=\sum_{k=0}^{\infty} \frac{r^k}{k!}\; M_{k/2}(\psi_\alpha), $$
where $M_{k}(\psi_\alpha)$ is defined by \eqref{moments}.
In order to prove Proposition \ref{prop:tails}, it is sufficient to show that the series in the right hand side has a positive and finite radius of convergence. By the Cauchy-Hadamard theorem, it suffices to prove the existence of some uniform constant $C>0$  such that for any  $\alpha  \in (0,\alpha_\star]$ and any $\psi_\alpha \in \mathscr{E}_\al$,
$$M_{k/2}(\psi_\alpha)\leq C^k \, k!, \qquad \forall k\in\N.$$
 Let us now introduce the renormalized moments
$$z_p(\psi_\alpha)=\frac{M_p(\psi_\alpha)}{\Gamma(2p+\gamma)}, \qquad p\geq 0$$
where $\Gamma$ denotes the Gamma function and $\gamma>0$ is a constant (to be fixed later on). Thereby, we are thus led to show the existence of some constants $\gamma>0$ and $K>0$ such that for any  $\alpha  \in (0,\alpha_\star]$ and any $\psi_\alpha \in \mathscr{E}_\al$,
\beq\label{claim_exp_tail}
z_{k/2}(\psi_\alpha)\leq K^k, \qquad \forall k\in\N_*.
\eeq
Now, we deduce from \eqref{eqmom}, Lemma \ref{estimab}, Proposition
\ref{propmom} and from Povzner's estimates \cite[Lemma 3.2]{jde} that
$$(1-\beta_p(\alpha))M_{p+1/2}(\psi_\alpha) \leq S_p
+ \left(\frac{2\alpha \, p}{d} \; \mathbf{M}_{3/2}
+ \alpha p \sqrt{\frac{d}{2}}\right) M_p(\psi_\alpha)$$
with
\begin{multline*}
 S_p\leq \beta_p(\alpha) \sum_{j=1}^{[\frac{p+1}{2}]}
\left(\bary{c}p\\j\\\eary \right)\left(M_{j+1/2}(\psi_\alpha)
M_{p-j}(\psi_\alpha)+M_j(\psi_\alpha) M_{p-j+1/2}(\psi_\alpha)\right)\\
+(1-\beta_p(\alpha))\sqrt{\frac{d}{2}} M_p(\psi_\alpha).
\end{multline*}
Consequently, by \cite[Lemma 4]{BoGaPa}
\bean
(1-\beta_p(\alpha)) \frac{\Gamma(2p+\gamma+1)}{\Gamma(2p+\gamma)}\;
z_{p+1/2}(\psi_\alpha) & \leq & C \beta_p(\alpha)
\frac{\Gamma(2p+2\gamma+1)}{\Gamma(2p+\gamma)} \; Z_p(\psi_\alpha) \\
& + & \left(\frac{2\alpha \, p}{d} \; \mathbf{M}_{3/2}
+ (1-\beta_p(\alpha)+\alpha p) \sqrt{\frac{d}{2}}\right)
z_p(\psi_\alpha),
\eean
where $C=C(\gamma)$ does not depend on $p$ and
$$Z_p(\psi_\alpha) = \max_{1\leq j\leq [\frac{p+1}{2}]}
\left\{ z_{j+1/2}(\psi_\alpha) z_{p-j}(\psi_\alpha),
z_j(\psi_\alpha) z_{p-j+1/2}(\psi_\alpha) \right\}.$$
We have $\beta_p(\alpha)=(1-\alpha)\varrho_p$ where $\varrho_p$ is defined by \eqref{varrhoK}. It is easily checked that
$\varrho_p=\mathbf{O}\left(\frac{1}{p}\right)$ for $d\geq 3$. Thus,
$\beta_p(\alpha)=\mathbf{O}\left(\frac{1}{p}\right)$  uniformly with respect to $\alpha$. Next, $1-\beta_p(\alpha)\geq 1-\varrho_p>0$ for $p>1$ since the mapping $p>1\mapsto \varrho_p$ is strictly decreasing and $\varrho_1=1$. Moreover, for any $\gamma>0$,
$$
\lim_{p\to +\infty} \frac{\Gamma(2p+2\gamma+1)}{\Gamma(2p+\gamma)}\; (2p)^{-\gamma-1} =1, \qquad \lim_{p\to +\infty} \frac{\Gamma(2p+\gamma+1)}{\Gamma(2p+\gamma)}\; (2p)^{-1} =1.$$
Let $\gamma\in(0,1)$. There exist some constants $c_1, c_2, c_3>0$ such that for $p\geq 3/2$,
\begin{equation}\label{cizp} c_1 \,p \,z_{p+1/2}(\psi_\alpha)\leq c_2 \,p^\gamma \,Z_p(\psi_\alpha)
+c_3 \,p \,z_p(\psi_\alpha).\end{equation}
Let $k_0\in\N $ and $K\in\R$ satisfying $k_0\geq 3$,
$$c_2\left(\frac{k_0}{2}\right)^{\gamma-1}\leq \frac{c_1}{2}, \qquad \mbox{ and } \qquad K\geq \max \left\{\max_{1\leq k\leq k_0}\sup_{\al\in (0,\al_\star)}\sup_{\psi_\al \in \mathscr{E}_\al} z_{k/2}(\psi_\alpha), 1, \frac{2c_3}{c_1}\right\}. $$
It follows from Proposition \ref{propmom} that such a $K$ exists.
We now proceed by induction to prove that \eqref{claim_exp_tail} holds.
For $k\leq k_0$, it readily follows from the definition of $K$. Let
$k_*\geq k_0$. Assume that \eqref{claim_exp_tail} holds for
any $k\leq k_*$. Then, taking $p=\frac{{k}_*}{2}$ in the above inequality \eqref{cizp} and noting that $Z_{\frac{{k}_*}{2}}(\psi_\alpha)$ only involves renormalized moments $z_{j/2}(\psi_\alpha)$ for $j\leq {k}_*$, we may use the induction hypothesis and we get
$$c_1\,  z_{\frac{{k}_*+1}{2}}(\psi_\alpha)\leq  c_2 \left(\frac{{k}_*}{2} \right)^{\gamma-1} K^{{k}_*+1}+c_3\, K^{k_*} \leq c_1\, K^{k_*+1},$$
whence \eqref{claim_exp_tail}.
\end{proof}

\subsection{Regularity of the steady state}\label{sec:regularity}
  We investigate here the regularity of any nonnegative  solution $\psi_\al$ to \eqref{tauT}. We begin with showing uniform $L^2_k$-estimates of $\psi_\al$:
\begin{prop}\label{propo:L2} For any $k \geq 0,$  one has
\beq\label{L2norm}
\sup_{\al\in (0,\al_\star)}\sup_{\psi_\al \in \mathscr{E}_\al}\|\psi_\al\|_{L^2_{k}}< \infty.
\eeq
 \end{prop}
 \begin{proof} For given $k \geq 0$, we multiply \eqref{tauT} by $\psi_\al(\xi)\langle \xi \rangle^{2k}$ and integrate over $\R^d$. Then, one obtains
 \begin{multline*}
 \mathbf{A}_{\psi_\al} \|\psi_\al\|^2_{L^2_k} + \mathbf{B}_{\psi_\al} \int_{\R^d} \left(\xi \cdot \nabla_\xi\psi_\al(\xi)\right)\psi_\al(\xi)\langle \xi\rangle^{2k}\d\xi=\\
 (1-\al)\int_{\R^d}\Q^+(\psi_\al,\psi_\al)(\xi)\psi_\al(\xi)\langle \xi\rangle^{2k}\d\xi -\int_{\R^d} \Q^-(\psi_\al,\psi_\al)(\xi)\psi_\al(\xi)\langle \xi\rangle^{2k}\d \xi.
 \end{multline*}
 One notices that
\begin{equation*}\begin{split}
\int_{\R^d} \left(\xi \cdot\nabla_\xi \psi_\al(\xi)\right)\psi_\al(\xi)\langle \xi\rangle^{2k}\d\xi&=-\frac{d}{2}\|\psi_\al\|_{L^2_k}^2-k\int_{\R^d}\psi_\al^2(\xi)\,|\xi|^2\,\langle \xi\rangle^{2(k-1)}\d\xi\\
&=-\frac{d+2k}{2}\|\psi_\al\|_{L^2_k}^2 + k \|\psi_\al\|_{L^2_{k-1}}^2
\end{split}\end{equation*}
and we obtain
 \begin{multline}\label{outil1}
 \left(2\mathbf{A}_{\psi_\al} -(d+2k)\mathbf{B}_{\psi_\al}\right)\|\psi_\al\|^2_{L^2_k} + 2k\mathbf{B}_{\psi_\al} \|\psi_\al\|_{L^2_{k-1}}^2=\\
 2(1-\al)\int_{\R^d}\Q^+(\psi_\al,\psi_\al)(\xi)\psi_\al(\xi)\langle \xi\rangle^{2k}\d\xi -2\int_{\R^d} \Q^-(\psi_\al,\psi_\al)(\xi)\psi_\al(\xi)\langle \xi\rangle^{2k}\d \xi.
 \end{multline}
Now, according to \cite[Corollary 2.2]{AloGa}, for any $\varepsilon \in (0,1)$, there exists $C_\varepsilon > 0$ such that  $$\int_{\R^d}\Q^+(\psi_\al,\psi_\al)(\xi)\psi_\al(\xi)\langle \xi\rangle^{2k}\d\xi \leq C_\varepsilon \|\psi_\al\|_{L^1_{\frac{d(d-3)}{d-1}+k}}^{2-1/d}\,\|\psi_\al\|_{L^2_k}^{1+1/d}  + \varepsilon \|\psi_\al\|_{L^1_{k}}\,\|\psi_\al\|_{L^2_{k}}^2$$
and, using Prop. \ref{propmom}, one sees that
 \beq\label{outil2}
\int_{\R^d}\Q^+(\psi_\al,\psi_\al)(\xi)\psi_\al(\xi)\langle \xi\rangle^{2k}\d\xi \leq C_\varepsilon \, \mathbf{M}_{\frac{d(d-3)}{d-1}+k}^{2-1/d} \,\|\psi_\al\|_{L^2_k}^{1+1/d}  + \varepsilon\, \mathbf{M}_{k}\,\|\psi_\al\|_{L^2_{k}}^2
\eeq
 where $\mathbf{M}_k$ and $\mathbf{M}_{\frac{d(d-3)}{d-1}+k}$ do not depend on $\alpha$.
We first consider the case $k=0$ and then $k>0$.\\

$\bullet$ \textit{First step: $k=0$.} We look now for conditions on $\alpha$
ensuring that
$$-\left(2\mathbf{A}_{\psi_\al}-d\mathbf{B}_{\psi_\al}\right)\|\psi_\al\|_{L^2}^2 -2\int_{\R^d} \Q^-(\psi_\al,\psi_\al)(\xi)\psi_\al(\xi)\d \xi$$
can absorb the leading order term $2\varepsilon \mathbf{M}_{k}\,\|\psi_\al\|_{L^2_{k}}^2=2\varepsilon \,\|\psi_\al\|_{L^2}^2$. One has
\begin{equation}\label{2A-dB}\left(2\mathbf{A}_{\psi_\al}-d\mathbf{B}_{\psi_\al}\right)=-\dfrac{\alpha}{2}\left(d+4\right)\mathbf{a}_{\psi_\al} + \frac{\alpha \, d}{2}\;\mathbf{b}_{\psi_\al}.\end{equation}
Now, by Lemma \ref{estimab}, one has $\mathbf{b}_{\psi_\al} \geq \dfrac{1}{\sqrt{2}}\;\mathbf{a}_{\psi_\al}$ and it is enough to estimate
$$\mathbb{K}_2:=\dfrac{\alpha}{2}\left(d+4-\frac{d}{\sqrt{2}}\right)\mathbf{a}_{\psi_\al}\,\|\psi_\al\|_{L^2}^2  -2\int_{\R^d} \Q^-(\psi_\al,\psi_\al)(\xi)\psi_\al(\xi) \d \xi .$$
As in \cite[Section 4]{jde}, we compound $\|\psi_\al\|_{L^2}^2$ and $\mathbf{a}_{\psi_\al}$ into a unique integral to get
$$ \mathbf{a}_{\psi_\al}\,\|\psi_\al\|_{L^2}^2 \leq 2\int_{\R^d} \Q^-(\psi_\al,\psi_\al)(\xi)\psi_\al(\xi) \d \xi $$
and therefore, with   \eqref{concentration},
$$\mathbb{K}_2 \leq -\eta_2(\al)\int_{\R^d} \Q^-(\psi_\al,\psi_\al)(\xi)\psi_\al(\xi) \d \xi \leq -\eta_2(\al)\, C_0 \|\psi_\al\|_{L^2_{1/2}}^2 \leq -\eta_2(\al)\, C_0 \|\psi_\al\|_{L^2}^2$$
where $\eta_2(\alpha)=2-\a\left(d+4-\frac{d}{\sqrt{2}}\right).$ Thus
\beq\label{alpha2}
\eta_2(\al) >0 \Longleftrightarrow  \alpha < \alpha_2:=\dfrac{2\sqrt{2}}{4\sqrt{2}+d(\sqrt{2}-1)}.
\eeq
But, $\underline{\alpha}_1\leq \underline{\alpha}=\min(\alpha_2,\alpha_0)$ by Lemma \ref{evol:L2} and Proposition \ref{propH1q}. Thus,
choosing $\alpha_\star<\underline{\alpha}_1$ and $\varepsilon  \leq \eta_2(\alpha_\star) C_0/4$, we get for any $\al\in(0,\alpha_\star)$,
$$\frac{\eta_2(\alpha_\star) \, C_0}{2}\; \|\psi_\al\|_{L^2}^2 \leq  2\,  C_\varepsilon \,  \mathbf{M}_{\frac{d(d-3)}{d-1}}^{2-1/d}\; \|\psi_\al\|_{L^2}^{1+1/d},$$
where the constants are independent of $\alpha$. This completes the proof of \eqref{L2norm} for $k=0$.\\

$\bullet$ \textit{Second step: $k>0$.}
 Using  \eqref{concentration}, \eqref{outil2}, Remark \ref{ABbound} and  bounding the $L^2_{k-1}$ norm by the $L^2_k$ one, \eqref{outil1} leads to
  $$2C_0 \|\psi_\al\|^2_{L^2_{k+1/2}} \leq C_k \|\psi_\al\|_{L^2_k}^2 + 2\, C_\varepsilon \, \mathbf{M}_{\frac{d(d-3)}{d-1}+k}^{2-1/d} \,\|\psi_\al\|_{L^2_k}^{1+1/d}  + 2 \,\varepsilon\, \mathbf{M}_{k}\,\|\psi_\al\|_{L^2_{k}}^2 $$
for some constant $C_k >0$ independent of $\al\in(0,\alpha_\star)$.
Now, choosing  $\varepsilon$ such that $ 2 \varepsilon \mathbf{M}_{k} \leq C_0$ we get the existence of some positive constants $C_{1,k}  > 0$ and $C_{2,k} > 0$ (independent of $\al\in(0,\alpha_\star)$) such that
 $$C_0 \|\psi_\al\|_{L^2_{k+1/2}}^2\leq C_{1,k}\|\psi_\al\|_{L^2_k}^2 + C_{2,k} \|\psi_\al\|_{L^2_k}^{1+1/d}.$$
 Now, one uses the fact that, for any $R > 0$,
 $$\|\psi_\al\|_{L^2_{k}}^2 \leq (1+R^2)^{k}\|\psi_\al\|_{L^2}^2+R^{-1}\|\psi_\al\|_{L^2_{k+1/2}}^2$$
and, since $\sup_{\al \in (0,\alpha_\star)} \sup_{\psi_\al \in \mathscr{E}_\al}\|\psi_\al\|_{L^2} < \infty$, one can choose $R > 0$ large enough so that $C_{1,k} R^{-1}=C_0/2$ to obtain
$$\tfrac{C_0}{2}\|\psi_\al\|_{L^2_{k+1/2}}^2 \leq C_{3,k} + C_{2,k} \|\psi_\al\|_{L^2_k}^{1+1/d} \qquad \forall \al \in (0,\al_\star)\,,\,\forall k \geq 0.$$
The conclusion follows easily since $1+1/d < 2$. \end{proof}
We extend now these estimates to general $\mathbb{H}^m_k$ estimates. The key argument is the regularity of $\Q^+$ obtained recently in \cite{AloLo3} in dimension $d=3$.
\begin{theo}\label{regularite} For any $\varepsilon >0$ and any $s \geq \frac{3-d}{2},$ $\eta \geq 0$, there exists $C=C(\varepsilon,s,\eta)$ such that
\begin{multline}\label{regestim}
\|\Q^+(f,g)\|_{\mathbb{H}^{s+\frac{d-1}{2}}_\eta} \leq C \,\|g\|_{\mathbb{H}^s_{\eta+\kappa}}\|f\|_{L^1_{2\eta+\kappa}} + \varepsilon\| f\|_{\mathbb{H}^{s+\frac{d-3}{2}}_{\eta+\frac{d+3}{2}}}\,\|g\|_{\mathbb{H}^{s+\frac{d-3}{2}}_{\eta+1}}\\
+\varepsilon \sum_{|\l|=s+\frac{d-1}{2}}\left( \| g\|_{L^1_{\eta+1}}\,\|\partial^\ell f\|_{L^2_{\eta+1}}+\| f\|_{L^1_{\eta+1}}\,\| \partial^\ell g\|_{L^{2}_{\eta+1}}\right)
\end{multline}
where $\kappa>3/2$.
\end{theo}
The proof of this result, for general dimension $d$, is given in Appendix \ref{sec:reguQ+}. One then has
\begin{theo}\label{theo:HK}
 Setting
$${\hat{\al}_0=\underline{\alpha}_1 \qquad \mbox{and}\qquad \hat{\alpha}_m=\min\left(\underline{\alpha}_1\,,\:{\frac{2 \,C_0}{\sqrt{d}}}\left(\frac{d}{2}+m+2\right)^{-1}\right)\quad \mbox{for } m\geq1},$$
where $C_0>0$ is the constant from \eqref{concentration}, one has, for any  integer $m \geq 0$ and any $0<\overline{\al}_m<\hat{\al}_m$
$$\sup_{\al \in (0,\overline{\al}_m]} \sup_{\psi_\al \in \mathscr{E}_\al}\|\psi_\al\|_{\mathbb{H}^m_k} < \infty \qquad \forall k \geq 0.$$
\end{theo}
\begin{proof} In several steps of the proof, we shall resort to the following way of estimating weighted $L^1$-norms by $L^2$-norms with higher order weights:
 \begin{equation}
\label{taug21}\|h\|_{L^1_k} \leq M_\mu\|h\|_{L^2_{k+d/2+\mu}}\qquad \forall k \geq 0, \qquad \forall \mu >0
\end{equation}
where the universal constant $M_\mu$ is given  by $M_\mu=\|\langle \cdot \rangle^{-\tfrac{d}{2}-\mu}\|_{L^2} < \infty$.

As in \cite[Theorem 3.6]{AloLo3}, the proof uses induction over $m$. Namely, Proposition \ref{propo:L2} shows that the result is true if $m=0$ since $\alpha_{*} < \underline{\alpha}_{1}$ is arbitrary. Let then $m \geq 1$ and $0<\overline{\al}_m<\hat{\al}_m$ be fixed. Assume that  for any $0\leq n\leq m-1$, for any $k\geq 0$ and for any $\delta>0$, there exists  $C_{n,k} >0$ such that 
\begin{equation}\label{Hmk}
 \sup_{\al \in (0,\hat{\al}_n-\delta]}\sup_{\psi_\al \in \mathscr{E}_\al}\| \psi_\al\|_{\mathbb{H}^n_k}\leq C_{n,k}.
\end{equation}
Note that $\overline{\al}_m<\hat{\al}_n$ for any  $0\leq n \leq m$. {We then deduce from Lemma \ref{interpolSob} and \eqref{Hmk} that  for any real number $\frac{3-d}{2}\leq s\leq m-1$, for any $k\geq 0$, there exists  $C_{s,k} >0$ such that 
\begin{equation}\label{Hsk}
 \sup_{\al \in (0,\overline{\al}_m]}\sup_{\psi_\al \in \mathscr{E}_\al}\| \psi_\al\|_{\mathbb{H}^s_k}\leq C_{s,k}.
\end{equation}}
Let now $\ell$ be a given multi-index with $|\ell|=m$ and let $k \geq 0$. For simplicity, set
$$F_\ell(\xi)=\partial^\ell \psi_\al(\xi).$$
Since $\psi_\al$ is a solution to \eqref{tauT}, $F_\ell(\cdot)$ satisfies
$$\left(\mathbf{A}_{\psi_\al} +|\ell|\mathbf{B}_{\psi_\al}\right)\,F_\ell(\xi) + \mathbf{B}_{\psi_\al} \xi \cdot \nabla_\xi F_\ell(\xi) =(1-\al)\partial^\ell \Q^+(\psi_\al,\psi_\al)(\xi) -\partial^\ell \Q^-(\psi_\al,\psi_\al)(\xi)$$
where we simply noticed that $\partial^\ell \left(\xi\cdot\nabla_\xi \psi_\al(\xi)\right)=\xi \cdot \nabla_\xi \partial^\ell \psi_\al(\xi)+ |\ell|\partial^{\ell}\psi_\al(\xi).$ Multiplying this last identity by $F_\ell(\xi)\,\langle\xi\rangle^{2k}$ and integrating over $\R^d$ yields, as above,
\begin{multline}\label{estimHell}
\left(\mathbf{A}_{\psi_\al}+\left(|\ell|-\frac{d}{2}-k\right)\mathbf{B}_{\psi_\al}\right)\|F_\ell\|_{L^2_k}^2  + k\mathbf{B}_{\psi_\al}\,\|F_\ell\|_{L^2_{k-1}}^2=\\
(1-\al)\int_{\R^d}\partial^\ell\Q^+(\psi_\al,\psi_\al)(\xi)\,F_\ell(\xi)\langle \xi\rangle^{2k}\d\xi -\int_{\R^d}\partial^\ell\Q^-(\psi_\al,\psi_\al)(\xi)\,F_\ell(\xi)\langle \xi\rangle^{2k}\d\xi.
\end{multline}
Let us now estimate the integral involving $\Q^-$. Noticing that
$$\partial^\ell\Q^-(\psi_\al,\psi_\al)=\sum_{0\leq\nu \leq \ell} \left(\begin{array}{c}
\ell \\ \nu
\end{array}\right) \Q^-(\partial^\nu \psi_\al,\partial^{\ell-\nu}\psi_\al).$$
For any $\nu$ with $\nu\neq \ell$, there exists $i_0 \in \{1,\ldots,d\}$ such that $\ell_{i_0}-\nu_{i_0} \geq 1$ and integration by parts yields
\begin{equation*}\begin{split}
\left|\Q^{-}(\partial^\nu \psi_\al,\partial^{\ell-\nu}\psi_\al)(\xi)\right|&=\left|\partial^\nu \psi_\al(\xi)\right|\,\left|\int_{\R^d} \partial^{\ell-\nu} \psi_\al(\xi_*)|\xi-\xi_*|\d\xi_*\right|\\
&\leq \left|\partial^{\,\nu} \psi_\al(\xi)\right|\,\|\partial^{\,\sigma }\psi_\al\|_{L^1}
\end{split}\end{equation*}
where $\sigma=(\sigma_1,\ldots,\sigma_d)$ is defined with $\sigma_{i_0}=\ell_{i_0}-\nu_{i_0}-1$ and $\sigma_i=\ell_i-\nu_i$ if $i \neq i_0.$  Thus, estimating the $L^1$ norm by some weighted $L^2$ norm thanks to \eqref{taug21} (with $\mu=d/2$ for simplicity) we get
$$\left|\Q^{-}(\partial^\nu \psi_\al,\partial^{\ell-\nu}\psi_\al)(\xi)\right|\leq C\,|\partial^\nu \psi_\al(\xi)|\,\|\partial^{\,\sigma }\psi_\al\|_{L^2_d}$$
for some universal constant $C >0$ independent of $\al$.  From the induction hypothesis \eqref{Hmk}, this last quantity is uniformly bounded and using Cauchy-Schwarz inequality we obtain
\begin{multline}\label{Q-nu}
\left|\underset{\nu \neq \ell}{\sum_{0\leq \nu\leq \ell}} \left(\begin{array}{c} \ell \\ \nu \end{array}\right)
\int_{\R^d}  \Q^{-}(\partial^\nu \psi_\al,\partial^{\ell-\nu}\psi_\al)(\xi)\,F_\ell(\xi)\langle \xi\rangle ^{2k}\d \xi \right|
\\
\leq C_2 \underset{\nu \neq \ell}{\sum_{0\leq \nu\leq \ell}}
\left(\begin{array}{c}
\ell \\ \nu
\end{array}\right)\|\partial^\nu \psi_\al\|_{L^2_k}\,\|F_\ell\|_{L^2_k}
\leq C_{k,m}\|F_\ell\|_{L^2_k}
\end{multline}
for some positive constant $C_{k,m}$ independent of $\al$.  Second, whenever $\nu=\ell$ one has
$$\int_{\mathbb{R}^{d}}\Q^-(\partial^{\ell} \psi_\al,\psi_\al)(\xi)\,\partial^{\ell}\psi_\al(\xi)\langle \xi\rangle ^{2k}\d \xi=\int_{\R^d} F_\ell^2(\xi)\langle \xi\rangle^{2k}\d \xi\int_{\R^d} \psi_\al(\xi)\,|\xi-\xi_*|\d \xi_*.$$
Therefore, according to \eqref{concentration}, we  get the lower bound
\begin{equation}\label{Q-nu+}
\int_{\mathbb{R}^{d}}\Q^-(\partial^{\ell} \psi_\al,\psi_\al)(\xi)\,\partial^{\ell}\psi_\al(\xi)\langle \xi\rangle ^{2k}\d \xi \geq C_0 \|F_\ell\|^{2}_{L^2_{k+\frac{1}{2}}}.
\end{equation}
 Estimates \eqref{Q-nu} and \eqref{Q-nu+} together with \eqref{estimHell} yield then
 \begin{multline}\label{estimHell1}
 C_0 \|F_\ell\|_{L^2_{k+1/2}}^2 + \left(\mathbf{A}_{\psi_\al}+\left(|\ell|-\frac{d}{2}-k\right)\mathbf{B}_{\psi_\al}\right)\|F_\ell\|_{L^2_k}^2  \\
 + k\mathbf{B}_{\psi_\al}\,\|F_\ell\|_{L^2_{k-1}}^2 \leq C_{k,m}\|F_\ell\|_{L^2_k} + \int_{\R^d}\left| \partial^\ell\Q^+(\psi_\al,\psi_\al)(\xi)\right|\,| F_\ell(\xi)| \, \langle \xi\rangle^{2k}\d\xi
 \end{multline}
 where we simply bounded $(1-\al)$ by $1$. Let us assume now that $k \geq 1/2$. One has
\begin{equation*}
\int_{\R^d} \left| \partial^\ell\Q^+(\psi_\al,\psi_\al)(\xi)\right|\,|F_\ell(\xi)|\, \langle \xi\rangle^{2k}\d\xi \leq \|\partial^\ell \Q^+(\psi_\al,\psi_\al)\|_{L^2_{k-1/2}}\|F_\ell\|_{L^2_{k+1/2}}.
\end{equation*}
One can use Theorem \ref{regularite} with $s= m- \frac{d-1}{2}$, \eqref{Hsk} and the uniform $L^1_k$ bounds to get, for any $\varepsilon  > 0$, the existence of some positive constants $C_1(\varepsilon,k,m) > 0$ and $C_2(k) > 0$ such that, for any $\al \in (0,\overline{\alpha}_m)$,
$$\|\partial^\ell \Q^+(\psi_\al,\psi_\al)\|_{L^2_{k-1/2}} \leq \|\Q^+(\psi_\al,\psi_\al)\|_{\mathbb{H}^m_{k-\frac{1}{2}}} \leq C_1(\varepsilon,k,m) +  \varepsilon\,C_2(k)\, \sum_{|\eta|=m} \|F_\eta\|_{L^{2}_{k+\frac{1}{2}}}.$$
Summing \eqref{estimHell1} over all $\ell\in\N^d$ such that $|\ell|=m$, we deduce that, for some  $\overline{C_2}(k,m)>0$,
\begin{multline*}
 C_0 \sum_{|\ell|=m} \|F_\ell\|_{L^2_{k+1/2}}^2 + \left(\mathbf{A}_{\psi_\al}+\left(m-\frac{d}{2}-k\right)\mathbf{B}_{\psi_\al}\right) \sum_{|\ell|=m} \|F_\ell\|_{L^2_k}^2
 + k\mathbf{B}_{\psi_\al}\,\sum_{|\ell|=m} \|F_\ell\|_{L^2_{k-1}}^2 \\
\leq (C_{k,m}+ C_1(\varepsilon,k,m)) \sum_{|\ell|=m}\|F_\ell\|_{L^2_k+\frac{1}{2}} + \varepsilon\,\overline{C_2}(k,m)\,\sum_{|\ell|=m} \|F_\ell\|^2_{L^{2}_{k+\frac{1}{2}}}.
\end{multline*}
Let us fix $\chi \in (0,1)$ such that $\overline{\al}_m\leq (1-2 \chi)\hat{\al}_m$. Choosing then $\varepsilon > 0$ so that $\varepsilon\,\overline{C_2}(k,m)=\chi C_0 $, we get finally
\begin{multline}\label{C02}
(1-\chi)C_0 \sum_{|\ell|=m}\|F_\ell\|_{L^2_{k+1/2}}^2 + \left(\mathbf{A}_{\psi_\al}+\left(m-\frac{d}{2}-k\right)\mathbf{B}_{\psi_\al}\right)\sum_{|\ell|=m}\|F_\ell\|_{L^2_k}^2  \\
+ k\mathbf{B}_{\psi_\al}\,\sum_{|\ell|=m}\|F_\ell\|_{L^2_{k-1}}^2\leq C_3(k,m) \sum_{|\ell|=m}\|F_\ell\|_{L^2_{k+1/2}}
\end{multline}
where $C_3(k,m) > 0$ is a positive constant independent of $\al$. Recall that
\begin{equation*}
\mathbf{A}_{\psi_\al}+\left(m-\frac{d}{2}-k\right)\mathbf{B}_{\psi_\al}=-\frac{\al}{2}\left(\frac{d}{2}+m-k +2\right)\mathbf{a}_{\psi_\al} + \frac{\al}{2}\left(\frac{d}{2}+m-k\right)\mathbf{b}_{\psi_\al}\end{equation*}
while $k\mathbf{B}_{\psi_\al}=-\frac{\al\,k}{2}\mathbf{a}_{\psi_\al} + \frac{\al\,k}{2}\mathbf{b}_{\psi_\al}.$ \\

At this stage, we first consider the case of {\it small} $k$. Namely, let us assume that
$$k \leq m+\frac{d}{2}.$$
Then, neglecting all the terms involving $\mathbf{b}_{\psi_\al}$, one has
\begin{equation*}\begin{split}
\left(\mathbf{A}_{\psi_\al}+\left(m-\frac{d}{2}-k\right)\mathbf{B}_{\psi_\al}\right)\|F_\ell\|_{L^2_k}^2  &+ k\mathbf{B}_{\psi_\al}\,\|F_\ell\|_{L^2_{k-1}}^2\\
&\geq -\frac{\al}{2}\left(\frac{d}{2} + 2 +m\right)\mathbf{a}_{\psi_\al}\,\|F_\ell\|_{L^2_k}^2 + \frac{\al\,k}{2}\mathbf{a}_{\psi_\al} \left(\|F_\ell\|_{L^2_k}^2-\|F_\ell\|_{L^2_{k-1}}^2\right)\\
&\geq  -\frac{\al}{2}\sqrt{d}\left(\frac{d}{2}+2+m\right)\|F_\ell\|_{L^2_k}^2\end{split}
\end{equation*}
Consequently,
one sees that, for any $\al\in(0,\overline{\al}_m]$, \eqref{C02} yields
$$ \chi\, \hat{\al}_m \, {\frac{\sqrt{d}}{2}}\left(\frac{d}{2}+2+m\right) \sum_{|\ell|=m}\|F_\ell\|_{L^2_{k+1/2}}^2 \leq C_3(k,m) \sum_{|\ell|=m} \|F_\ell\|_{L^2_{k+1/2}}.$$
In particular, for any $k \leq m+\frac{d}{2}$, one has
\begin{equation}\label{L2k0}
\sup_{\alpha \in (0,\overline{\alpha}_m)} \sup_{\psi_\al \in \mathscr{E}_\al}\sum_{|\ell|=m}\|F_\ell\|_{L^2_{k+1/2}}^2 < \infty.\end{equation}
We now turn back to \eqref{C02} for {\it any} $k > \frac{d}{2}+m$. Bounding as in Prop. \ref{propo:L2} the absolute value of $|\mathbf{A}_{\psi_\al}|$ and $|\mathbf{B}_{\psi_\al}|$ uniformly with respect to $\al$, we see that there exist positive constants $\mathcal{C}_1(k,m) > 0$,  independent of $\al$ such that
$$(1-\chi)C_0 \sum_{|\ell|=m}\|F_\ell\|_{L^2_{k+1/2}}^2 \leq \mathcal{C}_1(k,m) \sum_{|\ell|=m}\|F_\ell\|_{L^2_k}^2 +  C_3(k,m) \sum_{|\ell|=m} \|F_\ell\|_{L^2_{k+1/2}} \qquad \forall \al \in (0,\overline{\al}_m).$$
Now, arguing as in the proof of Prop. \ref{propo:L2}, one has, for any $R > 0$,
$$\|F_\ell\|_{L^2_k}^2 \leq R^{-1}\|F_\ell\|_{L^2_{k+1/2}}^2 + (1+R^2)^{k-k_0} \|F_\ell\|_{L^2_{k_0}}^2$$
where $k_0=\frac{d}{2}+m.$ Choosing $R > 0$ big enough and using \eqref{L2k0}
completes the proof.\end{proof}

One deduces directly from Sobolev inequalities the following uniform $L^\infty$ bound.
\begin{cor}\label{cor:Linfty}
For any $m \geq 0$, there exists some explicit $\gamma_m > 0$ such that
$$\sup_{\alpha \in (0,\gamma_m)}\sup_{\psi_\al \in \mathscr{E}_\al} \|\psi_\al\|_{\mathbb{W}^{m,\infty}} < \infty.$$
\end{cor}

\subsection{High-energy tails for difference of steady solutions}

Now  we established the regularity of $\psi_\al$, we can extend Proposition \ref{prop:tails} to the high-energy tails to the first order derivative of $\psi_\al$. Namely
\begin{lem}\label{lem:tails}
Let $\overline{\alpha}_1\in(0,\hat{\al}_1)$ where $\hat{\alpha}_1$ is defined in Theorem \ref{theo:HK}.
There exist some uniform constant $C>0$ and $\gamma \in (0,1)$ such that for 
any $\alpha  \in (0,\overline{\alpha}_1]$ and any 
$\psi_\alpha \in \mathscr{E}_\al$,
$$\int_{\R^d} \left|\nabla \psi_{\al}(\xi)\right|\,|\xi|^{k} \d \xi \leq C^{k}\, \Gamma(k+\gamma), \qquad \forall k\in\N_*.$$
As a consequence, there exist some constants $A_1 >0$ and $M_1 >0$ such that, for any  $\alpha \in (0,\overline{\alpha}_1]$ and any $\psi_\alpha \in \mathscr{E}_\al$, one has
$$\int_{\R^d} \left|\nabla \psi_\alpha(\xi)\right| \exp\left(A_1 |\xi|\right) \d\xi \leq M_1.$$
\end{lem}
\begin{proof} The proof follows the strategy of Prop. \ref{prop:tails} and exploits some of the results of \cite{AloGamJMPA}. First, with the notations of Theorem \ref{theo:HK}, for any $\overline{\alpha}_1 \in (0,\hat{\alpha}_1)$, {for any $k\geq 0$}, there exists $M > 1$ such that, for any $ \alpha  \in (0,\overline{\al}_1]$ and any $\psi_\al \in \mathscr{E}_\al$, $ \|\psi_\al\|_{\mathbb{H}^1_k} \leq M$. In particular, by a simple use of Cauchy-Schwarz inequality, {for any $k\geq 0$}
\begin{equation}\label{moments_grad}
\sup_{\al \in (0,\overline{\al}_1]}\sup_{\psi_\al \in \mathscr{E}_\al} \left\|\nabla\,\psi_\al\right\|_{L^1_k} < \infty.
\end{equation}
For any fixed $\alpha \in (0,\overline{\al}_1)$ and any solution $\psi_\al \in \mathscr{E}_\al$ we denote (omitting for simplicity the dependence with respect to $\al$), for $p\geq 1$,
\begin{eqnarray*}
M_p & := & M_p(\psi_\al)=\int_{\R^d} \psi_\al(\xi)\, |\xi|^{2p}\d \xi\,;\\ 
m_p^{(j)}& :=& \int_{\R^d} \left|\partial_j \psi(\xi)\right|\,|\xi|^{2p}\,\d\xi, \qquad \forall j=1,\ldots,d.\end{eqnarray*}
For any fixed $j \in \{1,\ldots,d\}$, set $\Psi_j=\partial_j \psi_\al.$ Clearly, $\Psi_j$ satisfies
$$\left[\mathbf{A}_{\psi_\al}+\mathbf{B}_{\psi_\al}\right] \Psi_j(\xi) + \mathbf{B}_{\psi_\al}\,\xi \cdot \nabla \Psi_j(\xi)
=\partial_j \mathbb{B}_\alpha(\psi_\al,\psi_\al)(\xi)$$
and, for any $p \geq 3/2$, multiplying this identity by $\mathrm{sign}(\Psi_j(\xi))\,|\xi|^{2p}$ and integrating over $\R^d$, one gets
$$\left[\mathbf{A}_{\psi_\al}+\left(1-d-2p\right)\mathbf{B}_{\psi_\al}\right]m_p^{(j)}=\int_{\R^d}\partial_j \mathbb{B}_\alpha(\psi_\al,\psi_\al)(\xi)\mathrm{sign}(\Psi_j(\xi))\,|\xi|^{2p}\,\d\xi$$
where we used the fact that
$$\int_{\R^d} \xi \cdot \nabla \Psi_j(\xi)\mathrm{sign}(\Psi_j(\xi))\,|\xi|^{2p}\d \xi=-(d+2p)\int_{\R^d}|\Psi_j(\xi)|\,|\xi|^{2p}\,\d\xi.$$
Now, since
$$\partial_j \mathbb{B}_\alpha(\psi_\al,\psi_\al)=(1-\alpha)\partial_j \Q(\psi_\al,\psi_\al) -\alpha \Q^-(\Psi_j,\psi_\al) -\alpha \Q^-(\psi_\al,\Psi_j)$$
and using \eqref{ABab}, one checks easily that
\begin{multline*}
\alpha\,\mathbf{a}_{\psi_\al}\left(p-\frac{3}{2}\right)m_p^{(j)} = \alpha \left(p-\frac{1}{2}\right)\mathbf{b}_{\psi_\al}m_p^{(j)} + (1-\alpha)\int_{\R^d}\partial_j \Q(\psi_\al,\psi_\al)(\xi)\mathrm{sign}(\Psi_j(\xi))\,|\xi|^{2p}\,\d\xi\\
-\al \int_{\R^d} |\Psi_j(\xi)|\,|\xi|^{2p} \d\xi \int_{\R^d}\psi_\al(\xi_*)|\xi-\xi_*|\d\xi_* \\
-\alpha \int_{\R^d} \Q^-(\psi_\al,\Psi_j)(\xi)\mathrm{sign}(\Psi_j(\xi))\,|\xi|^{2p}\,\d\xi.
\end{multline*}
Now, using Jensen's inequality
$$\int_{\R^d} |\Psi_j(\xi)|\,|\xi|^{2p} \d\xi \int_{\R^d}\psi_\al(\xi_*)|\xi-\xi_*|\d\xi_* \geq m_{p+\frac{1}{2}}^{(j)}$$
while it is easy to check that
$$\left|\int_{\R^d} \Q^-(\psi_\al,\Psi_j)(\xi)\mathrm{sign}(\Psi_j(\xi))\,|\xi|^{2p}\,\d\xi\right| \leq M_{p+\frac{1}{2}}m_0^{(j)} + M_p\,m_{\frac{1}{2}}^{(j)}.$$
Therefore,
\begin{multline}\label{esti11}
\alpha\,\mathbf{a}_{\psi_\al}\left(p-\frac{3}{2}\right)m_p^{(j)} + \alpha m_{p+\frac{1}{2}}^{(j)} \leq \alpha \left(p-\frac{1}{2}\right)\mathbf{b}_{\psi_\al}m_p^{(j)} + \alpha \left(M_{p+\frac{1}{2}}m_0^{(j)} + M_p\,m_{\frac{1}{2}}^{(j)}\right) \\
+ (1-\alpha)\int_{\R^d}\partial_j \Q(\psi_\al,\psi_\al)(\xi)\mathrm{sign}(\Psi_j(\xi))\,|\xi|^{2p}\,\d\xi.
\end{multline}
According to \cite[Lemma 6]{AloGamJMPA}, one can estimate this last integral as follows: there is some universal constant $\eta > 0$ such that
\begin{multline}\label{esti12}
\int_{\R^d}\partial_j \Q(\psi_\al,\psi_\al)(\xi)\,\mathrm{sign}(\Psi_j(\xi))\,|\xi|^{2p}\,\d\xi \leq -\eta(1-\varrho_p)m_{p+\frac{1}{2}}^{(j)} + 2m_{\frac{1}{2}}^{(j)} M_p + 2 m_0^{(j)}\,M_{p+\frac{1}{2}} + \varrho_p \mathcal{S}_p^{(j)}
\end{multline}
where
$$\mathcal{S}_p^{(j)}:=\sum_{k=1}^{[\frac{p+1}{2}]}
\left(\bary{c}p\\k\\\eary \right)\left(m_{k+\frac{1}{2}}^{(j)} M_{p-k} + M_{k+\frac{1}{2}} m_{p-k}^{(j)} + m_{k}^{(j)} M_{p-k+\frac{1}{2}} + m_{p-k+\frac{1}{2}}^{(j)} M_k\right)$$
and $\varrho_p$ is defined in \eqref{varrhoK}. Neglecting the term involving $\mathbf{a}_{\psi_\al} \geq 0$ and since $\mathbf{b}_{\psi_\al} \leq \underline{\mathbf{b}}$ for some positive constant $\underline{\mathbf{b}}=\frac{2}{d}\mathbf{M}_{3/2}+\sqrt{\frac{d}{2}}$ independent of $\alpha$ (see Lemma \ref{estimab} and Prop. \ref{propmom}) we obtain from \eqref{esti11}
$$
\left[\alpha+\eta(1-\al)(1-\varrho_p)\right] m_{p+\frac{1}{2}}^{(j)}
 \leq \alpha \left(p-\frac{1}{2}\right)\underline{\mathbf{b}}\, m_p^{(j)} +   (2-\al) \left(M_{p+\frac{1}{2}}m_0^{(j)} + M_p\,m_{\frac{1}{2}}^{(j)}\right)
 +\beta_p(\al) \mathcal{S}^{(j)}_p
$$
where, as in Proposition \ref{prop:tails}, $\beta_p(\al)=(1-\al)\varrho_p=\mathbf{O}\left(\frac{1}{p}\right)$ uniformly with respect to $\al$.  We introduce the renormalized moments
$$z_p^{(j)}=\frac{m_p^{(j)}}{\Gamma(2p+\gamma)}, \qquad z_p:=\frac{M_p}{\Gamma(2p+\gamma)} \qquad p\geq 0$$
where $\Gamma$ denotes the Gamma function and $\gamma \in (0,1)$ is a constant. We proved in \eqref{claim_exp_tail} that there exist  some constants $\gamma\in (0,1)$ and $K>0$ such that for any  $\alpha  \in (0,\alpha_\star]$ and any $\psi_\alpha \in \mathscr{E}_\al$,
\beq\label{claim_exp_tail2}
z_{k/2} \leq K^{k}, \qquad \forall k\in\N_*.
\eeq
Moreover, reproducing the arguments of both \cite[Lemma 4]{BoGaPa} and \cite[Lemma 7]{AloGamJMPA}, there exists some positive constant $A$ depending only on $\gamma$ such that
$$\mathcal{S}_p^{(j)} \leq A \Gamma(2p+1 + 2\gamma) \mathcal{Z}_p^{(j)} \qquad \forall p > 3/2$$
where
$$\mathcal{Z}_p^{(j)}=\max_{1 \leq k \leq [\frac{p+1}{2}]}\left(z_k^{(j)}\,z_{p-k+\frac{1}{2}} + z_k\,z_{p-k+\frac{1}{2}}^{(j)}, z_{k+\frac{1}{2}}^{(j)} z_{p-k} + z_{k+\frac{1}{2}}\,z_{p-k}^{(j)}\right).$$
Arguing as in the proof of Proposition \ref{prop:tails} and using \eqref{claim_exp_tail2}, for any $\alpha \in (0,\alpha_\star)$, there exist positive constants $c_i$, $i=1,\ldots,4$ such that
\begin{equation}\label{cizpj}c_1 {\frac{k}{2} }  z_{\frac{k+1}{2}}^{(j)} \leq c_2 \left(\frac{k}{2}\right)^\gamma \mathcal{Z}^{(j)}_{k/2} + \underline{\mathbf{b}}\,\frac{k}{2}\,z^{(j)}_{k/2} + c_3\,\frac{k}{2}\,K^{k+1}+c_4\,K^{k} \qquad k \geq 3.\end{equation}
Let us show then that there exists $C  \geq K$ such that
\begin{equation}\label{zpjC}
z_{\frac{k}{2}}^{(j)} \leq C^{k} \qquad \forall k \in \mathbb{N}.\end{equation}
One argues as in Proposition \ref{prop:tails}. Namely, choose $k_0 \in \mathbb{N}$ large enough so that 
$$ 2\:\frac{c_4}{k_0} \leq \underline{\mathbf{b}}\qquad \mbox{ and }\qquad c_2 \left(\frac{k_0}{2}\right)^{\gamma-1}\leq \frac{c_1}{2}, $$ 
and let $C \geq K > 1$ be such that
$$C \geq \max_{1 \leq k \leq k_0} \sup_{\alpha \in (0,\alpha_\star)} \sup_{\psi_\al \in \mathscr{E}_\al} z_{\frac{k}{2}}^{(j)} \quad \text{ and } \quad
2\underline{\mathbf{b}} \, C^{-1} + c_3 \left(\frac{K}{C}\right)^{k_0+1} \leq \frac{c_1}{2}.$$
Thanks to \eqref{moments_grad}, such a $C$ exists and let us prove by induction that  \eqref{zpjC} holds. If $k \leq k_0$, it readily holds by definition of $C$. Let now
$k_*\geq k_0$ and assume that \eqref{zpjC} holds for any $k\leq k_*$. Then, taking $k={k}_*$ in the above inequality \eqref{cizpj} and since  $\mathcal{Z}_{\frac{{k}_*}{2}}^{(j)}$ only involves $z_{\frac{k}{2}}^{j}$ for $k \leq {k}_*$,  we may use the induction hypothesis to get first that  $\mathcal{Z}_{\frac{k_*}{2}}^{(j)} \leq C^{k_*+1}$ (recall that $K \leq C$) and then to get
\begin{equation*}\begin{split}
c_1\,  z_{\frac{{k}_*+1}{2}}^{(j)}&\leq c_2 \left(\frac{{k}_*}{2} \right)^{\gamma-1}C^{{k}_*+1} + \left(\frac{2c_4}{k_*} +\underline{\mathbf{b}}\right)C^{{k}_*}+c_3\, K^{k_*+1}\\
&\leq c_2 \left(\frac{{k}_0}{2} \right)^{\gamma-1}C^{{k}_*+1} +\left(\frac{2c_4}{k_0} +\underline{\mathbf{b}}\right)C^{{k}_*}+c_3\, K^{k_*+1}
\leq  \frac{c_1}{2}\: C^{{k}_*+1} +2\underline{\mathbf{b}}\,C^{k_*} + c_3 K^{k_*+1}.\end{split}\end{equation*}
The choice of $C$ implies then that $2\underline{\mathbf{b}}\,C^{k_*} + c_3 K^{k_*+1} \leq \frac{c_1}{2}\: C^{k_*+1}$ since $k_* \geq k_0$. This proves the result.
\end{proof}

Thanks to the above technical Lemma, we are in position to extend Proposition \ref{prop:tails}  to the difference of two solutions as in \cite[Proposition 2.7, Step 1]{MiMo3}
\begin{prop}\label{prop:tailsdiff}
Let $\overline{\alpha}_1\in(0,\hat{\al}_1)$ where $\hat{\alpha}_1$ is defined in Theorem \ref{theo:HK}. There exist some positive constants $r >0$ and $M>0$ such that
\begin{equation*}
\int_{\R^d}|\psi_{\al,1}(\xi) - \psi_{\al,2}(\xi)|\exp(r \,|\xi|)\,\d\xi \leq M\,\|\psi_{\al,1}-\psi_{\al,2}\|_{L^1_3} \qquad \forall \al \in (0,{ \overline{\alpha}_1}]
\end{equation*}
for any $\psi_{\al,i} \in \mathscr{E}_\al$, $i=1,2$.
\end{prop}

\begin{proof}
Let $\alpha  \in (0,\overline{\alpha}_1]$ and $\psi_{\alpha,i} \in \mathscr{E}_\al$, for $i=1,2$. Then, for $i\in\{1,2\}$, $\psi_{\alpha,i}$ satisfies
\begin{equation}\label{AaiBbi}\mathbf{A}_\al^i \,\psi_{\al,i}(\xi)  + \mathbf{B}_\al^i \,\xi \cdot \nabla \psi_{\al,i}(\xi)=\mathbb{B}_\al(\psi_{\al,i},\psi_{\al,i})(\xi)\end{equation}
where $\mathbf{A}_\al^i$ and $\mathbf{B}_\al^i$ are defined by \eqref{ABalpha} with $\psi_\al$  obviously replaced by $\psi_{\al,i}$, $i=1,2$.  We set
$$ g_\al=\psi_{\al,1}-\psi_{\al,2} \qquad \mbox{ and } \qquad
s_\al=\psi_{\al,1}+\psi_{\al,2}.$$
Clearly, $g_\al$ satisfies
\begin{equation}\label{eqsym:gal}\mathbf{A}_\al^1\, g_\al(\xi) + \mathbf{B}_\al^1 \,\xi \cdot \nabla g_\al(\xi) = \frac{1}{2}\left(\mathbb{B}_\al(g_\al,s_\al)(\xi)+ \mathbb{B}_\al(s_\al,g_\al)(\xi)\right) + \mathcal{G}_\al(\xi)\end{equation}
with
\begin{equation}\label{eq:Gal}\mathcal{G}_\al(\xi)=\left(\mathbf{A}_\al^2-\mathbf{A}^1_\al\right)\,\psi_{\al,2}(\xi)+ \left(\mathbf{B}_\al^2-\mathbf{B}_\al^1\right)\,\xi \cdot \nabla \psi_{\al,2}(\xi).\end{equation}
Multiplying the previous equation by $\varphi(\xi)= |\xi|^{2p}\, \mathrm{sign}(g_\alpha(\xi))$ and integrating over $\R^d$, we get, after an integration by parts,
\beq \label{tobo1}
 \al (p-1) \,a_\al^1 D_p=\al \,p\, b_\al^1 D_p
+ \int_{\R^d} \mathcal{G}_\al (\xi)\varphi(\xi) \d\xi
+ \frac{1}{2} \int_{\R^d} \left(\mathbb{B}_\al(g_\al,s_\al)+ \mathbb{B}_\al(s_\al,g_\al)\right) \varphi \,\d \xi,
\eeq
where
$$D_p=\int_{\R^d} |\xi|^{2p}\, |g_\al(\xi)|\,\d\xi, \qquad
\mathbf{a}_\al^i=\mathbf{a}_{\psi_{\al,i}}\qquad \mbox{ and } \qquad  \mathbf{b}_\al^i=\mathbf{b}_{\psi_{\al,i}} \qquad \mbox{for } i=1,2.$$
Thanks to the pre/post-collisionnal change of variables, we have
\begin{multline*}
\int_{\R^d} \left(\mathbb{B}_\al(g_\al,s_\al)+ \mathbb{B}_\al(s_\al,g_\al)\right) \varphi \,\d \xi
=(1-\alpha) \int_{\R^d} \int_{\R^d} g_\al \, {s_\al}_* \int_{\S^{d-1}}( \varphi'+\varphi'_*)\:\frac{\d\sigma}{|\S^{d-1}|}\;  |\xi-\xi_*| \, \d\xi\, \d\xi_*\\
- \int_{\R^d} \int_{\R^d} g_\al \, {s_\al}_*( \varphi+\varphi_*)|\xi-\xi_*| \, \d\xi\, \d\xi_*
\end{multline*}
where, for any function $f$, we use the shorthand notations $f=f(\xi)$, $f_{*}=f(\xi_{*})$, $f'=f(\xi')$ and $f'_{*}=f(\xi'_{*})$. Thus,
\begin{multline*}
\frac{1}{2} \int_{\R^d} \left(\mathbb{B}_\al(g_\al,s_\al)+ \mathbb{B}_\al(s_\al,g_\al)\right) \varphi \,\d \xi
\leq (1-\alpha) \int_{\R^d} \int_{\R^d} |g_\al| \, {s_\al}_*  G_p(\xi,\xi_*)  |\xi-\xi_*| \, \d\xi\, \d\xi_*\\
-  \frac{1}{2} \int_{\R^d} \int_{\R^d} |g_\al| \, {s_\al}_*\, |\xi|^{2p}\, |\xi-\xi_*| \, \d\xi\, \d\xi_*
+  \frac{1}{2} \int_{\R^d} \int_{\R^d} |g_\al| \, {s_\al}_*  \, |\xi_*|^{2p}\, |\xi-\xi_*| \, \d\xi\, \d\xi_*,
\end{multline*}
where
$$G_p(\xi,\xi_*) := \frac{1}{2} \int_{\S^{d-1}} (|\xi'|^{2p} + |\xi'_*|^{2p}) \frac{\d\sigma}{|\S^{d-1}|} \leq \frac{1}{2} \varrho_p\, (|\xi|^2+|\xi_*|^2)^p,$$
by \cite[Lemma 3.1]{jde} with $\varrho_p$ defined by \eqref{varrhoK}.
Setting $\beta_p(\alpha)=(1-\alpha)\varrho_p$, we then deduce from the Jensen's inequality, the estimate $|\xi-\xi_*|\leq |\xi|+|\xi_*|$ and \cite[Lemma 2]{BoGaPa} that
\begin{multline*}
\frac{1}{2} \int_{\R^d} \left(\mathbb{B}_\al(g_\al,s_\al)+ \mathbb{B}_\al(s_\al,g_\al)\right) \varphi \,\d \xi \leq - (1-\beta_p(\alpha)) D_{p+1/2} \\
+ \frac{1}{2}\beta_p(\alpha) \sum_{k=1}^{[\frac{p+1}{2}]} \left(\bary{c} p\\ k\\ \eary\right)\int_{\R^d}\int_{\R^d} |g_\al| {s_\al}_*\left( |\xi|^{2k} |\xi_*|^{2(p-k)} + |\xi|^{2(p-k)}|\xi_*|^{2k}\right) ( |\xi|+|\xi_*| ) \d\xi\d\xi_* \\
+ \frac{1}{2} \;(1+\beta_p(\alpha))\int_{\R^d} \int_{\R^d} |g_\al| \,{s_\al}_*\,  |\xi_*|^{2p}( |\xi|+|\xi_*| ) \d\xi\d\xi_*.
\end{multline*}
Setting
$$S_p =\int_{\R^d} s_\alpha \, |\xi|^{2p} \, \d\xi,$$
we obtain
\begin{multline}\label{tobo2}
\frac{1}{2} \int_{\R^d} \left(\mathbb{B}_\al(g_\al,s_\al)+ \mathbb{B}_\al(s_\al,g_\al)\right) \varphi\,\d \xi \leq - (1-\beta_p(\alpha)) D_{p+1/2} \\
+ \frac{1}{2} \; \beta_p(\alpha) \sum_{k=1}^{[\frac{p+1}{2}]} \left(\bary{c} p\\ k\\ \eary\right)\left(D_{k+1/2} S_{p-k} +D_k S_{p-k+1/2} +D_{p-k+1/2} S_k + D_{p-k}S_{k+1/2} \right) \\
+ \frac{1}{2} (1+\beta_p(\alpha)) (D_{1/2} S_p+D_0 S_{p+1/2}).
\end{multline}
Next,
$$\int_{\R^d} \mathcal{G}_\al(\xi) \varphi(\xi)\, \d\xi
\leq \left|\mathbf{A}_\al^1-\mathbf{A}^2_\al\right| \,M_p(\psi_{\al,2})
+ \left|\mathbf{B}_\al^1-\mathbf{B}_\al^2\right| \, \left|\int_{\R^d} (\xi \cdot \nabla \psi_{\al,2}(\xi))\, \varphi(\xi)\, \d\xi\right|.$$
But,
$$  \left|\int_{\R^d} (\xi \cdot \nabla \psi_{\al,2}(\xi))\, \varphi(\xi)\, \d\xi\right| \leq M_{p{ +\frac{1}{2}}}(|\nabla \psi_{\al,2}|)$$
where we recall that, for any $\varphi \geq 0$, $M_p(\varphi)=\int_{\R^d}\varphi(\xi)\,|\xi|^{2p}\,\d\xi$, $p \geq 0.$ Moreover, by Proposition \ref{propmom}, there exists some constant $C>0$ depending only on $d$ such that
$$ \left|\mathbf{A}_\al^1-\mathbf{A}^2_\al\right| + \left|\mathbf{B}_\al^1-\mathbf{B}^2_\al\right| \leq c (D_0+D_{3/2}).$$
Thus, there is $C > 0$ (independent of $p$) such that
\beq\label{tobo3}
\int_{\R^d} \mathcal{G}_\al(\xi) \varphi(\xi)\, \d\xi
\leq C (D_0+D_{3/2})  M_{p+\frac{1}{2}}(|\nabla \psi_{\al,2}|+\psi_{\al,2}).
\eeq
Let $\gamma\in(0,1)$. Introducing renormalized moments
$$\delta_k=\frac{D_k}{(D_0+D_{3/2}) \Gamma(2k+\gamma)} \qquad \mbox{ and }\qquad
\sigma_k=\frac{S_k}{\Gamma(2k+\gamma)}, $$
setting
$$\tilde{Z}_p:=\max_{1\leq k\leq [\frac{p+1}{2}]} \left\{ \delta_{k+1/2} \sigma_{p-k}, \delta_{p-k} \sigma_{k+1/2}, \delta_k \sigma_{p-k+1/2} , \delta_{p-k+1/2} \sigma_k\right\},$$
and gathering \eqref{tobo1}, \eqref{tobo2} and \eqref{tobo3}, we obtain,
thanks to \cite[Lemma 4]{BoGaPa}, the existence of a constant $C_\gamma$ depending only on $\gamma$ such that
\begin{multline*}
(1-\beta_p(\alpha))\;\frac{\Gamma(2p+1+\gamma)}{\Gamma(2p+\gamma)} \;\delta_{p+1/2}
+\al (p-1) \,a_\al^1 \delta_p \leq  \al \,p\, b_\al^1 \delta_p
+ \frac{1}{2} \;C_\gamma\,  \beta_p(\alpha)\; \frac{\Gamma(2p+1+2\gamma)}{\Gamma(2p+\gamma)} \;
\tilde{Z}_p  \\
+ \frac{1}{2} (1+\beta_p(\alpha))
\left( \sigma_p+ \frac{\Gamma(2p+1+\gamma)}{\Gamma(2p+\gamma)}\; \sigma_{p+1/2}\right)
+C \; \frac{ M_{p+\frac{1}{2}}(|\nabla \psi_{\al,2}|+\psi_{\al,2})}{\Gamma(2p+\gamma)}.
\end{multline*}
Then, for $p\geq 1$, $\al (p-1) \,a_\al^1 \delta_p\geq 0$ and by Lemma \ref{estimab},  $b_\al^1$ is bounded uniformly in $\alpha$. Thus, arguing as in the proof of Proposition \ref{prop:tails}, there exist some constants $c_1,c_2,c_3,c_4,c_5>0$ such that, for $p\geq 3/2$,
\beq\label{tobo4}
c_1\, p\, \delta_{p+1/2} \leq  c_2 \, p \, \delta_p
+ c_3 \, p^\gamma \, \tilde{Z}_p  +  p\, \sigma_p+  c_4 \, p \, \sigma_{p+1/2}
+ c_5 \; \frac{M_{p+\frac{1}{2}}(|\nabla \psi_{\al,2}|+\psi_{\al,2})}{\Gamma(2p+\gamma)}.
\eeq
According to Proposition \ref{prop:tails} and Lemma \ref{lem:tails}, if $\gamma \in (0,1)$, there exists $\Lambda > 0$ such that
$$\frac{M_{\frac{k+1}{2}}(|\nabla \psi_{\al,2}|+\psi_{\al,2})}{\Gamma(k+1+\gamma)} \leq \Lambda^{k+1} \qquad \text{ for any } k \in\N.$$ Therefore, \eqref{tobo4} reads
$$c_1\, \frac{k}{2}\, \delta_{\frac{k+1}{2}} \leq  c_2 \, \frac{k}{2} \, \delta_{k/2}
+ c_3 \, \left( \frac{k}{2} \right)^\gamma \, \tilde{Z}_{k/2}  +  \frac{k}{2}\, \sigma_{k/2}+  c_4 \, \frac{k}{2} \, \sigma_{\frac{k+1}{2}}
+ c_5 \; k \Lambda^{k+1}.$$
One concludes then as in the proofs of Proposition \ref{prop:tails} and Lemma \ref{lem:tails} (details are left to the reader).
 \end{proof}

\section{Uniqueness and convergence results}\label{sec:sec3}

\subsection{Boltzmann limit}\label{sec:BoLi}

On the basis of the results of the previous section, we can prove the convergence of any solution $\psi_\al \in \mathscr{E}_\al$ towards the normalized Maxwellian $\M$ given by \eqref{M}. Namely, we have the following  convergence result:
\begin{theo}\label{convergenceMax}
 For any $k \geq 0$ and $m \geq 0$, one has
$$\lim_{\al \to 0}\|\psi_\al-\mathcal{M}\|_{\mathbb{H}^m_k}=0$$
where $\mathcal{M}$ is the Maxwellian
$$\mathcal{M}(\xi)=\pi^{-\frac{d}{2}}\exp\left(-|\xi|^2\right).$$
 \end{theo}
\begin{proof} The proof is inspired by \cite[Theorem 4.1]{AloLo3} and is based upon a compactness argument through Theorem \ref{theo:HK}. Namely, let us fix $m > d/2+1$ and $k_0 \geq 1.$ Let then $\alpha^{\dag}_m < \hat{\alpha}_m$ be given (where $\hat{\alpha}_m$ is the parameter in Theorem \ref{theo:HK}). According to Theorem \ref{theo:HK},
$$\sup_{\al \in (0, \alpha^{\dag}_m)}\sup_{\psi_\al \in \mathscr{E}_\al} \|\psi_\al\|_{\mathbb{H}^m_{k_0}} < \infty$$
and there is  a sequence $(\al_n)_n \subset (0,\alpha^{\dag}_m)$ with $\al_n \to 0$ and $\psi_0 \in \mathbb{H}^m_{k_0}$ such that $\left(\psi_{\al_n}\right)_n$ converges weakly, in $\mathbb{H}^m_{k_0}$, to $\psi_0$ (notice that, at this stage,  the limit function $\psi_0$ may depend  on the choice of $m$ and $k_0$). Using the decay of $\gl$ guaranteed by the \textit{polynomially weighted} Sobolev estimates, we can prove easily as in \cite[Theorem 4.1]{AloLo3} that the convergence is  \textit{strong} in $\mathbb{H}^1_k$ for any $0\leq k < k_0$:
\begin{equation}\label{strongconv}\lim_{n \to \infty}\|\psi_{\al_n}-\psi_0\|_{\mathbb{H}^1_k}=0.\end{equation}
It remains therefore to identify the limit $\psi_0$. Since,  for any $\alpha \in (0,\alpha^{\dag}_m)$, $\psi_\al$ satisfies \eqref{tauT}, one gets that
\begin{multline*}
\|\mathbb{B}_\al(\gl,\gl)\|_{L^2}= \|\mathbf{A}_{\psi_\al} \gl + \mathbf{B}_{\psi_\al} \xi \cdot \nabla \gl\|_{L^2} \\
\leq |\mathbf{A}_{\psi_\al}| \|\gl\|_{L^2} + |\mathbf{B}_{\psi_\al}|\,\|\xi \cdot \nabla \gl\|_{L^2}\\
\leq \left(|\mathbf{A}_{\psi_\al}|+|\mathbf{B}_{\psi_\al}|\right)\,\|\gl\|_{\mathbb{H}^1_1}.
\end{multline*}
Now, according to Remark \ref{ABbound}, one sees that there exists $C >0$ such that
$$\|\mathbb{B}_\al(\gl,\gl)\|_{L^2} \leq C \al\,\|\gl\|_{\mathbb{H}_1^1}$$
and, using the uniform Sobolev estimates provided by Theorem \ref{theo:HK}, one sees there exists $C_1 >0$ such that
$$\|\mathbb{B}_\al(\gl,\gl)\|_{L^2} \leq C_1\,\al \qquad \forall \alpha \in (0,\alpha^{\dag}_m).$$
In particular,
\begin{equation}\label{Qgl}\|\Q(\gl,\gl)\|_{L^2} \leq \|\mathbb{B}_\al(\gl,\gl)\|_{L^2} + \alpha \|\Q^+(\gl,\gl)\|_{L^2} \leq C_1 \al + \alpha \|\Q^+(\gl,\gl)\|_{L^2} \qquad \forall \alpha \in (0,\alpha^{\dag}_m).\end{equation}
By Theorem \ref{alogam}, there is some positive constants $C_2$ (independent of $\al$) such that
$$\|\Q^+(\gl,\gl)\|_{L^2} \leq C_2 \|\gl\|_{L^2_1}\,\|\gl\|_{L^1_1}$$
and, thanks to Propositions \ref{propmom} and \ref{propo:L2}, there exists $C_3 > 0$ so that \eqref{Qgl} reads
$$\|\Q(\gl,\gl)\|_{L^2}  \leq C_3\,\al \qquad \forall \alpha \in (0,\alpha^{\dag}_m).$$
In particular,  $\lim_{n \to \infty}\|\Q\left(\psi_{\al_n},\psi_{\al_n}\right)\|_{L^2}=0$ and, since $\psi_{\al_n}$ converges to $\psi_0$, one easily deduces that $\psi_0$ satisfies
$$\Q(\psi_0,\psi_0)=0$$
i.e. $\psi_0$ is a Maxwellian distribution. Now, according to \eqref{init}, we clearly get that $\psi_0=\M$. The above reasoning actually shows that {\it any} convergent subsequence $(\psi_{\al_n})_n$ with $\al_n \to 0$ is converging towards the same limit $\M$. As in \cite[Theorem 4.1]{AloLo3}, this means that the whole net $(\psi_\al)_{\al \in (0,\alpha^{\dag}_m)}$ is converging to $\M$ for the $\mathbb{H}^1_k$ topology. Arguing in the very same way we can prove that  the convergence actually holds in any weighted Sobolev space $\mathbb{H}^m_k$, $k \geq 0$ and $m \geq 0$.
\end{proof}
\begin{rmq} Because of the use of some compactness argument, the above convergence result is clearly non quantitative, i.e. no indication about the rate of convergence is provided.
\end{rmq}

As in \cite[Corollary 4.2]{AloLo3}, the above  convergence in Sobolev spaces can be extended easily to weighted $L^1$-spaces with exponential weights. Namely, for any $a \geq 0$, let
$$m_a(\xi)=\exp(a|\xi|),\qquad \qquad \xi \in \R^d.$$
Then, one has the following result (we refer to \cite[Corollary 4.2]{AloLo3} for a proof which uses simple interpolation together with Proposition \ref{prop:tails}):
\begin{cor}\label{cor:limit}  For any {$a \in [0,A/2)$ (where $A$ is given by Proposition \ref{prop:tails})} and any $k \geq 0$ it holds
$$\lim_{\al \to 0} \left\|\gl-\mathcal{M}\right\|_{L^1_k(m_a)}=0.$$
\end{cor}
\subsection{Uniqueness}\label{sec:unique}

We will now deduce from the above (non quantitative) convergence result that the set $\mathscr{E}_\al$ actually reduces to a singleton whenever $\alpha$ is small enough. Before proving such a result, we first establish some important estimate on the difference between two solutions to \eqref{tauT}:
\begin{prop}\label{prop:estimDiff} For any $N \geq 0$, there exist $\alpha^\ddag_N > 0$, $\mathfrak{q}(N) > 0$  and $C_N > 0$ such that
\begin{equation}\label{estimDiff}
\|\psi_{\al,1} - \psi_{\al,2}\|_{\mathbb{H}^N} \leq C_N\,\|\psi_{\al,1}-\psi_{\al,2}\|_{L^1_{\mathfrak{q}(N)}} \qquad \forall \alpha \in (0,\alpha_N^\ddag)
\end{equation}
for any $\psi_{\al,i} \in \mathscr{E}_\al$, $i=1,2$.
\end{prop}
\begin{proof} The proof uses some of the arguments of Theorem \ref{theo:HK} and follows the method of \cite[Proposition 2.7]{MiMo3}. For a given $\alpha  \in (0,\alpha_\star)$, let $\psi_{\al,1}$ and $\psi_{\al,2}$ be two elements of $\mathscr{E}_\al$. Set
$$g_\al=\psi_{\al,1} - \psi_{\al,2}.$$
Clearly, $g_\al$ satisfies
\begin{equation}\label{eq:gal}\mathbf{A}_\al^1\, g_\al + \mathbf{B}_\al^1 \,\xi \cdot \nabla g_\al(\xi) = \mathbb{B}_\al(g_\al,\psi_{\al,1})+ \mathbb{B}_\al(\psi_{\al,2},g_\al) + \mathcal{G}_\al(\xi)\end{equation}
with $\mathcal{G}_\al$ defined by \eqref{eq:Gal}.

\textit{First step: $N=0$}.  We actually prove  here a stronger estimate than \eqref{estimDiff}. Namely, we show that {there exists $\alpha^\ddag_0 > 0$ such that for any $k \geq 0$, there exists $C_k >0$} such that:
\begin{equation}\label{estimDiffk}
\|\psi_{\al,1} - \psi_{\al,2}\|_{L^2_k} \leq C_k\,\|\psi_{\al,1}-\psi_{\al,2}\|_{L^1_{k^*}} \qquad \forall \alpha \in (0,{\alpha^\ddag_0})
\end{equation}
where $k^*=\max(1+k,3).$  {Fix $\overline{\alpha}_1 \in (0,\hat{\alpha}_1)$ and $k \geq 0$ (see Theorem \ref{theo:HK} for the definition of $\hat{\alpha}_1$)}. Multiplying \eqref{eq:gal} by $g_\al(\xi)\langle \xi \rangle^{2k}$ and integrating over $\R^d$, we get
\begin{multline}\label{galL2_1}
\left(\mathbf{A}_\al^1-( k+\tfrac{d}{2})\mathbf{B}_\al^1\right) \,\|g_\al\|_{L^2_k}^2 +\int_{\R^d}g_\al(\xi)\Q^-(g_\al,\psi_{\al,1})(\xi)\langle \xi \rangle^{2k}\d \xi\\
\leq \mathcal{I}_1 + \mathcal{I}_2 -\mathcal{I}_3
+ k|\mathbf{B}_\al^1| \|g_\al\|_{L^2_{k-1}}^2.
\end{multline}
where
$$\mathcal{I}_1=\int_{\R^d} \mathcal{G}_\al(\xi)\,g_\al(\xi)\langle \xi \rangle^{2k}\d \xi, \quad \mathcal{I}_2=\int_{\R^d} g_\al(\xi) \left( \Q^+(g_\al,\psi_{\al,1})(\xi) + \Q^+(\psi_{\al,2},g_\al)(\xi)\right)\langle \xi \rangle^{2k} \d\xi,$$
and $$\mathcal{I}_3=
 \int_{\R^d} g_\al(\xi) \Q^-(\psi_{\al,2},g_\al)(\xi)\langle \xi \rangle^{2k}\d\xi.$$
Let us estimate these three terms separately. One has
$$|\mathcal{I}_1| \leq \|\mathcal{G}_\al\|_{L^2_k}\,\|g_\al\|_{L^2_k} \leq \left(|\mathbf{A}_\al^1-\mathbf{A}^2_\al| + |\mathbf{B}_\al^1-\mathbf{B}^2_\al|\right) \left(\|\psi_{\al,2}\|_{L^2_k} + \|\nabla \psi_{\al,2}\|_{L^2_{1+k}}\right)\,\|g_\al\|_{L^2_k}.$$
Now one easily gets that there exists $C > 0$ such that
\begin{equation}\label{diffAB}|\mathbf{A}_\al^1-\mathbf{A}^2_\al| + |\mathbf{B}_\al^1-\mathbf{B}^2_\al| \leq C \alpha \|g_\al\|_{L^1_3}\end{equation}
We deduce from \eqref{diffAB} and Theorem \ref{theo:HK} the existence of some  constant $C_1 > 0$ such that
\begin{equation}\label{I1}
|\mathcal{I}_1| \leq C_1 \,\|g_\al\|_{L^1_3} \,\|g_\al\|_{L^2_k} \qquad \forall \alpha \in (0,{\overline{\alpha}_1}).\end{equation}
On the other hand, by virtue of Theorem \ref{alogam}
\bean
|\mathcal{I}_2 |& \leq &   \left(\| \Q^+(g_\al,\psi_{\al,1})\|_{L^2_k} + \|\Q^+(\psi_{\al,2},g_\al)\|_{L^2_k} \right) \|g_\al\|_{L^2_k}\\
& \leq &  C_k \left(  \| \psi_{\al,1}\|_{L^2_{1+k}}+ \| \psi_{\al,2}\|_{L^2_{1+k}}\right)\|g_\al\|_{L^1_{1+k}}\,  \|g_\al\|_{L^2_k}
\eean
where, for $d\geq 3$,
 $$C_k= \frac{c_{1,k,2}(d) \, 2^{d/4}}{ |\S^{d-1}|} \int_{-1}^1 (1-x)^{(d-6)/4} \, (1+x)^{(d-3)/2}\, \d x <\infty.$$
Finally, by virtue of Proposition \ref{propo:L2}, the norms involving $\psi_{\al,i}$ $i=1,2$ are uniformly bounded with respect to $\al$ so that there exists $C_2 > 0$ so that
\begin{equation}\label{I22}|\mathcal{I}_2| \leq C_2 \, \|g_\al \|_{L^1_{1+k}}\,\|g_\al \|_{L^2_k} \qquad \forall \al \in (0,{\overline{\alpha}_1}).\end{equation}
To deal with the last term, one has
$$|\mathcal{I}_3 |\leq \int_{\R^d} |g_\al(\xi)|\langle \xi \rangle^{2k}\,\psi_{\al,2}(\xi)\d \xi \int_{\R^d} |g_\al(\xi_*)|\,|\xi-\xi_*|\d\xi_* \leq \|\psi_{\al,2}\|_{L^2_{1+k}}\,\|g_\al\|_{L^1_1}\,\|g_\al\|_{L^2_k}$$
and, according to Proposition \ref{propo:L2}, there exists $C_3 > 0$ such that
\begin{equation}\label{I3}
|\mathcal{I}_3 |\leq C_3\, \|g_\al\|_{L^1_1}\,\|g_\al\|_{L^2_k} \qquad \forall \alpha \in (0,{\overline{\alpha}_1}).\end{equation}
Now, according to \eqref{concentration}, one has
\begin{multline}\label{I4}
\int_{\R^d}g_\al(\xi)\Q^-(g_\al,\psi_{\al,1})(\xi)\langle \xi \rangle^{2k}\d \xi=\int_{\R^{2d}} g_\al^2(\xi)\langle \xi \rangle^{2k}\psi_{\al,1}(\xi_*)|\xi-\xi_*|\d\xi\d\xi_*
\\ \geq C_0 \int_{\R^d}g_\al^2(\xi)\langle \xi \rangle^{2k+1} \d\xi=C_0\|g_\al\|_{L^2_{k+1/2}}^2.\end{multline}
Therefore, collecting \eqref{galL2_1}, \eqref{I1}, \eqref{I22}, \eqref{I3} and \eqref{I4}, one gets
\begin{multline}\label{galL2}\left(\mathbf{A}_\al^1-( k+\tfrac{d}{2})\mathbf{B}_\al^1\right) \,\|g_\al\|_{L^2_k}^2
  +C_0 \|g_\al\|_{L^2_{k+1/2}}^2 \\ \leq \left(C_1 + C_2 + C_3\right)\|g_\al\|_{L^1_{k^*}}  \|g_\al\|_{L^2_k} + k|\mathbf{B}_\al^1| \|g_\al\|_{L^2_{k-1}}^2\end{multline}
As in the proof of Proposition \ref{propo:L2}, we first consider the case $k=0$ and then $k>0$. When $k=0$, we deduce from \eqref{2A-dB} that
$$\mathbf{A}_\al^1-\tfrac{d}{2}\mathbf{B}_\al^1=-\dfrac{\alpha}{4}\left(d+4\right)\mathbf{a}_\al^1 + \frac{\alpha\, d }{4}\;\mathbf{b}_\al^1 \geq -\frac{\al}{4}\left[d+4-\frac{d}{\sqrt{2}}\right]\mathbf{a}_\al^1 $$
with $$\mathbf{a}_\al^1= \int_{\R^d} \Q^-(\psi_{\al,1},\psi_{\al,1})(\xi)\, \d \xi \leq \sqrt{d}\leq \sqrt{2}\,\mathbf{b}_\al^1 $$
by virtue of Lemma \ref{estimab}. Therefore, one gets
$$-\frac{\al\sqrt{d}}{4}\left[d+4-\frac{d}{\sqrt{2}}\right]  \,\|g_\al\|_{L^2}^2 +C_0 \|g_\al\|_{L^2_{1/2}}^2 \leq\left(C_1 + C_2 + C_3\right)\|g_\al\|_{L^1_3}  \|g_\al\|_{L^2}.$$
Now, setting $\alpha^\ddag_0=\min\left( {\overline{\alpha}_1}, \frac{2C_0}{\sqrt{d}}(d+4-\frac{d}{\sqrt{2}})^{-1}\right)$, we have,
$$\frac{C_0}{2}\|g_\al\|_{L^2_{1/2}}^2 \leq \left(C_1 + C_2 + C_3\right)\|g_\al\|_{L^1_{3}}  \|g_\al\|_{L^2} \qquad \forall \alpha \in (0,\alpha_0^\ddag),$$
whence \eqref{estimDiffk} for $k=0$.

For $k>0$, using  Remark \ref{ABbound} and  bounding the $L^2_{k-1}$ norm by the $L^2_k$ one, \eqref{galL2} leads to
  $$C_0 \|g_\al\|^2_{L^2_{k+1/2}} \leq \tilde{C}_k \|g_\al\|_{L^2_k}^2 +\left(C_1 + C_2 + C_3\right)\|g_\al\|_{L^1_{k^*}}  \|g_\al\|_{L^2_k} $$
for some constant $\tilde{C}_k >0$ independent of $\al\in(0,{\overline{\alpha}_1})$.
 Now, one uses the fact that, for any $R > 0$,
 $$\|g_\al\|_{L^2_{k}}^2 \leq (1+R^2)^{k}\|g_\al\|_{L^2}^2+R^{-1}\|g_\al\|_{L^2_{k+1/2}}^2$$
and one can choose $R > 0$ large enough so that $\tilde{C}_k R^{-1}=C_0/2$ to obtain
$$\tfrac{C_0}{2}\|g_\al\|_{L^2_{k+1/2}}^2 \leq C_{R,k} \|g_\al\|_{L^2}^2 +\left(C_1 + C_2 + C_3\right)\|g_\al\|_{L^1_{k^*}}  \|g_\al\|_{L^2_k} \qquad \forall \al \in (0,\alpha^\ddag_0)\,,\,\forall k \geq 0.$$
Since we have already proved \eqref{estimDiffk} for $k=0$, we easily deduce that
\eqref{estimDiffk} holds for any $k\geq 0$.

\medskip
\textit{Second step: $N > 0$}. For larger $N $, one proves the result by induction using now Theorem \ref{regularite}. Namely, let $N \geq 1$ be fixed and assume that for any $0\leq n\leq N-1$, there exist $\alpha_n^\ddag>0$, $\mathfrak{q}(n)>0$  and  $C_n >0$ such that
\begin{equation*}
 \|\psi_{\al,1} - \psi_{\al,2}\|_{\mathbb{H}^n} \leq C_n\,\|\psi_{\al,1}-\psi_{\al,2}\|_{L^1_{\mathfrak{q}(n)}} \qquad \forall \alpha \in (0,\alpha_n^\ddag)
\end{equation*}
for any $\psi_{\al,i} \in \mathscr{E}_\al$, $i=1,2$. 
Let now $\ell$ be a given multi-index with $|\ell|=N $ and set $G_\ell=\partial^\ell g_\al$. From  \eqref{eq:gal}, $G_\ell$ satisfies
$$
\left(\mathbf{A}_\al^1+|\ell|\,\mathbf{B}_\al^1\right)\,G_\ell + \mathbf{B}_\al^1 \xi \cdot \nabla G_\ell = \partial^\ell \mathbb{B}_\al( g_\al,\psi_{\al,1}) + \partial^\ell \mathbb{B}_\al(\psi_{\al,2},g_\al) +\partial^\ell \mathcal{G}_\al.$$
Multiplying this identity by $G_\ell$ and integrating over $\R^d$ yields, as above,
\begin{multline*}
\left(\mathbf{A}_\al^1+\left(|\ell|-\frac{d}{2}\right)\,\mathbf{B}_\al^1\right)\|G_\ell\|_{L^2}^2 =  (1-\al)\int_{\R^d}\left[\partial^\ell \Q^+(g_\al,\psi_{\al,1})+ \partial^\ell \Q^+(\psi_{\al,2},g_\al)\right]\,G_\ell\,\d\xi\\
-\int_{\R^d}\partial^\ell\Q^-(g_\al,\psi_{\al,1})\,G_\ell\,\d\xi -\int_{\R^d}\partial^\ell \Q^-(\psi_{\al,2},g_\al)\,G_\ell\,\d\xi+\int_{\R^d}\partial^\ell \mathcal{G}_\al\,\,G_\ell\,\d\xi.
\end{multline*}
Recall, as in the proof of Theorem \ref{theo:HK}, that
$$\partial^\ell\Q^-(g_\al,\psi_{\al,1})=\sum_{\nu=0}^\ell \left(\begin{array}{c}
\ell \\ \nu
\end{array}\right) \Q^-(\partial^\nu g_\al,\partial^{\ell-\nu}\psi_{\al,1}).$$
and, for $\nu=\ell$, one has
$$\int_{\R^d}\Q^-(\partial^\ell g_\al,\psi_{\al,1})\,G_\ell\d\xi=\int_{\R^{2d}}G_\ell^2(\xi)\psi_{\al,1}(\xi_*)|\xi-\xi_*|\d\xi\d\xi_*\geq C_0\|G_\ell\|_{L^2_{1/2}}^2$$
thanks to \eqref{concentration}. Thus,
\begin{multline*}
  \left(\mathbf{A}_\al^1+\left(|\ell|-\frac{d}{2}\right)\,\mathbf{B}_\al^1\right)\|G_\ell\|_{L^2}^2 + C_0\|G_\ell\|_{L^2_{1/2}}^2 \\
  \leq (1-\al)\int_{\R^d}\left(\partial^\ell \Q^+(g_\al,\psi_{\al,1})+ \partial^\ell \Q^+(\psi_{\al,2},g_\al)\right)\,G_\ell\,\d\xi \\
+ \underset{\nu\neq \ell}{\sum_{\nu=0}^\ell}\left(\begin{array}{c}
\ell \\ \nu
\end{array}\right) \int_{\R^d}\Q^-(\partial^\nu g_\al,\partial^{\ell-\nu}\psi_{\al,1})\,G_\ell\,\d\xi\\
-\int_{\R^d}\partial^\ell \Q^-(\psi_{\al,2},g_\al)\,G_\ell\,\d\xi+\int_{\R^d}\partial^\ell \mathcal{G}_\al\,\,G_\ell\,\d\xi.
\end{multline*}
As in Theorem \ref{theo:HK} (see Eq. \eqref{Q-nu}), one obtains
$$\underset{\nu\neq \ell}{\sum_{\nu=0}^\ell}\left(\begin{array}{c}
\ell \\ \nu
\end{array}\right) \int_{\R^d}\Q^-(\partial^\nu g_\al,\partial^{\ell-\nu}\psi_{\al,1})\,G_\ell\,\d\xi \leq C\|G_\ell\|_{L^2}\underset{\nu\neq \ell}{\sum_{\nu=0}^\ell}\left(\begin{array}{c}
\ell \\ \nu
\end{array}\right)\|\partial^\nu g_\al\|_{L^2}$$
for some positive constant $C>0$ depending only on uniform weighted $L^2$-norms of $\partial^\sigma \psi_{\al,1}$ (with $|\sigma| < |\ell|$). Therefore, thanks to the induction hypothesis, there is $C_1 >0$ such that
 $$\underset{\nu\neq \ell}{\sum_{\nu=0}^\ell}\left(\begin{array}{c}
\ell \\ \nu
\end{array}\right) \int_{\R^d}\Q^-(\partial^\nu g_\al,\partial^{\ell-\nu}\psi_{\al,1})\,G_\ell\,\d\xi \leq C_1\,\|G_\ell\|_{L^2}\,\|g_\al\|_{L^1_{\mathfrak{q}(N-1)}} \qquad \forall \alpha \in (0,\alpha^\ddag_{N-1}).$$
Now, in the same way
$$\partial^\ell \Q^-(\psi_{\al,2},g_\al)=\sum_{\nu=0}^\ell \left(\begin{array}{c}
\ell \\ \nu
\end{array}\right) \Q^-(\partial^\nu \psi_{\al,2},\partial^{\ell-\nu}g_{\al}).$$
For any $\nu$ with $\nu\neq \ell$, there exists $i_0 \in \{1,\ldots,d\}$ such that $\ell_{i_0}-\nu_{i_0} \geq 1$ and integration by parts yields
$$|\Q^-(\partial^\nu \psi_{\al,2},\partial^{\ell-\nu}g_{\al})| \leq \left|\partial^{\,\nu} \psi_{\al,2}(\xi)\right|\,\|\partial^{\,\sigma }g_\al\|_{L^1}$$
where $\sigma=(\sigma_1,\ldots,\sigma_d)$ is defined with $\sigma_{i_0}=\ell_{i_0}-\nu_{i_0}-1$ and $\sigma_i=\ell_i-\nu_i$ if $i \neq i_0.$ Therefore, if $\nu \neq \ell$, one has
$$ \|\Q^-(\partial^\nu \psi_{\al,2},\partial^{\ell-\nu}g_{\al})\|_{L^2} \leq \|\psi_{\al,2}\|_{\mathbb{H}^{N-1}}\,\|\partial^{\,\sigma }g_\al\|_{L^1}.$$
For $\nu=\ell$, we have directly
$$\|\Q^-(\partial^\ell \psi_{\al,2},g_{\al})\|_{L^2} \leq \|\psi_{\al,2}\|_{\mathbb{H}_1^N}\,\|g_\al\|_{L^1_1}$$
so that,  thanks to the uniform bounds on the derivatives of $\psi_{\al,2}$ provided by Theorem \ref{theo:HK}, there exists $C > 0$ so that
$$\|\partial^\ell  \Q^-(\psi_{\al,2},g_\al)\|_{L^2}  \leq C \|g_\al\|_{\mathbb{W}^{N-1,1}_1}$$
and
$$\int_{\R^d}\partial^\ell \Q^-(\psi_{\al,2},g_\al)\,G_\ell\,\d\xi \leq C \|G_\ell\|_{L^2}\,\|g_\al\|_{\mathbb{W}^{N-1,1}_1}.$$
Now, thanks to Bouchut-Desvillettes estimates for $\Q^+$ (see Theorem \ref{theoBD}),  {Theorem \ref{theo:HK} and Lemma \ref{interpolSob}} it is easy to deduce that
\begin{multline*}
\int_{\R^d}\left(\partial^\ell \Q^+(g_\al,\psi_{\al,1})+ \partial^\ell \Q^+(\psi_{\al,2},g_\al)\right)\,G_\ell\,\d\xi  \leq \|\Q^+(g_\al,\psi_{\al,1}) + \Q^+(\psi_{\al,2},g_\al)\|_{\mathbb{H}^{N}} \|G_\ell\|_{L^2}\\
\leq C\left(\|g_\al\|_{L^1_2} + \|g_\al\|_{\mathbb{H}_2^{N-\frac{d-1}{2}}}\right)\|G_\ell\|_{L^2} 
\end{multline*}
for some positive constant $C > 0$ depending on uniform weighted $\mathbb{H}^{N-\frac{d-1}{2}}$ and $L^1$ norms of both $\psi_{\al,1}$ and $\psi_{\al,2}$. Finally,
$$\int_{\R^d} \partial^\ell \mathcal{G}_\al\,\,G_\ell\,\d\xi \leq \|\partial^\ell\mathcal{G}_\al\|_{L^2}\,\|G_\ell\|_{L^2}
\leq C \left(|\mathbf{A}_\al^1 -\mathbf{A}_\al^2| + | \mathbf{B}_\al^1-\mathbf{B}_\al^2|\right) \|\psi_{\al,2}\|_{\mathbb{H}^{N+1}_1}\|G_\ell\|_{L^2}$$
so that, as in Theorem \ref{theo:HK},
$$\int_{\R^d} \partial^\ell \mathcal{G}_\al\,\,G_\ell\,\d\xi \leq c \|g_\al\|_{L^1_3}\,\|G_\ell\|_{L^2}$$
for some positive constant $c > 0$. Gathering all theses estimates, we obtain the existence of a positive constant $C > 0$ such that
\begin{multline*}
  \left(\mathbf{A}_\al^1+\left(N-\frac{d}{2}\right)\,\mathbf{B}_\al^1\right)\|G_\ell\|_{L^2}^2 + C_0\|G_\ell\|_{L^2_{1/2}}^2\\
\leq C\left(\|g_\al\|_{L^1_{\mathfrak{q}(N-1)}}
+ \|g_\al\|_{\mathbb{H}^{N-\frac{d-1}{2}}_2} + \|g_\al\|_{\mathbb{W}^{N-1,1}_1} \right)\;\|G_\ell\|_{L^2}\end{multline*}
Now, estimating the $L^1$ norms by weighted $L^2$ norms as above {and using Lemma \ref{interpolSob}}, we get that there exists $q > 0$ so that
 $$\left(\mathbf{A}_\al^1+\left(N-\frac{d}{2}\right)\,\mathbf{B}_\al^1\right)\|G_\ell\|_{L^2}^2 + C_0\|G_\ell\|_{L^2_{1/2}}^2 \leq C\left(\|g_\al\|_{L^1_{\mathfrak{q}(N-1)}} + \|g_\al\|_{\mathbb{H}^{N-1}_q} \right) \;\|G_\ell\|_{L^2}.$$

Since $\left|\mathbf{A}_\al^1+\left(N-\frac{d}{2}\right)\,\mathbf{B}_\al^1\right|=\mathbf{O}(\alpha)$, choosing $\alpha_N^\ddag$ small enough (but explicit), the above left-hand side can be bounded from below by $\frac{C_0}{2}\|G_\ell\|_{L^2}^2$ to get
$$\frac{C_0}{2}\|G_\ell\|_{L^2}^2 \leq C\left(\|g_\al\|_{L^1_{\mathfrak{q}(N-1)}} + \|g_\al\|_{\mathbb{H}^{N-1}_q} \right) \;\|G_\ell\|_{L^2} \qquad \forall \alpha \in (0,\alpha_N^\ddag).$$
Now, using Lemma \ref{interpolSob} in Appendix (see also Remark \ref{rmqinterpolSob}) with $s_2=k_1=0$ and $s=N-1$, $s_1=N$, we get that there exists $C >0 $ such that
$$\|g_\al\|_{\mathbb{H}^{N-1}_q} \leq C\left(\|g_\al\|_{\mathbb{H}^N} + \|g_\al\|_{L^2_{qN}}\right).$$
Moreover, according to Step 1, up to reduce $\alpha_N^\ddag$ again {(so that $\alpha_N^\ddag\leq\alpha_0^\ddag $)}, one sees easily that there exists $\mathfrak{q}(N) > 0$ and $c=c(q,N) >0$ such that
$$\|g_\al\|_{L^2_{qN}} \leq c\|g_\al\|_{L^1_{\mathfrak{q}(N)}} \qquad \forall \alpha < \alpha_N^\ddag.$$
Gathering all this, one obtains the existence of some constant $C >0$ so that
$$\frac{C_0}{2}\|G_\ell\|_{L^2}^2 \leq C \left(\|g_\al\|_{L^1_{\mathfrak{q}(N)}} + \|g_\al\|_{\mathbb{H}^{N-1}} \right) \;\|G_\ell\|_{L^2} \qquad  \forall \alpha \in (0,\alpha_N^\ddag)$$
and this concludes the proof thanks to the induction hypothesis.
 \end{proof}

We have the following consequence of the above estimate (we refer to \cite[Proposition 3.8]{AloLo3} for the proof which uses simple interpolation combined with Proposition \ref{prop:tailsdiff}):
\begin{cor}\label{cor:estimdiff} {Let $a\in[0,r/12]$ (where $r$ is given by Proposition \ref{prop:tailsdiff}).} For any $N \geq 0$, there exist $\alpha^\ddag_N > 0$  and $C_N > 0$ such that
\begin{equation}\label{estimDiffL1}
\|\psi_{\al,1} - \psi_{\al,2}\|_{\mathbb{W}^{N,1}_1(m_a)} \leq C_N\,\|\psi_{\al,1}-\psi_{\al,2}\|_{L^1(m_a)} \qquad \forall \alpha \in (0,\alpha_N^\ddag)
\end{equation}
for any $\psi_{\al,i} \in \mathscr{E}_\al$, $i=1,2$.
\end{cor}
We introduce here the spaces
$$\mathcal{X}=L^1(m_a) \qquad \text{ and } \qquad \mathcal{Y}=L^1_1(m_a)$$
with $m_a(\xi)=\exp(a|\xi|),$ where $a  > 0$  is small enough (the precise range of parameters will be specified when needed).  We recall here the continuity properties of $\Q^{\pm}$ in this space: there exists $C_1 > 0$ such that
\begin{equation}\label{continueQ1}
\left\|\Q^{\pm}(f,g)\right\|_{\mathcal{X}}+ \left\|\Q^{\pm}(g,f)\right\|_{\mathcal{X}} \leq C_1  \|f\|_{\mathcal{Y}}\|g\|_{\mathcal{Y}}.
\end{equation}
In this space, let us introduce the Boltzmann linearized operator $\mathscr{L}\::\:\D(\mathscr{L}) \subset \mathcal{X} \to \mathcal{X}$ with $\D(\mathscr{L})=\mathcal{Y}$ and
$$\mathscr{L}(g)=\Q(g,\M)+\Q(\M,g), \qquad g \in \mathcal{Y}$$
where $\M$ is the Maxwellian distribution defined in \eqref{M}. The spectral analysis of $\mathscr{L}$ in the space $\mathcal{X}$ is by now well-documented \cite{Mo,BCL} and $0$ is a simple eigenvalue of $\mathscr{L}$ associated to the null set 
$$\mathscr{N}(\mathscr{L} )=\mathrm{Span}(\mathcal{M},\xi_1\mathcal{M},\ldots,\xi_d\mathcal{M},|\xi|^2\mathcal{M});$$
while  $\mathscr{L}$ admits a positive spectral gap $\nu >0$. 
  In particular, if
$$\widehat{\mathcal{X}}=\left\{g \in \mathcal{X}\,;\,\int_{\R^d} g(\xi)\d \xi=\int_{\R^d}{\xi_{j}}\,g(\xi)\,\d \xi=\int_{\R^d} |\xi|^2 g(\xi)\,\d \xi=0 {\text{ for any } j=1,\ldots,d}\right\},$$
and 
   $$\widehat{\mathcal{Y}}=\mathcal{Y} \cap \widehat{\mathcal{X}}$$
then   $\mathscr{N}(\mathscr{L}) \cap \widehat{\mathcal{Y}}=\{0\}$ and
  $\mathscr{L} $ is invertible from $\widehat{\mathcal{Y}}$ to
  $\widehat{\mathcal{X}}$ and there exists some explicit $c_0:=\left\|\mathscr{L}^{-1}\right\|_{\widehat{\mathcal{X}} \to \widehat{\mathcal{Y}}}$ such that
  \begin{equation}\label{XY2}
\|\mathscr{L}(g)\|_\mathcal{X} \geq c_0 \|g\|_\mathcal{Y} \qquad \forall g \in \widehat{\mathcal{Y}}.
\end{equation}
With this in hands, one can prove the following (non quantitative) uniqueness result
\begin{prop}\label{uniquenonquant} There exists some  $\alpha^\sharp$ such that the set $\mathscr{E}_\al$ reduces to a singleton for any $\alpha \in [0,\alpha^\sharp]$, i.e. for $\alpha \in [0,\alpha^\sharp]$ there exists a unique solution $\psi_\al$ to \eqref{tauT} that satisfies \eqref{init}.
\end{prop}

\begin{proof} The proof, as explained in the introduction, follows an approach initiated in \cite{MiMo3} and revisited (and somehow simplified) in \cite{AloLo3, BCL}. We shall work in the above spaces $\mathcal{X}=L^1(m_a)$  and $\mathcal{Y}=L^1_1(m_a)$ { with $a\in(0,\min\{A/2,A_1,r/12\})$ where $A$, $A_1$  and $r$ are given respectively by Proposition \ref{prop:tails}, Lemma \ref{lem:tails} and Corollary \ref{cor:estimdiff}}. For any $\al \in (0,\al_\star)$, let as above  $\psi_{\al,1}$ and $\psi_{\al,2}$ be two elements of $\mathscr{E}_\al$, i.e. $\psi_{\al,i}$ satisfies \eqref{AaiBbi} for $i=1,2$ and set
$$g_\al=\psi_{\al,1} - \psi_{\al,2}.$$
According to Proposition \ref{prop:tailsdiff}, $g_\al \in \mathcal{Y}$ and, since both $\psi_{\al,1}$ and $\psi_{\al,2}$ satisfy \eqref{init}, one actually has
$$g_\al \in \widehat{\mathcal{Y}}.$$ Moreover, $g_\al$ satisfies \eqref{eq:gal} from which we easily check that
\begin{multline*}
\mathscr{L}(g_\al)=\bigg[\Q(g_\al,\M)-\mathbb{B}_\al(g_\al,\M)\bigg] + \bigg[\Q(\M,g_\al)-\mathbb{B}_\al(\M,g_\al)\bigg] \\
+ \bigg[\mathbb{B}_\al(g_\al,\M-\psi_{\al,1}) + \mathbb{B}_\al(\M-\psi_{\al,2},g_\al)\bigg] + \mathbf{A}_\al^1 g_\al + \mathbf{B}_\al^1\,\xi \cdot \nabla g_\al(\xi)\\
+ \bigg(\mathbf{A}_\al^1 -\mathbf{A}_\al^2\bigg)\psi_{\al,2} + \bigg(\mathbf{B}_\al^1-\mathbf{B}_\al^2\bigg) \,\xi \cdot \nabla \psi_{\al,2}(\xi).
\end{multline*}
We compute then the $L^1(m_a)$ norm of $\mathscr{L}(g_\al)$ to get
\begin{multline}\label{Lgal}
\|\mathscr{L}(g_\al)\|_{\mathcal{X}} \leq \big\| \Q(g_\al,\M)-\mathbb{B}_\al(g_\al,\M)\big\|_{\mathcal{X}}+ \big\|\Q(\M,g_\al)-\mathbb{B}_\al(\M,g_\al)\big\|_{\mathcal{X}} \\
+ \big\|\mathbb{B}_\al(g_\al,\M-\psi_{\al,1})\big\|_{\mathcal{X}} + \big\|\mathbb{B}_\al(\M-\psi_{\al,2},g_\al)\big\|_{\mathcal{X}} \\
+ |\mathbf{A}_\al^1|\|g_\al\|_{\mathcal{X}} + |\mathbf{B}_\al^1| \|\nabla g_\al\|_{\mathcal{Y}} \\
+ \bigg(\left|\mathbf{A}_\al^1 -\mathbf{A}_\al^2\right|+ \left|\mathbf{B}_\al^1-\mathbf{B}_\al^2\right|\bigg) \|\psi_{\al,2}\|_{\mathbb{W}^{1,1}_1(m_a)}.
\end{multline}
According to the continuity estimate \eqref{continueQ1}, one easily sees that there exists $C > 0$ such that
$$\big\| \Q(g_\al,\M)-\mathbb{B}_\al(g_\al,\M)\big\|_{\mathcal{X}}+ \big\|\Q(\M,g_\al)-\mathbb{B}_\al(\M,g_\al)\big\|_{\mathcal{X}} \leq C \al \|g_\al\|_{\mathcal{Y}} \qquad \forall \al \in (0,\al_\star)$$
and
\begin{multline*} \big\|\mathbb{B}_\al(g_\al,\M-\psi_{\al,1})\big\|_{\mathcal{X}} + \big\|\mathbb{B}_\al(\M-\psi_{\al,2},g_\al)\big\|_{\mathcal{X}} \leq \\
 C \|g_\al\|_{\mathcal{Y}}\left(\|\M-\psi_{\al,1}\|_{\mathcal{Y}} + \|\M-\psi_{\al,2}\|_{\mathcal{Y}}\right) \qquad \forall \al \in (0,\al_\star).\end{multline*}
Moreover, according to Remark \ref{ABbound}, one also has
$$|\mathbf{A}_\al^1|\|g_\al\|_{\mathcal{X}} + |\mathbf{B}_\al^1| \|\nabla g_\al\|_{\mathcal{Y}}  \leq C\alpha\,\|g_\al\|_{\mathbb{W}^{1,1}_1(m_a)} \qquad \forall \al \in (0,\al_\star).$$
Finally, thanks to {Proposition \ref{prop:tails},  Lemma \ref{lem:tails}} together with \eqref{diffAB}, there exist some explicit $\delta \in (0,\al_\star)$ and some positive constant $C >0$  so that
$$\bigg(\left|\mathbf{A}_\al^1 -\mathbf{A}_\al^2\right|+ \left|\mathbf{B}_\al^1-\mathbf{B}_\al^2\right|\bigg) \|\psi_{\al,2}\|_{\mathbb{W}^{1,1}_1(m_a)} \leq C \alpha \|g_\al\|_{L^1_3} \leq C \alpha \|g_\al\|_{\mathcal{Y}} \qquad \forall \alpha \in (0,\delta).$$
Gathering all these estimates, we deduce from \eqref{Lgal} that there is some $\mu \in (0,\delta)$ and some positive constant $C > 0$ independent of $\alpha$ so that
\bean
\|\mathscr{L}(g_\al)\|_{\mathcal{X}} & \leq & C \alpha \,\|g_\al\|_{\mathbb{W}^{1,1}_1(m_a)}  +
C\,\|g_\al\|_{\mathcal{Y}} \left(\|\M-\psi_{\al,1}\|_{\mathcal{Y}} + \|\M-\psi_{\al,2}\|_{\mathcal{Y}}\right)\\
 & \leq & C \alpha \|g_\al\|_{\mathcal{Y}} + C\,\|g_\al\|_{\mathcal{Y}} \left(\|\M-\psi_{\al,1}\|_{\mathcal{Y}} + \|\M-\psi_{\al,2}\|_{\mathcal{Y}}\right)
 \qquad \forall \al \in (0,\mu).
\eean
 where we used \eqref{estimDiffL1} to bound the $\mathbb{W}^{1,1}_1(m_a)$ norm of $g_\al$ by its $\mathcal{Y}$ norm. Now, according to Corollary \ref{cor:limit}, for any $\varepsilon > 0$, there exists $\alpha_\varepsilon > 0$ so that
 $$\left(\|\M-\psi_{\al,1}\|_{\mathcal{Y}} + \|\M-\psi_{\al,2}\|_{\mathcal{Y}}\right) \leq \varepsilon \qquad \forall \alpha \in (0,\alpha_\varepsilon)$$
 from which we get that
 $$\|\mathscr{L}(g_\al)\|_{\mathcal{X}} \leq C\left(\alpha  + \varepsilon\right)\|g_\al\|_{\mathcal{Y}} \qquad \forall \alpha \in (0,\alpha_\varepsilon).$$
 Now, from \eqref{XY2}, recalling that $g_\al \in \widehat{\mathcal{Y}}$, one obtains
 $$c_0 \|g_\al\|_{\mathcal{Y}} \leq C\left(\alpha  + \varepsilon\right)\|g_\al\|_{\mathcal{Y}} \qquad \forall \alpha \in (0,\alpha_\varepsilon).$$
Therefore, choosing $\varepsilon > 0$ and $\alpha$ small enough so that $C(\alpha + \varepsilon) < c_0$, one gets that $\|g_\al\|_{\mathcal{Y}}=0$ which proves the result.\end{proof}

\subsection{Quantitative version of the uniqueness result}\label{sec:quanti} Notice that, because of the method of proof which uses Theorem \ref{theo:HK}, the above parameter $\alpha^\sharp$ is not explicit and depends on the rate of convergence of $\psi_{\al}$ towards $\M$. As in \cite{AloLo3}, it is enough to estimate the rate of convergence towards $\M$ to get some explicit estimate of $\alpha^\sharp$. This is the object of the following
\begin{prop} There exist  an explicit $\delta^\dag$ and some explicit $\kappa > 0$ such that
$$\sup_{\psi_\al \in \mathscr{E}_\al} \|\psi_{\al} -\M\|_{\mathcal{X}} \leq \kappa\,\alpha \qquad \forall \alpha \in (0,\delta^\dag).$$
\end{prop}
\begin{proof} Let $a<\min\{A,A_1\}$ where $A$ and $A_1$ are given respectively by Proposition \ref{prop:tails} and  Lemma \ref{lem:tails}. As in \cite{AloLo3}, the idea of the proof is to find a {\it nonlinear} estimate for $\|\psi_\al- \mathcal{M}\|_{L^1(m_a)}$. Namely, let $\alpha \in (0,\alpha_\star)$ and $\psi_\al \in \mathscr{E}_\al$ be given.  One simply notices that, since $\Q(\M,\M)=0$,
\begin{multline*}
\mathscr{L}(\psi_\al-\M)=\Q(\psi_\al-\M,\M - \psi_\al) + \mathbb{B}_\al(\psi_\al,\psi_\al) + \bigg[\Q(\psi_\al,\psi_\al) -\mathbb{B}_\al(\psi_\al,\psi_\al)\bigg]\\
= \Q(\psi_\al-\M,\M-\psi_\al) + \mathbf{A}_{\psi_\al} \psi_\al + \mathbf{B}_{\psi_\al} \xi \cdot \nabla \psi_\al  +  \bigg[\Q(\psi_\al,\psi_\al) -\mathbb{B}_\al(\psi_\al,\psi_\al)\bigg].
\end{multline*}
Therefore,
\begin{multline*} \|\mathscr{L}(\psi_\al -\M) \|_{\mathcal{X}} \leq \|\Q(\psi_\al-\M,\M - \psi_\al) \|_{\mathcal{X}} + \left(|\mathbf{A}_{\psi_\al}| + |\mathbf{B}_{\psi_\al}|\right)\,\|\psi_\al\|_{\mathbb{W}^{1,1}_1(m_a)} \\+ \|\Q(\psi_\al,\psi_\al) -\mathbb{B}_\al(\psi_\al,\psi_\al)\|_{\mathcal{X}}\\
\leq C_1\,\|\psi_\al-\M\|_{\mathcal{Y}}^2 + C \alpha\,\|\psi_\al\|_{\mathbb{W}^{1,1}_1(m_a)} + C_1 \alpha\,\|\psi_\al\|_{\mathcal{Y}}^2\end{multline*}
where we used the continuity property of $\Q^{\pm}$ in \eqref{continueQ1} together with Remark \ref{ABbound}. Now, {thanks to Proposition \ref{prop:tails} and  Lemma \ref{lem:tails}}, one sees that, for some explicit $\delta > 0$, $\sup_{\al \in (0,\delta)} \|\psi_{\al}\|_{\mathbb{W}^{1,1}_1(m_a)} < \infty$ from which we see that there exist two positive constants $c_1,c_2 > 0$ such that
\begin{equation}\label{Lexp}
\|\psi_\al-\M \|_{\mathcal{Y}} \leq c_1 \|\psi_\al -\M\|_{\mathcal{Y}}^2 + c_2 \alpha \qquad \forall \alpha \in (0,\delta)
\end{equation}
where we also used \eqref{XY2} by noticing that $\psi_\al - \M \in \widehat{\mathcal{Y}}$. Now, since $\lim_{\al \to 0}\|\psi_\al -\M\|_{\mathcal{Y}}=0$, there exists some $\delta^\dag < \delta$ (non explicit at this state) such that
$$c_1 \|\psi_\al -\M\|_{\mathcal{Y}} \leq \frac{1}{2} \qquad \forall \al \in (0,\delta^\dag)$$
and estimate \eqref{Lexp} becomes
$$\|\psi_\al -\M\|_{\mathcal{Y}} \leq 2c_2\,\alpha \qquad \forall \alpha \in (0,\delta^\dag).$$
Such an estimates provides actually an explicit estimate for $\delta^\dag$ since the optimal parameter $\delta^\dag$ is the one for which the two last estimates are identity yielding $\delta^\dag = \frac{1}{4c_1\,c_2}.$ Since both $c_1$ and $c_2$ are explicitly computable, we get our result with $\kappa=2c_2.$
\end{proof}

With this in hands, one can complete the proof of Theorem \ref{theo:main}
\begin{proof}[Proof of Theorem \ref{theo:main}] As already explained, the only non quantitative estimate in the proof of Proposition \ref{uniquenonquant} was the convergence of $\psi_{\al,i}$ towards $\M$. This is made explicit by the above Proposition from which we conclude, as in \cite[Theorem 4.9]{AloLo3}, that the parameter $\alpha^\sharp$ is explicitly computable.  Details are omitted.
\end{proof}

\appendix

\section{Regularity properties of $\Q^+$ revisited}\label{sec:reguQ+}

We prove in this section the regularity result Theorem \ref{regularite}. The proof follows the paths of the similar result established in \cite[Theorem 2.5]{AloLo3} in dimension $d=3$ for the collision operator associated to inelastic collisions. The proof is simpler here since we are dealing with {\it elastic} interactions, however, the result differs in some points since we are dealing with dimension $d$ arbitrary: the case $d=3$ is very peculiar since exactly one derivative is gained. In general dimension $d$, the regularizing properties concern  $\frac{d-1}{2}$ derivative. The proof given in \cite{AloLo3} extends the results and is inspired, in several aspects, by the results of \cite{MouhVill04}  dealing with smooth kernels that are \textit{not compactly supported}: in such a case, the price to pay for the control  of large velocities consists in additional moments estimates.  The starting point is a suitable Carleman representation of the gain part of $\Q^+$.

\subsection{Carleman representation}

 Let $\mathcal{B}(u,\sigma)$ be a collision kernel of the form
$$
\mathcal{B}(u,\sigma)=\Phi(|u|)b(\widehat{u} \cdot \sigma)
$$
where $\widehat{u}=u/|u|$, $\Phi(\cdot) \geq 0$ and $b(\cdot) \geq 0$ satisfies $\left\|b\right\|_{L^{1}(\mathbb{S}^{d-1})}=1.$ Let us introduce the associated gain part of the collision operator:
$$ \Q^+(f,g)(v) =  \int _{\R^d \times {\S}^{d-1}}   \mathcal{B}(v-v_*,\s) g'_* f'\, \d v_* \,  \d\sigma.$$
Let us also introduce the following linear operator $\Gamma_B$  given by
\begin{equation}\label{expgammaB}
\Gamma_B(\varphi)(x)=\int_{x^\perp} B(z+x,|x|)\varphi(x+z)\d\pi_z,
 \qquad x \in \R^d
 \end{equation}
where $\d\pi_{z}$ is the Lebesgue measure in the hyperplane $x^\perp$ perpendicular to $x$ and
\begin{equation}\label{calB}
B(z,\varrho)=\frac{2^{d-1}\Phi(|z|)}{|z|^{d-2}\varrho}b\left(1-2\frac{\varrho^2}{|z|^2}\right), \qquad \varrho \geq 0, \quad z \in \R^d.
\end{equation}
 Then, one has the following Carleman representation

\begin{lem} For any velocity distribution functions $f,g$ one has
\begin{equation}\label{qfgamma}\begin{split}
\Q^+(g,f)(v)&=\int_{\R^d}  g(w)\d w\int_{(v-w)^{\perp}} B(z-v+w,|v-w|)f(v-z)\d\pi_z\\
&=\int_{\R^d} g(w)\left[\left(t_w  \circ \Gamma_B \circ t_w\right)f\right] (v)\d w\end{split}\end{equation}
where $[t_v \psi](x)=\psi(v-x)$ for any $v,x \in \R^d$ and test-function $\psi$.
\end{lem}
\begin{proof} The proof is rather standard and can be found in several places (see for instance \cite[Lemma 4.1]{AloLo1}); we recall it here for the sake of completeness. Notice simply that, for any test-function $\varphi$, setting $u=v-v_*$,
$$
\int_{\R^d} \Q^+(g,f)(v)\varphi(v)\d v = \int_{\R^d} g(v) \int_{\R^d} f(v-u)
\Phi(|u|)\int_{\S^{d-1}} F_{v,u}\left(\frac{u-|u|\sigma}{2}\right) \d\sigma\d u
\d v$$
where
$$F_{v,u}(z)=\varphi(v-z) b\left(1-\frac{2|z|^2}{|u|^2}\right) \qquad \forall (v,u,z) \in \R^{3d}.$$
Using then the general identity
$$
\int_{\mathbb{S}^{d-1}} F\left(\dfrac{u-|u|\sigma}{2}\right)\d\sigma=\dfrac{2^{d-1}}{|u|^{d-2}}\int_{\R^d} \delta(|x|^2- x \cdot u ) F(x)\d x
$$
valid for any given function $F$ we obtain
\bean
& &\hspace{-8mm} \int_{\R^d} \Q^+(g,f)(v)\varphi(v)\d v \\
& & = \int_{\R^d} g(v) \int_{\R^d} \varphi(v-x)
\int_{\R^d} f(v-u)\, \Phi(|u|)\: \dfrac{2^{d-1}}{|u|^{d-2}}\:
\delta(|x|^2- x \cdot u )
\: b\left(1-\frac{2|x|^2}{|u|^2}\right)\d u\, \d x \,  \d v\\
& & = \int_{\R^d} g(v) \int_{\R^d} \varphi(v-x)
\int_{\R^d} f(v-z-x)\, \Phi(|z+x|)\:
\dfrac{2^{d-1}\, \delta(x\cdot z) }{|z+x|^{d-2}}
\: b\left(1-\frac{2|x|^2}{|z+x|^2}\right)\d z\, \d x \,  \d v
\eean
where we set $u=z+x$. Keeping $x$ fixed, we remove the Dirac mass using the identity
$$\int_{\R^d} h(z)\delta(x \cdot z)\d z =\frac{1}{|x|} \int_{x^\perp} h(z)\d\pi_z,$$
and obtain
\bean
\int_{\R^d} \Q^+(g,f)(v)\varphi(v)\d v & = & \int_{\R^d} g(v)
\int_{\R^d} \varphi(v-x)
\int_{x^\perp} B(z+x,|x|)\, f(v-z-x )\,\d\pi_z\, \d x \,  \d v\\
& = & \int_{\R^d} \varphi(y)\int_{\R^d} g(v)
\int_{(v-y)^\perp} B(v+z-y,|v-y|)\, f(y-z )\,\d\pi_z\,  \d v \,  \d y
\eean
which is the desired result.
\end{proof}

\begin{rmq} Notice that, from the above representation, one sees that $\Gamma_B f=\Q^+(\delta_0,f)$ (see \cite{AloGa}).
\end{rmq}

\subsection{Regularity properties for cut-off collision kernels}

For this section we assume that the kernel $\mathcal{B}(u,\sigma)$ satisfies:
\begin{equation}\label{smoothphi}
\Phi(\cdot) \in \mathcal{C}^\infty(0,\infty), \quad b(\cdot) \in \mathcal{C}_0^\infty(-1,1) \quad \text{ and } \quad \Phi(r )=\left\{
\begin{array}{ccl}
0 & \mbox{for} & r<\epsilon \\
r  & \mbox{for} & r>2\epsilon,
\end{array}
\right.
\end{equation}
for some $\epsilon>0$.

\begin{lem}\label{Sob}
Assume that the collision kernel $\mathcal{B}(u,\sigma)$ satisfies assumption \eqref{smoothphi}. Then, for any {$s \geq \frac{3-d}{2}$}, there exists $C >0$ such that
\begin{equation}\label{QBHS}
\left\|\Q^+(g,f)\right\|_{\mathbb{H}^{s+\frac{d-1}{2}}_\eta} \leq C\,\left\|f\right\|_{\mathbb{H}^s_{\eta+\kappa}}\,\|g\|_{L^1_{2\eta+\kappa}},\qquad \forall \eta \geq 0
\end{equation}
with $\kappa>3/2$ and where the constant $C=C(s,\mathcal{B},\epsilon)$ depends only on $s$ and on the collision kernel $\mathcal{B}$.
\end{lem}
\begin{proof} There is no loss of generality in assuming that $s+\frac{d-1}{2}$ is an integer. {Indeed the general case shall follow thanks to interpolation.} The proof of this estimate follows the approach given in \cite[Theorem 3.1]{MouhVill04} where a similar estimate has been obtained, for $\kappa=0$, under the additional assumption that $\Phi(\cdot)$ has support in $[\epsilon,M]$ with $M<\infty$. Our proof will consist essentially in proving that the weighted estimate (i.e. with $\kappa >0$) allows to take into account large velocities. First, one notices that  the representation formula \eqref{qfgamma} together with Minkowski's inequality leads to
$$\|\Q^+(g,f)\|_{\mathbb{H}^{s+\frac{d-1}{2}}_\eta} \leq \int_{\R^d} |g(w)|\,\|t_w \circ \Gamma_B \circ t_{w} f\|_{\mathbb{H}^{s+\frac{d-1}{2}}_\eta}\d w.$$
Now, since $\|t_w \psi\|_{\mathbb{H}^N_k} \leq 2^{k/2}\langle w \rangle^k \|\psi\|_{\mathbb{H}^N_k}$ for any $\psi \in \mathbb{H}^N_k$, for all $N \in \mathbb{N}$ and any $k \geq 0$, the lemma would follow from
\begin{equation}\label{gammaBHS}
\left\|\Gamma_B(f)\right\|_{\mathbb{H}^{s+\frac{d-1}{2}}_\eta} \leq C\,\left\|f\right\|_{\mathbb{H}^s_{\eta+\kappa}},\qquad \forall \eta \geq 0.
\end{equation}
For $(r,\sigma)\in\R\times \S^{d-1}$, we introduce
\beq\label{coord_spheri}
\widetilde{\Gamma_B(f)}(r,\sigma): = \Gamma_B(f)(r\sigma)= \int_{\R^d} B(u,|r|)f(u) \delta\left(u\cdot\sigma-r\right)\d u.
\eeq
Since we assumed that $s+\frac{d-1}{2}\in\N$, we have
$$\left\|\Gamma_B(f)\right\|^2_{\mathbb{H}^{s+\frac{d-1}{2}}_\eta}
= \sum_{|\ell|\leq s+\frac{d-1}{2}} \| \partial_x^\ell \Gamma_B(f) \|_{L^2_\eta}^2,$$
where for $\ell\in\N^d$,  $\partial_x^\ell=\partial_{x_1}^{\ell_1} \ldots \partial_{x_d}^{\ell_d}$ and $|\ell|=\ell_1+\ldots +\ell_d$. Changing coordinates, we get
\bean
 \| \partial_x^\ell \Gamma_B(f) \|_{L^2_\eta}^2
& = &  \int_0^\infty \int_{\S^{d-1}} | \partial_x^\ell \Gamma_B(f)(r\sigma)|^2 \,
\langle r\rangle^{2\eta}\, r^{d-1} \, \d r\, \d\sigma \\
& = &  \frac{1}{2} \int_{\R} \int_{\S^{d-1}} | \partial_x^\ell \Gamma_B(f)(r\sigma)|^2 \,\langle r\rangle^{2\eta}\, |r|^{d-1} \, \d r\,\d\sigma.
\eean
It follows from \eqref{coord_spheri} that, for $j\in\{1,\ldots, d\}$,
$$\partial_{\sigma_j}\left(\widetilde{\Gamma_B(f)}\right)(r,\sigma) = r
\partial_{x_j}(\Gamma_B(f))(r\sigma).$$
Thus, by induction, we deduce from \eqref{coord_spheri} and the above  equality that
$$ \partial_x^\ell \Gamma_B(f)(r\sigma)=\frac{1}{r^{|\ell|}}\;\partial_\sigma^\ell
\widetilde{\Gamma_B(f)}(r,\sigma)
=\frac{1}{r^{|\ell|}}\int_{\R^d} B(u,|r|) u^\ell f(u) \delta^{(|\ell|)}\left(u\cdot\sigma-r\right)\d u,$$
where $u^\ell=u_1^{\ell_1} \ldots u_d^{\ell_d}$.
One easily checks that for any sufficiently smooth function $\psi$ and for any $p\in\N$,
$$\int_{\R^d} \psi(u,r) \delta^{(p)}(u\cdot \sigma-r)\d u= \sum_{i=0}^p
\left(\bary{c} p \\i \eary \right) (-1)^i \frac{\partial^i}{\partial r^i} \int_{\R^d} \partial^{p-i}_r\psi(u,r)  \delta(u\cdot \sigma-r)\d u.$$
Thus, defining
\beq\label{def_D}
D_\ell(u,r)=u^\ell B(u,|r|) \qquad \mbox{ for } (u,r)\in\R^d\times\R,
\eeq
we get
\begin{multline*}
\left\|\Gamma_B(f)\right\|^2_{\mathbb{H}^{s+\frac{d-1}{2}}_\eta}
\leq \sum_{|\ell|\leq s+\frac{d-1}{2}}  \frac{|\ell|+1}{2} \sum_{i=0}^{|\ell|}
\left(\left(\bary{c} |\ell| \\i \eary \right)\right)^2 \\
\int_{\R} \int_{\S^{d-1}}
\left|  \frac{\partial^i}{\partial r^i} \left( \widetilde{\Gamma_{\partial^{|\ell|-i}_r D_\ell}(f)}\right)(r,\sigma)\right|^2 \langle r\rangle^{2\eta}\,  |r|^{d-1-2|\ell|}\,\d\sigma\, \d r.
\end{multline*}
 Now, for any sufficiently smooth functions $\psi$ and $\varphi$ and for any $i\in\N$, we have
$$\frac{\partial^i \psi }{\partial r^i}(r,\sigma)\varphi(r)=\sum_{k=0}^i (-1)^{k+i}
\left(\bary{c} i \\ k \eary \right) \frac{\partial^k }{\partial r^k}\left( \psi(r,\sigma) \varphi^{(i-k)}(r)\right).$$
Consequently, setting
\beq\label{def_g}
 g(r)= \langle r\rangle^{\eta}\,  |r|^{\frac{d-1}{2}-|\ell|}\qquad \mbox{ for } r\in\R,
\eeq
we obtain
\begin{multline*}
\left\|\Gamma_B(f)\right\|^2_{\mathbb{H}^{s+\frac{d-1}{2}}_\eta}
\leq \sum_{|\ell|\leq s+\frac{d-1}{2}}  \frac{|\ell|+1}{2} \sum_{i=0}^{|\ell|} (i+1)
\left(\left(\bary{c} |\ell| \\i \eary \right)\right)^2 \sum_{k=0}^i
\left(\left(\bary{c} i \\ k \eary \right) \right)^2\\
\int_{\S^{d-1}} \int_{\R}
\left|  \frac{\partial^k}{\partial r^k} \left( \widetilde{\Gamma_{\partial^{|\ell|-i}_r D_\ell}(f)}\right)(r,\sigma) \; g^{(i-k)}(r) \right|^2 \, \d r \,\d\sigma.
\end{multline*}
We introduce the radial Fourier transform $\RF[h]$ for a function $h$ defined
on $\R\times \S^{d-1}$ and the Fourier transform $\F$ in $\R^d$ with the formulas
$$\RF \left[h\right] (\varrho ,\omega) =(2\pi)^{-1/2} \int_{\R}
\exp(i\varrho r) h(r ,\omega) \d r\,,\qquad \forall \varrho \in\R, \;\omega \in \mathbb{S}^{d-1},$$
and $$ \F\left[ f\right](\xi) ={(2\pi)^{-d/2}} \int_{\R^d}
\exp(i v\cdot \xi) f(v) \d v,\qquad \forall \xi \in\R^d.$$
The Plancherel theorem then implies that
\bean
 \int_{\R} \left|  \frac{\partial^k}{\partial r^k} \left( \widetilde{\Gamma_{\partial^{|\ell|-i}_r D_\ell}(f)}\right)(r,\sigma) \;g^{(i-k)}(r)\right|^2 \d r
& =  & \int_{\R} \left| \RF\left[ \frac{\partial^k}{\partial r^k} \left( \widetilde{\Gamma_{\partial^{|\ell|-i}_r D_\ell}(f)}\right)(\cdot,\sigma) \; g^{(i-k)}\right](r) \right|^2  \d r \\
& =  &  \int_{\R}  |r|^{2k} \left| \RF\left[ \widetilde{\Gamma_{\partial^{|\ell|-i}_r D_\ell}(f)}(\cdot,\sigma) \; g^{(i-k)}\right](r) \right|^2  \d r.
\eean
Then, as in \cite{MouhVill04}, we have
$$\RF\left[  \widetilde{\Gamma_{\partial^{|\ell|-i}_r D_\ell}(f)}(\cdot,\sigma) \; g^{(i-k)}\right](r)= (2\pi)^{\frac{d-1}{2}}\F\left[\mathbf{G}^\kappa_{i,k,\ell}(\cdot,\sigma)\langle \cdot \rangle^\kappa f(\cdot)\right](r\sigma)$$
where $\kappa\in \R$ shall be fixed later on and
$$\mathbf{G}^\kappa_{i,k,\ell}(u,\sigma) = \langle u \rangle^{-\kappa} g^{(i-k)}(u\cdot\sigma)\partial^{|\ell|-i}_r D_\ell(u,u\cdot\sigma) ,\qquad  u\in\R^d, \sigma\in\S^{d-1}.$$
Consequently, setting $\xi=r \sigma$, since $\d r \d\sigma =|\xi|^{-(d-1)}\d\xi$ we get
\begin{multline*}
\left\|\Gamma_B(f)\right\|^2_{\mathbb{H}^{s+\frac{d-1}{2}}_\eta}
\leq 2 (2\pi)^{d-1} \sum_{|\ell|\leq s+\frac{d-1}{2}}  \frac{|\ell|+1}{2} \sum_{i=0}^{|\ell|} (i+1)
\left(\left(\bary{c} |\ell| \\i \eary \right)\right)^2 \sum_{k=0}^i
\left(\left(\bary{c} i \\ k \eary \right) \right)^2\\
\int_{\R^d} |\xi|^{2k-d+1}
\left| \F\left[\mathbf{G}^\kappa_{i,k,\ell}\left(\cdot,\frac{\xi}{|\xi|}\right)
\langle \cdot \rangle^\kappa f(\cdot) \right](\xi) \right|^2 \, \d\xi.
\end{multline*}
 Splitting the above integral according to $|\xi|\leq 1$ and $|\xi| > 1$, one gets, as in \cite[p. 183-184]{MouhVill04},
 \begin{multline*}
\int_{\R^d} |\xi|^{2k-d+1}
\left| \F\left[\mathbf{G}^\kappa_{i,k,\ell}\left(\cdot,\frac{\xi}{|\xi|}\right)
\langle \cdot \rangle^\kappa f(\cdot) \right](\xi) \right|^2 \d\xi \leq C_1 \sup_{|\xi| \leq 1}  \left| \F
 \left[ \mathbf{G}^\kappa_{i,k,\ell}\left(\cdot,\frac{\xi}{|\xi|}\right)
\langle \cdot \rangle^\kappa f(\cdot) \right] (\xi)   \right|^2\\
+ C_2
\int_{|\xi| >1} \langle \xi \rangle^{2k-d+1} \left| \F
 \left[ \mathbf{G}^\kappa_{i,k,\ell}\left(\cdot,\frac{\xi}{|\xi|}\right)
\langle \cdot \rangle^\kappa f(\cdot) \right] (\xi)   \right|^2 \d\xi,
\end{multline*}
 for two positive constants $C_1$ (depending only on $d$) and $C_2$ (depending on $d$ and $k$). By Cauchy-Schwartz inequality, we have
$$\left| \F \left[ \mathbf{G}^\kappa_{i,k,\ell}
\left(\cdot,\frac{\xi}{|\xi|}\right)\langle \cdot \rangle^\kappa  f(\cdot) \right] (\xi)   \right|^2
\leq (2\pi)^{-d} \|f\|_{L^2_{\eta+\kappa}}^2 \sup_{\omega\in\S^{d-1}} \|\langle\cdot\rangle^{-\eta}\mathbf{G}^\kappa_{i,k,\ell}\left(\cdot,\omega \right) \|_{L^2(\R^d)}^2.$$
By \cite[Lemma A.5]{MouhVill04}, we obtain for $\displaystyle r_k=\left[\left|k-\frac{d-1}{2}\right|+\frac{d}{2} \right] +1$,
$$\int_{\R^d} \langle \xi \rangle^{2k-d+1} \left| \F
 \left[ \mathbf{G}^\kappa_{i,k,\ell}\left(\cdot,\frac{\xi}{|\xi|}\right)
\langle \cdot \rangle^\kappa f(\cdot) \right] (\xi)   \right|^2 \d\xi \leq C_{s,k,d} \|f\|^2_{\mathbb{H}^s_{\eta+\kappa}}\sup_{\omega\in\S^{d-1}} \|\langle\cdot\rangle^{-\eta}\mathbf{G}^\kappa_{i,k,\ell}\left(\cdot,\omega \right) \|_{\mathbb{H}^{r_k}(\R^d)}^2.$$
This finally leads to \eqref{gammaBHS} with
\begin{equation}\label{imc}
 {C}(s,B)^2= C_s \sum_{|\ell|\leq s+\frac{d-1}{2}}   \sum_{i=0}^{|\ell|}
\left(\left(\bary{c} |\ell| \\i \eary \right)\right)^2 \sum_{k=0}^i
\left(\left(\bary{c} i \\ k \eary \right) \right)^2 \sup_{\omega\in\S^{d-1}} \|\langle\cdot\rangle^{-\eta}\mathbf{G}^\kappa_{i,k,\ell}\left(\cdot,\omega \right) \|_{\mathbb{H}^{r_k}(\R^d)}^2.
\end{equation}
To conclude, it remains only to check that the above quantity is indeed finite, i.e.
$$ \sup_{\omega\in\S^{d-1}} \|\langle\cdot\rangle^{-\eta}\mathbf{G}^\kappa_{i,k,\ell}\left(\cdot,\omega \right) \|_{\mathbb{H}^{r_k}(\R^d)}^2
< \infty,$$
for any multi-index $\ell$ and any integers $i,k$ with
$0\leq k \leq i\leq |\ell| \leq s+\frac{d-1}{2}$. This leads us to investigate
the regularity and integrability properties of the mapping
$$F_\omega\::\: u \in \R^{d}\longmapsto \langle u \rangle^{-\eta-\kappa} g^{(i-k)}(u\cdot\omega)\; \partial^{|\ell|-i}_r D_\ell(u,u\cdot\omega),$$
where  $\ell\in\N^d$,  $i,k\in\N$ with
$0\leq k \leq i\leq |\ell| \leq s+\frac{d-1}{2}$, $D_\ell$ and $g$ are
defined by \eqref{def_D} and \eqref{def_g}
respectively. Observe that
$$F_\omega(u)=\widetilde{b}^{(|\ell|-i)}(\widehat{u}\cdot \omega)R_\omega(u),$$
with
$$\widetilde{b}(x)=\dfrac{ b(1-2x^2)}{|x|} \qquad \mbox{ and } \qquad
R_\omega(u):=2^{d-1}\; \dfrac{u^\ell \Phi(|u|) g^{(i-k)}(u\cdot\omega) }{\langle u\rangle^{\eta+\kappa}|u|^{d+|\ell|-i-1}}.$$
{Because of our cut-off assumptions \eqref{smoothphi}, $b(1-2x^2)=0$ for $|x| \leq \delta$ for some $\delta >0$ and $\Phi(|u|)=0$ for $|u|<\epsilon$. Thus, 
for any
$\omega \in\S^{d-1}$, $F_\omega$ is well-defined and belongs to $\C^\infty(\R^d)$. 
Hence, it suffices to
investigate the integrability properties of the mapping $F_\omega$ and of its
derivatives for large values of $u$ (uniformly with respect to $\omega$).
It is easy to check that any derivative (with respect to $x$) of $\widetilde{b}$ remains bounded on $(-1,1)$ while any $u$-derivative of $R_\omega$ has a faster decay (for $|u| \to \infty$) than $R_\omega(u)$.}
It is then easy to check that
$$| R_\omega(u)| \leq C |u|^{k-|\ell|-\kappa-\frac{d}{2}+\frac{3}{2}}\qquad \forall \omega\in\S^{d-1}, |u|>2\epsilon,$$
for some $C>0$. Since $k\leq |\ell|$, taking $\kappa>3/2$ ensures that $\sup_{\omega\in \mathbb{S}^{d-1}}\|F_\omega(\cdot)\|_{L^2(\R^d_u)} < \infty.$  This  achieves the proof.
\end{proof}

\subsection{{Regularity properties for hard-spheres collision kernel}} We now use the previous result for smooth collision kernels to estimate the regularity properties of $\Q^+(f,g)$ for true hard-spheres interactions. We first recall the following convolution-like estimates for $\Q^+$ as established in \cite{AloCar}:
\begin{theo}[\textbf{Alonso-Carneiro-Gamba \cite{AloCar}}]\label{alogam}
Assume that the collision kernel $\mathcal{B}(u,\sigma)=\Phi(|u|)b(\widehat{u} \cdot \sigma)$ with $\|b\|_{L^1(\mathbb{S}^{d-1})}=1$ and  $\Phi(\cdot) \in L^\infty_{-k}$ for some $k \in \R$ and let $1 \leq p,q,r\leq \infty$ with $1/p+1/q=1+1/r$. Then, for any $\eta \geq 0$, there exists $C_{p,r,\eta,k}(b)$ such that
$$\|\Q^+(f,g)\|_{L^r_\eta} \leq C_{r,p,\eta,k}(b)\left\|\Phi\right\|_{L^\infty_{-k}}\|f\|_{L^p_{\eta+k}}\,\|g\|_{L^q_{\eta+k}}$$
where the constant $C_{r,p,\eta,k}(b)$ is given by
\begin{multline}\label{Crpakb}
C_{r,p,\eta,k}(b)=c_{k,\eta,r}(d)\left(\int_{-1}^1 \left(\dfrac{1-x}{2}\right)^{-\frac{d}{2r'}}b(x)(1-x^2)^{\frac{d-3}{2}}\d x\right)^{\frac{r'}{q'}}\\
\times \left(\int_{-1}^1 \left(\dfrac{1+x}{2}\right)^{-\frac{d}{2r'}}b(x)(1-x^2)^{\frac{d-3}{2}}\d x\right)^{\frac{r'}{p'}}
\end{multline}
for some numerical constant $c_{k,\eta,r}(d)$ independent of $b$ and where $r',p',q'$ are the conjugate exponents of $r,p,q$ respectively.
\end{theo}

We can combine  Theorem  \ref{alogam} together with the estimates of the previous section to prove Theorem \ref{regularite}

\begin{proof}[Proof of Theorem \ref{regularite}] Notice that, for hard-spheres interactions, one has $\mathcal{B}(u,\sigma)=\Phi(|u|)b(\widehat{u} \cdot \sigma)$ with $\Phi(|u|)=|u| \in L^\infty_{-1}$ and $b(x)=b_d$ is constant for any $x \in (-1,1)$. In particular, for any $\eta \geq 0$, both the constant $C_{2,1,\eta,1}(b)$ and $C_{2,2,\eta,1}(b)$ appearing in \eqref{Crpakb} are \textit{finite}. Let us now fix $\eta \geq 0$ and $\varepsilon >0$ and split the kernel into four pieces
\begin{equation}\begin{split}\label{splitting}
\mathcal{B}(u,\sigma)&=\mathcal{B}_{SS}(u,\sigma)+\mathcal{B}_{SR}(u,\sigma)+\mathcal{B}_{RS}(u,\sigma)+\mathcal{B}_{RR}(u,\sigma)\\
&:=\Phi_S(|u|)b_{S}(\hat{u}\cdot\sigma)+\Phi_S(|u|)b_{R}(\hat{u}\cdot\sigma)
+\Phi_R(|u|)b_{S}(\hat{u}\cdot\sigma)+\Phi_R(|u|)b_{R}(\hat{u}\cdot\sigma)
\end{split}\end{equation}
with the following properties:
\begin{itemize}
\item [\it(i)] $b_{S}$ and $\Phi_{S}$ are smooth satisfying the assumptions of the previous Section.
\item [\it(ii)] $b_{R}(s):=b_d-b_{S}(s)$ is the angular remainder satisfying
$$C_{2,1,\eta,1}(b_R)\leq  \varepsilon \qquad \text{ and  } \qquad C_{2,2,\eta,1}(b_R)\leq  \varepsilon .$$
\item [\it(iii)] $\Phi_{R}(|u|)=|u|-\Phi_{S}(|u|)$ is the magnitude remainder satisfying $$\left\|\Phi_{R}\right\|_{L^{\infty}}\leq \dfrac{\varepsilon}{\left(C_{2,1,\eta,1}(b_S)+C_{2,2,\eta,1}(b_S)\right)}.$$
\end{itemize}
Notice that $\Phi_S \in L^\infty_{-1}$ while $\Phi_R \in L^\infty$.
Notice that, in contrast to previous approaches, the last point is made possible because $\Phi_S(|u|)=|u|$ for large $|u|$ which makes $\Phi_R$ compactly supported. Thus, on the basis of relation \eqref{splitting}, one splits $\Q^+$ into the following four parts,
$$
\Q^+=\Q^+_{{SS}}+\Q^+_{SR}  + \Q^+_{RS}  + \Q^+_{RR}.
$$
We shall then deal separately with each of these parts. We prove the result for $s$ such that $s + \frac{d-1}{2} \in \mathbb{N}_*.$ {The general case will follow by interpolation.} First, we know from Lemma \ref{Sob} that, {for $\mu>3/2$},
$$\|\Q^+_{SS}(f,g)\|_{\mathbb{H}^{s+\frac{d-1}{2}}_\eta}\leq C_{s}\|g\|_{\mathbb{H}^s_{\eta+\mu}}\|f\|_{L^1_{2\eta+\mu}}.$$
Second, we estimate $\Q^+_{SR}$. Since
$$\partial^\ell\, \Q^+_{SR}(f,g)=\sum_{\nu=0}^\ell \left(\begin{array}{c}
\l \\ \nu
\end{array}\right)
\ \Q^+_{SR}(\partial^\nu f,\partial^{\l-\nu}g)
$$
for any multi-index $\l$ with $|\l|\leq s+\frac{d-1}{2}$, one gets
$$\|\Q^+_{SR}(f,g)\|_{\mathbb{H}^{ s+\frac{d-1}{2}}_\eta}^2 \leq  C_{s,d} \sum_{|\l|\leq  s+\frac{d-1}{2}}\sum_{\nu=0}^\l\left(\begin{array}{c}
\l \\ \nu
\end{array}\right)\|\Q^+_{{SR}}(\partial^\nu f,\partial^{\l-\nu}g)\|_{L^2_\eta}^2$$
for some $C_{s,d}>0$. We treat differently the cases $|\l|=s+\tfrac{d-1}{2}$ and $|\l| < s+\tfrac{d-1}{2}$.  According to Theorem \ref{alogam} if $|\l|\leq s+\tfrac{d-1}{2}-1$ one has for any $|\nu|\leq |\l|$
\begin{equation*}\begin{split}
\|\Q^+_{{SR}}(\partial^\nu f,\partial^{\l-\nu}g)\|_{L^2_\eta}&\leq  C_{2,1,\eta,1}(b_R)\|\Phi_S\|_{L^\infty_{-1}}\|\partial^\nu f\|_{L^1_{\eta+1}}\,\|\partial^{\l-\nu}g\|_{L^2_{\eta+1}}\\
&\leq   \varepsilon \,\|\partial^\nu f\|_{L^1_{\eta+1}}\,\|\partial^{\l-\nu}g\|_{L^2_{\eta+1}}
\end{split}\end{equation*}
where we used the assumption (ii) with the fact that $\|\Phi_S\|_{L^\infty_{-1}} \leq 1$. Using the general estimate \eqref{taug21} with $\mu=1/2$ for simplicity and since $|\l|\leq s+\tfrac{d-1}{2}-1$,
$$\sum_{|\l| < s+\tfrac{d-1}{2}}\sum_{\nu=0}^\ell \left(\begin{array}{c}
\l \\ \nu
\end{array}\right)\|\Q^+_{{SR}}(\partial^\nu f,\partial^{\l-\nu}g)\|_{L^2_\eta}\leq A_s \varepsilon \,\| f\|_{\mathbb{H}^{s+\frac{d-3}{2}}_{\eta+\frac{d+3}{2}}}\,\|g\|_{\mathbb{H}^{s+\frac{d-3}{2}}_{\eta+1}}$$
for some constant $A_s >0$ depending only on $s$.  In the case $|\ell|=s+\frac{d-1}{2}$, argue in the same way to obtain
$$\|\Q^+_{{SR}}(\partial^\nu f,\partial^{\l-\nu}g)\|_{L^2_\eta}\leq \varepsilon\| f\|_{\mathbb{H}^{s+\frac{d-3}{2}}_{\eta+\frac{d+3}{2}}}\,\|g\|_{\mathbb{H}^{s+\frac{d-3}{2}}_{\eta+1}}$$
for any $0 < |\nu| < |\l|.$ If $\nu=0$ one still has
$$\|\Q^+_{{SR}}(f,\partial^{\l}g)\|_{L^2_\eta}\leq  C_{2,1,\eta,1}(b_R)\| f\|_{L^1_{\eta+1}}\,\|\partial^\l g\|_{L^2_{\eta+1}}$$
additionally, for $\nu=\ell$ we use Theorem \ref{alogam} with $(p,q)=(2,1)$ to get
$$\|\Q^+_{{SR}}(\partial^{\l}f,g)\|_{L^2_\eta}\leq   C_{2,2,\eta,1}(b_R)\| g\|_{L^1_{\eta+1}}\,\|\partial^\l f\|_{L^2_{\eta+1}}.$$
Therefore,
\begin{multline*}
\|\Q^+_{SR}(f,g)\|_{\mathbb{H}^{s+\frac{d-1}{2}}_\eta} \leq
 A_s\,\varepsilon \| f\|_{\mathbb{H}^{s+\frac{d-3}{2}}_{\eta+\frac{d+3}{2}}}\,\|g\|_{\mathbb{H}^{s+\frac{d-3}{2}}_{\eta+1}} \\
+ A_s\,\varepsilon \sum_{|\l|=s+\frac{d-1}{2}}\left(  \| g\|_{L^1_{\eta+1}}\,\|\partial^\ell f\|_{L^2_{\eta+1}}+\| f\|_{L^1_{\eta+1}}\,\| \partial^\ell g\|_{L^{2}_{\eta+1}}\right).\end{multline*}
We argue in the same way using the smallness assumption (ii) to prove that,
\begin{multline*}
\|\Q^+_{RR}(f,g)\|_{\mathbb{H}^{s+\frac{d-1}{2}}_\eta} \leq
 A_s\,\varepsilon \| f\|_{\mathbb{H}^{s+\frac{d-3}{2}}_{\eta+\frac{d+3}{2}}}\,\|g\|_{\mathbb{H}^{s+\frac{d-3}{2}}_{\eta+1}} \\
+ A_s\,\varepsilon \sum_{|\l|=s+\frac{d-1}{2}}\left( \| g\|_{L^1_{\eta+1}}\,\|\partial^\ell f\|_{L^2_{\eta+1}}+\| f\|_{L^1_{\eta+1}}\,\| \partial^\ell g\|_{L^{2}_{\eta+1}}\right).
\end{multline*}
Finally, the estimate for $\Q^+_{RS}$ follows from the fact that $\|\Phi_R\|_{L^\infty}$ is small,
\begin{multline*}
\|\Q^+_{RS}(f,g)\|_{\mathbb{H}^{s+\frac{d-1}{2}}_\eta} \leq A_s\,\varepsilon \| f\|_{\mathbb{H}^{s+\frac{d-3}{2}}_{\eta+\frac{d+1}{2}}}\,\|g\|_{\mathbb{H}^{s+\frac{d-3}{2}}_{\eta}}\\
 + A_s\,\varepsilon  \sum_{|\l|=s+\frac{d-1}{2}}\left( \| g\|_{L^1_{\eta}}\,\|\partial^\ell f\|_{L^2_{\eta}}+\| f\|_{L^1_{\eta}}\,\| \partial^\ell g\|_{L^{2}_{\eta}}\right).
\end{multline*}
Combining all these estimates and replacing $A_s \varepsilon$ to $\varepsilon$  we get \eqref{regestim}.
\end{proof}

We finally recall some useful estimates, of slightly different nature, on the collision operator first established in \cite{BD} and extended to the bilinear (covering also dissipative interactions) in \cite{MiMouh06}:
\begin{theo}[\textbf{Bouchut-Desvillettes} \cite{BD}]\label{theoBD}
For any $s \in \mathbb{R}$ and any $\eta \geq 0$, there exists $C=C(s) > 0$ such that
$$\|\Q^+(f,g)\|_{\mathbb{H}^{s+\frac{d-1}{2}}_\eta} \leq C(s) \left(\|f\|_{\mathbb{H}^s_{\eta+2}}\,\|g\|_{\mathbb{H}^s_{\eta+2}} + \|f\|_{L^1_{\eta+2}}\,\|g\|_{L^1_{\eta+2}}\right).$$
\end{theo}
\begin{rmq} The original statement is for $s \geq 0$ but a direct inspection of the Fourier-based proof shows that the above statement is valid for any $s \in \mathbb{R}$.
\end{rmq}

\section{Reminder about the existence of the self-similar solution}

We recall here, for the sake of completeness, the main steps in the construction of a solution $\psi_\a$ to \eqref{tauT} with unit mass and energy equal to $d/2.$ We also revisit slightly our proof in order to sharpen the range of parameters $\alpha$ for which such a steady solution is known to hold. We recall that the solution $\psi_\a$ constructed in \cite{jde} is obtained through a dynamical proof and the application of a time-dependent version of Tykhonov's fixed point Theorem. Therefore, the core of the analysis of \cite{jde} is the study of the well-posedness and the properties of the flow associated to the following  time-dependent \textit{annihilation equation}
\begin{equation}\label{BEscaled}
\partial_t \psi(t,\xi) + \mathbf{A}_\psi(t)\,\psi(t,\xi) + \mathbf{B}_\psi(t) \,\xi \cdot \nabla_\xi \psi(t,\xi)=\mathbb{B}_\al(\psi,\psi)(t,\xi)
\end{equation}
supplemented with some nonnegative initial condition
\beq\label{CI}
\psi(0,\xi)=\psi_0(\xi),
\eeq
where $\psi_0$ satisfies
\beq\label{massenergie}
\int_{\R^d} \psi_0(\xi)\, \d\xi=1, \qquad
\int_{\R^d} \psi_0(\xi)\,|\xi|^2\,  \d\xi=\frac{d}{2},
\eeq
while
$$\mathbf{A}_\psi(t)=-\frac{\alpha}{2}
\int_{\R^d}\left(d+2-2|\xi|^2\right)\Q^-(\psi,\psi)(t,\xi)\d \xi,$$
and
$$\mathbf{B}_\psi(t)=-\frac{\alpha}{2d}\int_{\R^d}
\left(d-2|\xi|^2\right)\Q^-(\psi,\psi)(t,\xi)\d \xi.$$
{We assume that the initial datum $\psi_{0}$ is nonnegative, isotropic and such that
\beq\label{hypini}\psi_0 \in L^1_{10+d+4\kappa}(\R^d) \cap  L^2_{\frac{d+9}{2}+2\kappa}(\R^d)\cap \mathbb{H}^1_{\frac{7+d}{2}+\kappa}(\R^d)
\eeq
for some $\kappa>0$.} Under such an assumption, there exists a unique 
nonnegative solution $\psi \in\C([0,\infty);L_2^1(\R^d))
 \cap L^1_{\mathrm{loc}}((0,\infty);L^1_{4}(\R^d))
\cap L^\infty_{\mathrm{loc}}((0,\infty);L^1_{3}(\R^d))$ 
to \eqref{BEscaled} such that $\psi(0,\cdot)=\psi_0$ and
$$\int_{ \R^d} \psi(t,\xi) \, \d\xi= 1, \qquad
\int_{ \R^d} \psi(t,\xi) \,|\xi|^2 \, \d\xi= \frac{d}{2} \qquad \forall t \geq 0.$$
Notice that, in this Appendix, we shall always assume that $\psi_0$ is {\it an isotropic function} of $\xi$ so that, for any $t \geq 0$, the solution $\psi(t,\xi)$ is still an isotropic function. In particular,
\begin{equation}\label{null}\int_{\R^d} \psi(t,\xi)\,\xi\,\d\xi=0 \qquad \forall t \geq 0.\end{equation}
According to \cite[Proposition 3.4]{jde}, we set $\alpha_0=\dfrac{1-\varrho_{\frac{3}{2}}}{\frac{3}{2}-\varrho_{\frac{3}{2}}} \in (0,1]$ where
\begin{equation}\label{varrhoK}
\varrho_k =  \int_{\S^{d-1}}\left[ \left(\dfrac{1+ \hat{U} \cdot \sigma}{2}\right)^k+\left(\dfrac{1- \hat{U} \cdot \sigma}{2}\right)^k\right]\dfrac{\d\sigma }{|\S^{d-1}|} \qquad \forall k \geq 0.
\end{equation}
Then, if $0 <\a < \alpha_0$, there exists a constant $\overline{M}(\alpha)$ depending only on $\alpha$  and $d$ such that the unique solution $\psi(t)$ to \eqref{BEscaled} satisfies
\beq\label{bornemoment}
\sup_{t \geq 0}M_{\frac{3}{2}}(t) \leq \max\left\{M_{\frac{3}{2}}(0),\overline{M}(\alpha)\right\}
\eeq
with
$$ {M}_{k}(t)=\int_{\R^d} \psi(t,\xi) |\xi|^{2k}\, \d\xi,\qquad k\geq 0.$$
Moreover, one checks easily that, if $\overline{\alpha}_0 < \alpha_0$ then, $\sup_{\alpha \in (0,\overline{\alpha}_0)} \overline{M}(\alpha) < \infty.$
\begin{rmq} Notice that, in dimension $d=3$, one has $\varrho_{\frac{3}{2}}=\dfrac{4}{5}$ and $\alpha_0=\dfrac{2}{7}.$\end{rmq}

With this in hands, one can show \cite[Corollary 3.6]{jde}, that, if initially bounded, all moments of $\psi(t)$ will remain uniformly bounded. Moreover, by H\"older's inequality,
$$\frac{d}{2}=M_1(t) \leq \sqrt{M_{1/2}(t)  M_{3/2}(t)}$$
Thus, \eqref{bornemoment} leads to
\begin{equation}\label{CAl}
\inf_{t \geq 0}M_{\frac{1}{2}}(t) \geq \frac{d^2}{4}  \left(\max\left\{M_{\frac{3}{2}}(0),\overline{M}(\alpha)\right\}\right)^{-1}.
\end{equation}
On the other hand, as in the proof of \cite[Lemma 3.10]{jde}, one has
$$\int_{\R^d} \psi(t,\xi_*) |\xi-\xi_*| \d\xi_* \geq  C \left(|\xi| + \int_{\R^d}  \psi(t,\xi_*) |\xi_*| \d\xi_*\right),$$
for some constant $C>0$. Consequently, it follows from \eqref{CAl} that there exists some constant $C_0>0$ such that
\beq\label{outil3}
\int_{\R^d} \psi(t,\xi_*) |\xi-\xi_*| \d\xi_* \geq  C_0 \langle \xi\rangle.
\eeq

The next step for proving the existence of a steady state is to prove propagation and estimates on $L^p$-norms of the solution $\psi$. This is the object of the following in which, with respect to \cite[Theorem 1.6]{jde} we sharpen the range of parameters for which uniform estimates would hold true:
\begin{lem}\label{evol:L2} There exists some explicit $\underline{\alpha}  \in (0,\alpha_0)$ such that, if $\psi_0 \in L^2(\R^d) \cap L^1_3(\R^d)$ then, for any $\alpha \in (0,\underline{\alpha})$, the solution $\psi(t)$ to \eqref{BEscaled} satisfies
$$\sup_{t \geq 0}\|\psi(t)\|_{L^2} \leq  \max\left\{\|\psi_0\|_{L^2},C_2\right\}$$
for some explicit constant $C_2>0$ depending only on $\max\left\{M_{\frac{3}{2}}(0),\overline{M}(\alpha)\right\}.$\end{lem}
\begin{proof} Multiplying \eqref{BEscaled} by $2\psi(t,\xi)$ and integrating over $\R^d$, we get
\begin{equation}\label{dLp}\begin{split}
\dfrac{\d}{\d t}\|\psi(t)\|_{L^2}^2  &+ \left(2\mathbf{A}_\psi(t)-d\mathbf{B}_\psi(t)\right)\|\psi(t)\|_{L^2}^2\\
&=2(1-\alpha)\,\int_{\R^d}\Q^+(\psi,\psi)(t,\xi)\psi(t,\xi)\d \xi
-2\int_{\R^d}\Q^-(\psi,\psi)(t,\xi)\psi(t,\xi)\d \xi. \end{split}\end{equation}
Now performing the same manipulations as in the proof of Proposition \ref{propo:L2} in the case $k=0$, choosing  $\varepsilon  \leq \eta_2(\alpha) C_0/4$ where $C_0$ is now given by \eqref{outil3} and setting $\underline{\alpha}=\min(\alpha_2,\alpha_0)$ where $\alpha_2$ is defined in \eqref{alpha2}, we get  for any $\al\in(0,\underline{\alpha})$,
$$\dfrac{\d}{\d t}\|\psi(t)\|_{L^2}^2  +\frac{\eta_2(\alpha) \, C_0}{2}\; \|\psi(t)\|_{L^2}^2 \leq  K  \|\psi(t)\|_{L^2}^{1+1/d},$$
for some positive constant $K$ independent of $\alpha$.\end{proof}
\begin{rmq} Whenever $d=3$, one has $\alpha_2 \simeq .401$. In particular, $\alpha_2 \geq \alpha_0 $.
\end{rmq}

These uniform estimates on the moments and the $L^2$-norm of the solution enable us to get weak-compactness in $L^1(\R^d)$ thanks to the Dunford-Pettis theorem. It remains now to identify a subset of $L^1(\R^d)$ that is left invariant by the evolution semi-group $(\mathcal{S}_t)_{t \geq 0}$ governing
\eqref{BEscaled}. 
We thus investigate  the regularity properties of the solution to the time-dependent \textit{annihilation equation}
\eqref{BEscaled}-\eqref{CI}. We  show in particular how the approach used in Section \ref{sec:regularity} for the steady solution $\psi_\al$ is robust enough to cover regularity properties of time-dependent solutions to \eqref{BEscaled}. 

We begin with the propagation of weighted $L^2$-norms in the spirit of Proposition \ref{propo:L2}:
\begin{prop}\label{prop:propa-L2} For any $k \geq0$ one has
$$\psi_0 \in L^2_k \cap L^1_{\frac{d(d-3)}{d-1}+k} \Longrightarrow \sup_{t \geq 0} \|\psi(t)\|_{L^2_k}=\ell_k < \infty.$$
\end{prop}
\begin{proof} We only consider here the case $k>0$ since for $k=0$, Proposition 
 \ref{prop:propa-L2} follows from Lemma  \ref{evol:L2}. 
The proof is nothing but a {\it dynamic} version of the proof of Prop. \ref{propo:L2}. Namely, multiply \eqref{BEscaled} by $2 \psi(t,\xi)\langle \xi \rangle^{2k}$ and integrate over $\R^d$. After an integration by parts, one gets
 \begin{multline}\label{outil4}
 \frac{\d}{\d t} \|\psi(t)\|^2_{L^2_k}+ \left(2\mathbf{A}_{\psi}(t) -(d+2k)\mathbf{B}_{\psi}(t)\right)\|\psi(t)\|^2_{L^2_k} + 2k\mathbf{B}_{\psi}(t) \|\psi(t)\|_{L^2_{k-1}}^2=\\
 2(1-\al)\int_{\R^d}\Q^+(\psi,\psi)(t,\xi)\psi(t,\xi)\langle \xi\rangle^{2k}\d\xi -2\int_{\R^d} \Q^-(\psi,\psi)(t,\xi)\psi(t,\xi)\langle \xi\rangle^{2k}\d \xi.
 \end{multline}
Now, according to \cite[Corollary 2.2]{AloGa}, for any $\varepsilon \in (0,1)$, there exists $C_\varepsilon > 0$ such that  $$\int_{\R^d}\Q^+(\psi,\psi)(t,\xi)\psi(t,\xi)\langle \xi\rangle^{2k}\d\xi \leq C_\varepsilon \|\psi(t)\|_{L^1_{\frac{d(d-3)}{d-1}+k}}^{2-1/d}\,\|\psi(t)\|_{L^2_k}^{1+1/d}  + \varepsilon \|\psi(t)\|_{L^1_{k}}\,\|\psi(t)\|_{L^2_{k}}^2.$$
According to \cite[Corollary 3.6]{jde}, since $\psi_0 \in  L^1_{\frac{d(d-3)}{d-1}+k}$ one has
$$\sup_{t \geq 0} \|\psi(t)\|_{ L^1_{\frac{d(d-3)}{d-1}+k}} < \infty$$ and, in turns, $\sup_{t \geq 0}\|\psi(t)\|_{L^1_k} < \infty.$ On the other hand, we have
\begin{equation}\label{C_0}\sup_{t\geq 0} \left|\mathbf{A}_{\psi}(t)\right|\leq C, \quad  \sup_{t\geq 0} \left|\mathbf{B}_{\psi}(t) \right|   \leq C \quad \mbox{ and }\quad  \int_{\R^d}\psi (t,\xi_*)|\xi-\xi_*|\d\xi_* \geq C_0\langle \xi\rangle, \quad\forall \xi \in \R^d,\end{equation}
for some constant $C>0$ and $C_0>0$.
Thus, bounding the $L^2_{k-1}$ norm by the $L^2_k$ one, \eqref{outil4} leads to
 $$\frac{\d}{\d t} \|\psi(t)\|^2_{L^2_{k}}+ 2C_0 \|\psi(t)\|_{L^2_{k+\frac{1}{2}}}\\ \leq C \|\psi(t)\|_{L^2_{k}}^2 + 2\, C_\varepsilon \, \|\psi(t)\|_{L^2_{k}}^{1+1/d}  + 2 \,\varepsilon\,M\,\|\psi(t)\|_{L^2_{k}}^2,$$
for some constants $C >0$ and $M > 0$ (depending on $k$). Now, choosing  $\varepsilon$ such that $ 2 \varepsilon M \leq C_0$ we get the existence of some positive constants $C_{1}  > 0$ and $C_{2} > 0$ (still depending on $k$) such that
 $$\frac{\d}{\d t} \|\psi(t)\|^2_{L^2_{k}}+C_0 \|\psi(t)\|_{L^2_{k+\frac{1}{2}}}^2\leq C_{1}\|\psi(t)\|_{L^2_{k}}^2 + C_{2} \|\psi(t)\|_{L^2_{k}}^{1+1/d}.$$
 Now, one uses the fact that, for any $R > 0$, $$\|\psi(t)\|_{L^2_{k}}^2 \leq (1+R^2)^{k}\|\psi(t)\|_{L^2}^2+R^{-1}\|\psi(t)\|_{L^2_{k+1/2}}^2$$
and, since $\sup_{t\geq 0}\|\psi(t)\|_{L^2} < \infty$ by  \cite[Theorem 1.6]{jde}, one can choose $R > 0$ large enough so that $C_{1} R^{-1}=C_0/2$ to obtain
$$\frac{\d}{\d t} \|\psi(t)\|^2_{L^2_{k}}+ \frac{C_0}{2}\|\psi(t)\|_{L^2_{k+\frac{1}{2}}}^2 \leq C_{3} + C_{2} \|\psi(t)\|_{L^2_{k}}^{1+1/d}.$$
The conclusion follows easily since $1+1/d < 2$.
\end{proof}

We now prove the "propagation" of Sobolev regularity together with the creation of higher-order moments. We begin with first-order derivatives to illustrate the techniques:
\begin{prop}\label{propH1q} {Let $\delta>0$.} There exists some explicit $\underline{\alpha}_1 \in (0,1)$ such that the following holds: for any $\frac{1}{2} \leq q \leq 1+\frac{d}{2}$, if the initial datum $\psi_0$ satisfies \eqref{init} with moreover
$$\psi_0 \in L^1_{q_1}(\R^d) \cap L^2_{q+1+\delta}(\R^d) \cap \mathbb{H}^1_q$$
where $q_1=\max\{2q+\frac{1}{2}+\delta,q+1+\delta+\frac{d(d-3)}{d-1}\}$, then, for any $\alpha \in (0,\underline{\alpha}_1)$, the solution $\psi(t)=\psi(t,\xi)$ to \eqref{BEscaled} satisfies
\begin{equation}\label{H1q}\sup_{t \geq 0}\,\|\psi(t)\|_{\mathbb{H}^1_q} < \infty\end{equation}
and
$$\int_0^T \,\|\psi(t)\|_{\mathbb{H}^1_{q+\frac{1}{2}}}\d t <\infty\qquad \forall T > 0.$$
\end{prop}
\begin{proof} Again the proof is only a dynamic version of the proof of Theorem \ref{theo:HK}. For the solution $\psi(t,\xi)$ to \eqref{BEscaled}, we set
$G_j(t,\xi)=\partial_j\psi(t,\xi)$ for $j\in\{1,\ldots,d\}$. Then, $G_j$ satisfies
$$\partial_t G_j(t,\xi) + \left(\mathbf{A}_\psi(t) + \mathbf{B}_\psi(t)\right) G_j(t,\xi) + \mathbf{B}_\psi(t) \,\xi \cdot \nabla_\xi G_j(t,\xi)=\partial_j \mathbb{B}_\alpha(\psi,\psi)(t,\xi)$$
where one has
$$\partial_j \mathbb{B}_\alpha(\psi,\psi)(t,\xi)=(1-\alpha)\partial_j \Q^+(\psi,\psi)(t,\xi) -\Q^-(\psi,G_j)(t,\xi) -\Q^-(G_j,\psi)(t,\xi).$$
For given $q \geq 1/2$, we multiply this equation by $2\,G_j(t,\xi)\langle \xi \rangle^{2q}$ and integrate over $\R^d$. Then,  after an integration by parts and using \eqref{outil3},  one obtains
 \begin{multline}\label{eq_GJ}
\frac{\d}{\d t} \|G_j(t)\|^2_{L^2_q}+ \left(2 \mathbf{A}_{\psi}(t)+ (2-d-2q) \mathbf{B}_{\psi}(t)\right) \|G_j(t)\|^2_{L^2_q} + 2q \mathbf{B}_{\psi}(t)  \|G_j(t)\|^2_{L^2_{q-1}}  \\
\leq 2 (1-\al)\int_{\R^d}\partial_j\Q^+(\psi,\psi)(t,\xi)G_j(t,\xi)\langle \xi\rangle^{2q}\d\xi\\ - 2C_0\|G_j(t)\|^2_{L^2_{q+\frac{1}{2}}}
-2\int_{\R^d} \Q^-(\psi,G_j)(t,\xi)G_j(t,\xi)\langle \xi\rangle^{2q}\d \xi.
 \end{multline}
Clearly, one has
\begin{equation*}\begin{split}
\int_{\R^d}\left|\partial_j\Q^+(\psi,\psi)(t,\xi)\right|\,|G_j(t,\xi)|\, \langle \xi\rangle^{2q}\d\xi &\leq \|\partial_j \Q^+(\psi,\psi)(t)\|_{L^2_{q-\frac{1}{2}}}\|G_j(t)\|_{L^2_{q+\frac{1}{2}}}\\
&\leq \|\Q^+(\psi,\psi)(t)\|_{\mathbb{H}^1_{q-\frac{1}{2}}}\|G_j(t)\|_{L^2_{q+\frac{1}{2}}}.\end{split}\end{equation*}
Now, using Theorem \ref{regularite} with $s=\frac{3-d}{2}$ and $\kappa=\frac{3}{2}+\delta$, for any $\varepsilon > 0$, there exists $C_\varepsilon > 0$ such that
\begin{multline*}
\|\Q^+(\psi,\psi)(t)\|_{\mathbb{H}^1_{q-\frac{1}{2}}} \leq C_\varepsilon \|\psi(t)\|_{\mathbb{H}^{\frac{3-d}{2}}_{q+1+\delta}}\|\psi(t)\|_{L^1_{2q+\frac{1}{2}+\delta}}+\\
\varepsilon\,\|\psi(t)\|_{L^2_{q+1+\frac{d}{2}}}\,\|\psi(t)\|_{L^2_{q+\frac{1}{2}}}
+2\varepsilon\,\|\psi(t)\|_{L^1_{q+\frac{1}{2}}}\, \sum_{i=1}^d \|G_i(t)\|_{L^2_{q+\frac{1}{2}}}.
\end{multline*}
Since $d \geq 3$, one estimates the $\mathbb{H}^{\frac{3-d}{2}}_{q+1+\delta}$ norm by the $L^2_{q+1+\delta}$ norm and, using Proposition \ref{prop:propa-L2} together with \cite[Corollary 3.6]{jde}, our assumptions on the initial datum implies that
$$\sup_{t \geq 0} \|\psi(t)\|_{L^2_{q+1+\delta}} < \infty \quad \text{ and } \quad \sup_{t \geq 0}\,\|\psi(t)\|_{L^1_{ 2q+\frac{1}{2}+\delta}} < \infty.$$
Therefore, for any $\varepsilon > 0$, there exists $C_1(\varepsilon,q) > 0$ and $C_2(q) > 0$ such that
$$\|\Q^+(\psi,\psi)(t)\|_{\mathbb{H}^1_{q-\frac{1}{2}}} \leq C_1(\varepsilon,q) +  \varepsilon\,C_2(q)\, \sum_{i=1}^d \|G_i(t)\|_{L^{2}_{q+\frac{1}{2}}}.$$
One estimates the last integral in \eqref{eq_GJ} as in the proof of Theorem \ref{theo:HK}; namely, an integration by parts yields
$$| \Q^{-}(\psi,G_j)(t,\xi)| = \psi(t,\xi) \left|\int_{\R^d} \partial_j \psi(t,\xi_*)|\xi-\xi_*| \, \d\xi_* \right| \leq  \psi(t,\xi) \|\psi(t)\|_{L^1}=\psi(t,\xi).$$
Then, Cauchy-Schwarz inequality yields
$$\left|\int_{\R^d}  \Q^{-}(\psi,G_j)(t,\xi)\,G_j(t,\xi)\langle \xi\rangle ^{2q}\d \xi\right| \leq \|\psi(t)\|_{L^2_{q}}\,\|G_j(t)\|_{L^2_q} \leq C_q \,\|G_j(t)\|_{L^2_q}$$
for some positive $C_q >0$ where we used the uniform bounds on the $L^2_q$-norm of $\psi(t)$ provided by Proposition \ref{prop:propa-L2}.
Recall that
\begin{equation*}
2\mathbf{A}_{\psi}(t)+\left(2-d-2q \right)\mathbf{B}_{\psi}(t)=-\frac{\al}{2}\left(d-2q +6\right)\mathbf{a}_{\psi}(t) + \frac{\al}{2}\left(d+2-2q\right)\mathbf{b}_{\psi}(t)\end{equation*}
while $2q\mathbf{B}_{\psi}(t)=-\al\,q\,\mathbf{a}_{\psi}(t) + \al\,q\,\mathbf{b}_{\psi}(t).$
Since $q\leq 1 +\frac{d}{2}$, one may neglect all the terms involving $\mathbf{b}_{\psi}(t)$ to obtain the bound from below:
\bean
& &\left(2\mathbf{A}_{\psi}(t)+\left(2-d-2q\right)\mathbf{B}_{\psi}(t)\right)\|G_j(t)\|_{L^2_q}^2  + 2q\mathbf{B}_{\psi}(t)\,\|G_j(t)\|_{L^2_{q-1}}^2\\
& & \hspace{4cm}\geq -\frac{\al}{2}\left(d + 6\right)\mathbf{a}_{\psi}(t)\,\|G_j(t)\|_{L^2_q}^2 + \al\,q \,\mathbf{a}_{\psi}(t) \left(\|G_j(t)\|_{L^2_q}^2-\|G_j(t)\|_{L^2_{q-1}}^2\right)\\
& & \hspace{4cm}\geq -\frac{\al}{2}\sqrt{d}\left(d+6\right)\|G_j(t)\|_{L^2_q}^2
\eean
using the fact that $\mathbf{a}_\psi(t) \leq \sqrt{d}$ for any $t \geq 0$ (following the arguments of Lemma \ref{estimab}). Thus, \eqref{eq_GJ} reads
\begin{multline*}
\frac{\d}{\d t} \|G_j(t)\|^2_{L^2_q}  -\frac{\al}{2}\sqrt{d}\left(d+6\right)\|G_j(t)\|_{L^2_q}^2 +2C_0\|G_j(t)\|_{L^2_{q+\frac{1}{2}}}^2\\
\leq 2(1-\alpha)C_1(\varepsilon,q)\|G_j(t)\|_{L^2_{q+\frac{1}{2}}} +  \varepsilon\,C_2(q)\,\|G_j(t)\|_{L^2_{q+\frac{1}{2}}}\, \sum_{i=1}^d \|G_i(t)\|_{L^{2}_{q+\frac{1}{2}}} + 2C_q \|G_j(t)\|_{L^2_{q}}
\end{multline*}
where $C_q, C_1(\varepsilon,q)$ and $C_2(q)$ are positive constants independent of $\alpha$ and $t$. Define, for any $k \geq 0$, the semi-norm
$$\|\psi(t)\|_{\overset{\circ}{\mathbb{H}}^1_k}=\left(\sum_{j=1}^d \|\partial_j \psi(t)\|_{L^2_k}^2\right)^{1/2}.$$
Setting $\underline{\alpha}_1:=\min\left\{\underline{\alpha},\frac{4C_0}{\sqrt{d}(d+6)}\right\}$ and summing  over all $j\in\{1,\ldots, d\}$, we get
\begin{multline*}
\frac{\d}{\d t} \|\psi(t)\|_{\overset{\circ}{\mathbb{H}}^1_q}^2+  \frac{\sqrt{d}}{2} (d+6)({\underline{\alpha}_1}-\alpha) \|\psi(t)\|_{\overset{\circ}{\mathbb{H}}^1_{q+\frac{1}{2}}}^2\\
\leq 2C_{1}(\varepsilon,q) \sum_{j=1}^d\|G_j(t)\|_{L^2_{q+\frac{1}{2}}}+    \varepsilon C_2(q) \left(\sum_{j=1}^d\|G_j(t)\|_{L^{2}_{q+\frac{1}{2}}}\right)^2  +2C_q\sum_{j=1}^d \|G_j(t)\|_{L^2_q} \\
\leq 2C_{1}(\varepsilon,q) \sum_{j=1}^d\|G_j(t)\|_{L^2_{q+\frac{1}{2}}}+    d\varepsilon C_2(q) \|\psi(t)\|_{\overset{\circ}{\mathbb{H}}^1_{q+\frac{1}{2}}}^2+2\sqrt{d}C_q\|\psi(t)\|_{\overset{\circ}{\mathbb{H}}^1_q}.
 \end{multline*}
 Using Young's inequality, for any $\bar{\delta} > 0$ one gets
\begin{multline*}
\frac{\d}{\d t} \|\psi(t)\|_{\overset{\circ}{\mathbb{H}}^1_q}^2+  \frac{\sqrt{d}}{2} (d+6)({\underline{\alpha}_1}-\alpha) \|\psi(t)\|_{\overset{\circ}{\mathbb{H}}^1_{q+\frac{1}{2}}}^2\\
\leq \left( 2\bar{\delta}\,C_1(\varepsilon,q)  + d\varepsilon\,C_2(q)\right)\|\psi(t)\|_{\overset{\circ}{\mathbb{H}}^1_{q+\frac{1}{2}}}^2 + \frac{2\,d\,C_1(\varepsilon,q)}{\bar{\delta}} + 2\sqrt{d}C_q\|\psi(t)\|_{\overset{\circ}{\mathbb{H}}^1_q}.\end{multline*}
For any fixed $\alpha < \underline{\alpha}_1$, one can choose first $\varepsilon > 0$ small enough and then $\bar{\delta} > 0$ small enough so that
 $\left( 2\bar{\delta}\,C_1(\varepsilon,q)  + d\varepsilon\,C_2(q)\right)= \frac{\sqrt{d}}{4} (d+6)({\underline{\alpha}_1}-\alpha) $ to get
$$\frac{\d}{\d t} \|\psi(t)\|_{\overset{\circ}{\mathbb{H}}^1_q}^2+ \frac{\sqrt{d}}{4} (d+6)({\underline{\alpha}_1}-\alpha)\|\psi(t)\|_{\overset{\circ}{\mathbb{H}}^1_{q+\frac{1}{2}}}^2 \leq
2\sqrt{d}\,C_q \|\psi(t)\|_{\overset{\circ}{\mathbb{H}}^1_q} + C$$
which yields easily the conclusion.
\end{proof}

Combining all the previous computations, one can find some explicit positive constants $M_1$, $M_2$, $M_3$, $M_4$ and $M_5$ 
such that, if {$\alpha \in (0,\underline{\alpha}_1)$}, then the evolution semi-group $(\mathcal{S}_t)_{t \geq 0}$ governing
\eqref{BEscaled} leaves invariant the convex subset of $L^1_2(\R^d)$
{\begin{multline*}\mathcal{Z}= \left\{ 0 \leq \psi \in L^1_2(\R^d), \hspace{0.3cm} \psi(\xi)=\overline{\psi}(|\xi|) \quad \forall \xi \in \R^d,\quad
         \int_{\R^d} \psi(\xi)\d \xi= 1,\quad \int_{\R^d}\psi(\xi)|\xi|^2\d\xi=\frac{d}{2},\right.\\
           \left.  \int_{\R^d}\psi(\xi)|\xi|^3\d\xi\leq M_1, \quad \|\psi\|_{L^{2}}\leq M_2, \quad   \|\psi\|_{L^2_{\frac{d+9}{2}+2\kappa}}\leq M_3,\right.\\
           \left.  \quad  \int_{\R^d}\psi(\xi)|\xi|^{\max\{\frac{9+d(d-2)}{2}+2\kappa,10+d+4\kappa\}}\d\xi\leq M_4\quad \mbox{ and } \quad
 \|\psi\|_{\mathbb{H}^1_{\frac{7+d}{2}+\kappa}} \leq M_5\right\}. \end{multline*}
and we obtain the first part of Theorem \ref{existence}. Taking $\alpha_\star<\underline{\alpha}_1$, then the constants $M_1$ and $M_3$ may be chosen   \textit{independent of $\alpha\in(0,\alpha_\star)$}, which leads to the second part of  Theorem \ref{existence}.\medskip

Notice that, in dimension $d=3$, $\underline{\alpha}=\frac{2}{7}$. Thus, $\underline{\alpha}_1\leq\frac{2}{7}$.  

\medskip
Let us now extend the above result to higher-order derivatives
\begin{prop}\label{propoHsq} For any $s_2 \geq 1$, setting $s_0=s_2-\frac{d-1}{2}$ and  $s_1=s_2-1$, there exists some explicit $\underline{\alpha}_{s_2} \in (0,1)$ such that the following holds: for any $\frac{1}{2} \leq q \leq s_2+\frac{d}{2}$, if the initial datum $\psi_0$ satisfies \eqref{init} and is such that
$$ \sup_{t\geq 0}\|\psi(t)\|_{\mathbb{H}^{s_1}_{q+1+\frac{d}{2}}} < \infty \quad  \text{ with moreover } \quad \psi_0 \in L^1_{2q+\frac{1+d}{2}}(\R^d) \cap  \mathbb{H}^{s_2}_q$$
then, for any $\alpha \in (0,\underline{\alpha}_{s_2})$, the solution $\psi(t)=\psi(t,\xi)$ to \eqref{BEscaled} satisfies
\begin{equation}\label{Hsq}\sup_{t \geq 0}\,\|\psi(t)\|_{\mathbb{H}^{s_2}_q} < \infty\
\qquad \mbox{ and }\qquad \int_0^T \,\|\psi(t)\|_{\mathbb{H}^{s_2}_{q+\frac{1}{2}}}\d t <\infty\qquad \forall T > 0.
\end{equation}
\end{prop}

\begin{proof} {The proof follows the arguments of Theorem \ref{theo:HK}. More precisely, we prove \eqref{Hsq} for $s_2\in\N$ and the general result follows by interpolation. We proceed by induction on $s_2$. First, for $s_2=1$, we deduce from Proposition \ref{propH1q} that \eqref{Hsq} holds.  Let us assume that for some $s_2\geq 1$, \eqref{Hsq} holds. Let $\al \in (0,\underline{\al}_1)$ be fixed and let $\frac{1}{2} \leq q \leq s_2+ \frac{d}{2}$. }For any multi-index $\ell \in \mathbb{N}^d$ with $|\ell|=s_2$, we set
$$G_\ell(t,\xi)=\partial_\xi^\ell \psi(t,\xi).$$
One sees that $G_\ell$ satisfies
$$\partial_t G_\ell(t,\xi)+ \mathbf{A}_\psi(t)G_\ell(t,\xi)+ \mathbf{B}_\psi(t)\,\partial_\xi^\ell \left(\xi \cdot \nabla \psi(t,\xi)\right)=\partial_\xi^\ell \mathbb{B}_\alpha(\psi,\psi)(t,\xi).$$
Noticing that $\partial_\xi^\ell \left(\xi \cdot \nabla \psi(t,\xi)\right)=\xi \cdot \nabla G_\ell(t,\xi) + |\ell|\,G_\ell(t,\xi)$ we get
$$\partial_t G_\ell(t,\xi) +\left[\mathbf{A}_\psi(t)+|\ell|\,\mathbf{B}_\psi(t)\right]G_\ell(t,\xi)+ \mathbf{B}_\psi(t)\,\xi\cdot\nabla G_\ell(t,\xi)\\
= \partial_\xi^\ell \mathbb{B}_\alpha(\psi,\psi)(t,\xi).$$
Given  $q \geq 1/2$, multiplying the above equation by $\langle \xi\rangle^{2q}\,G_\ell(t,\xi)$ and integrating over $\R^d$ we get
\begin{multline}\label{Eq:Gell}
\frac{1}{2}\dfrac{\d}{\d t}\|G_\ell(t)\|_{L^2_q}^2 + \left[\mathbf{A}_\psi(t)+\left(|\ell|-\frac{d+2q}{2}\right)\,\mathbf{B}_\psi(t)\right]\,\|G_\ell(t)\|_{L^2_q}^2 + q\,\mathbf{B}_\psi(t)\,\|G_\ell(t)\|_{L^2_{q-1}}^2\\
=(1-\al)\int_{\R^d}G_\ell(t,\xi)\,\partial_\xi^\ell\Q^+(\psi,\psi)(t,\xi)\langle\xi\rangle^{2q}\d\xi\\
-\int_{\R^d}G_\ell(t,\xi)\,\partial_\xi^\ell\,\Q^-(\psi,\psi)(t,\xi)\,\langle\xi\rangle^{2q}\d\xi.
\end{multline}
Since $q \geq 1/2$, one has
\begin{equation*}\begin{split}
\int_{\R^d}G_\ell(t,\xi)\,\partial_\xi^\ell\Q^+(\psi,\psi)(t,\xi)\langle\xi\rangle^{2q}\d\xi &\leq \|\partial_\xi^\ell\Q^+(\psi,\psi)(t)\|_{L^2_{q-\frac{1}{2}}}\,\|G_\ell(t)\|_{L^2_{q+\frac{1}{2}}}\\
&\leq \|\Q^+(\psi,\psi)(t)\|_{\mathbb{H}^{s_2}_{q-\frac{1}{2}}}\,\|G_\ell(t)\|_{L^2_{q+\frac{1}{2}}}.
\end{split}\end{equation*}
Let $\delta \in (0,\frac{d}{2}]$. Using now Theorem \ref{regularite}  with $s=s_{0}$ and $\kappa=\frac{3}{2}+\delta$, for any $\varepsilon > 0$, there exists $C=C(\varepsilon,s_2,q,\delta)$ such that
\begin{equation*}\begin{split} 
\|\Q^+(\psi,\psi)(t)\|_{\mathbb{H}^{s_2}_{q-\frac{1}{2}}} &\leq C\,\|\psi(t)\|_{\mathbb{H}^{s_0}_{q+1+\delta}}\,\|\psi(t)\|_{L^1_{2q+\frac{1}{2}+\delta}}\,\\
&+\varepsilon\,\|\psi(t)\|_{\mathbb{H}^{s_1}_{q+1+\frac{d}{2}}}
\,\|\psi(t)\|_{\mathbb{H}^{s_1}_{q+\frac{1}{2}}}
+2\varepsilon\,\|\psi(t)\|_{L^1_{q+\frac{1}{2}}}\sum_{|k|=s_2} \|G_k(t)\|_{L^2_{q+\frac{1}{2}}}.
\end{split}\end{equation*}
Therefore, our assumption together with interpolation imply that there exist $C_1(\varepsilon,s_2,q), C_2(q) > 0$ such that
$$\|\Q^+(\psi,\psi)(t)\|_{\mathbb{H}^{s_2}_{q-\frac{1}{2}}}\,\leq C_1(\varepsilon,s_2,q) + \varepsilon\,C_2(q)\sum_{|k|=s_2} \|G_k(t)\|_{L^2_{q+\frac{1}{2}}}$$
and this shows that, for any $\varepsilon > 0$
\begin{multline}\label{eqGlQ+}
\int_{\R^d}G_\ell(t,\xi)\,\partial_\xi^\ell\Q^+(\psi,\psi)(t,\xi)\langle\xi\rangle^{2q}\d\xi
\leq C_1(\varepsilon,s_2,q)\|G_\ell(t)\|_{L^2_{q+\frac{1}{2}}}\,\\
+ \varepsilon\,C_2(q)\|G_\ell(t)\|_{L^2_{q+\frac{1}{2}}}\sum_{|k|=s_2} \|G_k(t)\|_{L^2_{q+\frac{1}{2}}}
\end{multline}
Now, one estimates $\partial_\xi^\ell \Q^-(\psi,\psi)(t,\xi)$ as in the previous Proposition. Namely, one has
\begin{equation*}
\partial_\xi^{\ell}\Q^{-}\big(\psi,\psi\big)=\sum^{\ell}_{\nu=0}
\left(\begin{array}{c}
\ell \\ \nu
\end{array}\right)
\Q^{-}\big(\partial^\nu_\xi \psi,\partial^{\l-\nu}_\xi\psi\big).
\end{equation*}
For any $\nu$ with $\nu\neq \ell$, there exists $i_0 \in \{1,\ldots,d\}$ such that $\ell_{i_0}-\nu_{i_0} \geq 1$ and integration by parts yields
\begin{equation*}\begin{split}
\left|\Q^{-} \big(\partial^\nu \psi,\partial^{\l-\nu}\psi\big)(t,\xi)\right|&=\left|\partial^\nu \psi(t,\xi)\right|\,\left|\int_{\R^d}  \partial^{\l-\nu} \psi(t,\xi_*)|\xi-\xi_*|\d\xi_*\right|\\
&\leq \left|\partial^{\,\nu} \psi(t,\xi)\right|\,\|\partial^{\,\sigma }\psi(t)\|_{L^1}
\end{split}\end{equation*}
where $\sigma=(\sigma_1,\ldots,\sigma_d)$ is defined with $\sigma_{i_0}=\ell_{i_0}-\nu_{i_0}-1$ and $\sigma_i=\ell_i-\nu_i$ if $i \neq i_0.$  Thus, estimating the weighted $L^1$-norm by an appropriate weighted $L^2$-norm (see \eqref{taug21}) we obtain
\begin{equation*}
\left|\Q^{-} \big(\partial^\nu \psi,\partial^{\l-\nu}\psi\big)(t,\xi)\right|\leq C\,|\partial^\nu \psi(t,\xi)|\,\|\partial^{\,\sigma }\psi(t)\|_{L^2_{\frac{d+1}{2}}}
\end{equation*}
for some universal constant $C >0$ independent of $t$. Our induction hypothesis implies that this last quantity is uniformly bounded.  Hence, using Cauchy-Schwarz inequality we conclude that
\begin{multline}\label{App_Q-nu}
\underset{\nu \neq \ell}{\sum_{\nu =0}^{\ell}}
\left(\begin{array}{c}
\ell \\ \nu
\end{array}\right)\int_{\R^d}  \Q^{-} \big(\partial^\nu \psi,\partial^{\l-\nu}\psi\big)(t,\xi) \,G_\ell(t,\xi)\langle \xi\rangle ^{2q}\d \xi
\\
\leq C_2\sum_{|\nu|< |\ell|}
\left(\begin{array}{c}
\l \\ \nu
\end{array}\right)\|\partial^\nu \psi(t)\|_{L^2_q}\,\|G_\ell(t)\|_{L^2_q}
\leq { C_{q,s_2}}\|G_\ell(t)\|_{L^2_q} \qquad\forall t \geq 0
\end{multline}
for some positive constant { $C_{q,s_2}$} independent of $t$.  Whenever $\nu=\ell$ one has
\begin{equation*}
\int_{\mathbb{R}^{d}}\Q^{-} \big(\partial^{\ell} \psi,\psi\big)(t,\xi)\,\partial^{\ell}\psi(t,\xi)\langle \xi\rangle ^{2q}\d \xi=\int_{\R^d} G_\ell^2(t,\xi)\langle \xi\rangle^{2q}\d \xi\int_{\R^d} \psi(t,\xi_*)\,|\xi-\xi_*|\d \xi_*,
\end{equation*}
thus, thanks to \eqref{C_0} one obtains the lower bound
\begin{equation}\label{App_Q-nu+}
\int_{\mathbb{R}^{d}}\Q^{-} \big(\partial^{\ell} \psi,\psi\big)(t,\xi)\,\partial_\xi^{\l}\psi(t,\xi)\langle \xi\rangle ^{2q}\d \xi\geq C_0\,\|G_\ell(t)\|^{2}_{L^2_{q+\frac{1}{2}}}.
\end{equation}
Gathering \eqref{eqGlQ+}, \eqref{App_Q-nu} and \eqref{App_Q-nu+} with \eqref{Eq:Gell}, we get
\begin{multline}\label{Eq:Gell1}
\frac{1}{2}\dfrac{\d}{\d t}\|G_\ell(t)\|_{L^2_q}^2 + \left[\mathbf{A}_\psi(t)+\left(|\ell|-\frac{d+2q}{2}\right)\,\mathbf{B}_\psi(t)\right]\,\|G_\ell(t)\|_{L^2_q}^2 + q\,\mathbf{B}_\psi(t)\,\|G_\ell(t)\|_{L^2_{q-1}}^2\\
\leq { C_1(\varepsilon,s_2,q)}\|G_\ell(t)\|_{L^2_{q+\frac{1}{2}}}\,
+ \varepsilon\,C_2(q)\|G_\ell(t)\|_{L^2_{q+\frac{1}{2}}}\sum_{|k|=s_2} \|G_k(t)\|_{L^2_{q+\frac{1}{2}}}\\
 +{ C_{q,s_2}}\|G_\ell(t)\|_{L^2_q}-C_0\,\|G_\ell(t)\|^{2}_{L^2_{q+\frac{1}{2}}}.
\end{multline}
Now, noticing that $q\mathbf{B}_\psi(t)=-\frac{\al}{2}q\mathbf{a}_\psi(t)+\frac{\al}{2}q\mathbf{b}_\psi(t)$ and
$$\mathbf{A}_\psi(t)+\left(|\ell|-\frac{d+2q}{2}\right)\,\mathbf{B}_\psi(t)=-\frac{\al}{2}\left[s_0+d+\frac{3}{2}-q\right]\mathbf{a}_\psi(t)+\frac{\al}{2}\left[s_0+d-\frac{1}{2}-q\right]\mathbf{b}_\psi(t)$$
and since $q \leq s_0+d-\frac{1}{2}$, the terms involving $\mathbf{b}_\psi(t)$ can be neglected  to get
\begin{equation*}\begin{split}
\bigg[\mathbf{A}_\psi(t)+ \bigg.&\left.\left(|\ell|-\frac{d+2q}{2}\right)\,\mathbf{B}_\psi(t)\right]\,\|G_\ell(t)\|_{L^2_q}^2 + q\,\mathbf{B}_\psi(t)\,\|G_\ell(t)\|_{L^2_{q-1}}^2 \\
&\geq -\frac{\al}{2}\left[s_0+d+\frac{3}{2}\right]\mathbf{a}_\psi(t)\,\|G_\ell(t)\|_{L^2_q}^2 +\frac{\al\,q}{2}\,\mathbf{a}_\psi(t)\left(\|G_\ell(t)\|_{L^2_q}^2-\|G_\ell(t)\|_{L^2_{q-1}}^2\right)\\
&\geq -\frac{\al}{2}\sqrt{d}\left[s_0+d+\frac{3}{2}\right] \|G_\ell(t)\|_{L^2_q}^2
\end{split}\end{equation*}
and \eqref{Eq:Gell1} becomes
\begin{multline}
\frac{1}{2}\dfrac{\d}{\d t}\|G_\ell(t)\|_{L^2_q}^2 -\frac{\al\sqrt{d}}{2}\left[s_0+d+\frac{3}{2}\right] \|G_\ell(t)\|_{L^2_q}^2 +C_0\,\|G_\ell(t)\|^{2}_{L^2_{q+\frac{1}{2}}}  \\
\leq { C_1(\varepsilon,s_2,q)}\|G_\ell(t)\|_{L^2_{q+\frac{1}{2}}}\,
+ \varepsilon\,C_2(q)\|G_\ell(t)\|_{L^2_{q+\frac{1}{2}}}\sum_{|k|=s_2} \|G_k(t)\|_{L^2_{q+\frac{1}{2}}}\\
 +{ C_{q,s_2}}\|G_\ell(t)\|_{L^2_q}.
\end{multline}
Setting now
$$\underline{\al}_{s_2}=\dfrac{2C_0}{\sqrt{d}(s_0+d+\frac{3}{2})}=\dfrac{2C_0}{\sqrt{d}(s_2+\frac{d}{2}+2)} \qquad \text{ and } \qquad  \Theta_q(t)=\left(\sum_{|\ell|=s_2} \|G_\ell(t)\|_{L^2_{q}}^2\right)^{\frac{1}{2}}$$
we can argue as in the proof of Proposition \ref{propH1q} to get that there exists { $C_{s_2} > 0$} such that
 \begin{multline*}
 \dfrac{1}{2}\dfrac{\d}{\d t}\Theta_{q}^2(t) + \frac{\sqrt{d}}{2}\left[s_0+d+\frac{3}{2}\right]\left(\underline{\al}_{s_2}-\al\right)\Theta_{q+\frac{1}{2}}^2(t)\\
  \leq { C_1(\varepsilon,s_2,q)}{\sum_{|\ell|=s_2}}\|G_\ell(t)\|_{L^2_{q+\frac{1}{2}}} + \varepsilon\,C_2(q)\,{ C_{s_2}}\,\Theta_{q+\frac{1}{2}}^2(t) + { C_{s_2}\,C_{q,s_2}}\Theta_q(t).
 \end{multline*}
 Arguing again as in Prop. \ref{propH1q}, using Young's inequality with a parameter $\delta > 0$ and choosing {$\varepsilon > 0$ and then $\delta > 0$}  small enough, we obtain that
 $$ \dfrac{1}{2}\dfrac{\d}{\d t}\Theta_{q}^2(t) + \frac{\sqrt{d}}{4}\left[s_0+d+\frac{3}{2}\right]\left(\underline{\al}_{s_2}-\al\right)\Theta_{q+\frac{1}{2} }^2(t) \leq C_1\Theta_{q}(t)+ C_2$$
 for some positive constants $C_1,C_2 > 0$ which shows that, for any $\al < \underline{\al}_{s_2}$ the conclusion holds.
\end{proof}

\section{Useful interpolation inequalities}

We collect here several useful interpolation inequalities that are needed in several places in the text. First, one recalls the following consequence of Riesz-Thorin interpolation, we refer to \cite{MouhVill04} for a proof:
\begin{lem}\label{interpolSob} For any $k_1, k_2 \in \mathbb{R}^+$ and any {$s_1, s_2 \in \mathbb{R}$}, the following inequality holds for any smooth $f$:
$$\|f\|_{\mathbb{H}^s_k} \leq C\|f\|_{\mathbb{H}_{k_1}^{s_1}}^{\theta}\,\|f\|_{\mathbb{H}^{s_2}_{k_2}}^{1-\theta}$$
for any $\theta \in [0,1]$ and $k=\theta\,k_1+(1-\theta)k_2$, $s=\theta\,s_1+ (1-\theta)s_2$ and some positive constant $C >0$.
\end{lem}
\begin{rmq} The constant $C >0$ is actually missing in the statement \cite{MouhVill04} but appears clearly from the method of proof.
\end{rmq}
\begin{rmq}\label{rmqinterpolSob} Clearly, using Young's inequality $ab \leq \theta a^{\frac{1}{\theta}} + (1-\theta)b^{\frac{1}{1-\theta}}$, the above inequality shows that there is a constant $C=C_\theta > 0$ such that
$$\|f\|_{\mathbb{H}^s_k} \leq C_\theta\,\left(\|f\|_{\mathbb{H}_{k_1}^{s_1}} + \|f\|_{\mathbb{H}^{s_2}_{k_2}}\right)$$
for $k=\theta\,k_1+(1-\theta)k_2$, $s=\theta\,s_1+ (1-\theta)s_2$ and $\theta \in [0,1].$
\end{rmq}
We also recall the following result from \cite{MiMo3}
\begin{lem}\label{interpolm} For any $k, q \in \mathbb{N}$ and any exponential weight function
$$m(v) := \exp(-a|v|^s) \quad\hbox{for}\quad a  \in (0,\infty),$$
with $s \in (0,1]$, there exists $C > 0$ such that for any $h \in \mathbb{H}^{k^\ddagger} \cap L^1(m^{-12})$
with $k^\ddagger := 8 k + 7 (1+d/2)$
$$\| h \|_{\mathbb{W}^{k,1}_q(m^{-1})} \leq C \, \| h \|_{\mathbb{H}^{k^\ddagger}}^{1/8} \, \,\| h \|_{L^1( m^{-12})}^{1/8} \,  \| h \|_{L^1(m^{-1})}^{3/4}.$$
\end{lem}
{ \begin{rmq} The above Lemma is stated in \cite{MiMo3} for $m(v)=\exp(-a|v|^{s})$ with $a > 0$ and $s \in (0,1)$ but the proof can be extended in a straightforward way to the case $s=1$.\end{rmq}}
\subsection*{Acknowledgments} B. L. acknowledges support of the {\it de Castro Statistics Initiative}, Collegio C. Alberto, Moncalieri, Italy.


\begin{thebibliography}{99}



\bibitem{AloGa}
R. J. \textsc{Alonso} \& I. M. \textsc{Gamba}, Gain of integrability for the Boltzmann
collisional operator. \textit{Kinet. Relat. Models} {\bf 4}: 41--51, 2011.


\bibitem{AloCar}
\textsc{R. J. Alonso, E. Carneiro, \& I. M. Gamba,} Convolution inequalities for the Boltzmann
collision operator, {\it Comm. Math. Phys.}, \textbf{298}: 293--322, 2010.


\bibitem{AloGamJMPA}
\textsc{R. J. Alonso \& I. M. Gamba,} Propagation of {$L^1$} and {$L^\infty$} {M}axwellian
              weighted bounds for derivatives of solutions to the
              homogeneous elastic {B}oltzmann equation, {\it J. Math. Pures Appl.}, \textbf{89}: 575--595, 2008.

\bibitem{AloLo1}
R. J. \textsc{Alonso} \& B. \textsc{Lods}, Free cooling and high-energy tails of granular gases with variable restitution coefficient, \textit{SIAM J. Math. Anal.} \textbf{42}:  2499--2538, 2010.

\bibitem{AloLo3}
R. J. \textsc{Alonso} \& B. \textsc{Lods}, Uniqueness and regularity of steady states of the Boltzmann equation for viscoelastic hard-spheres driven by a thermal bath, \textit{Comm. Math. Sciences} {\bf 11}: 851--906, 2013.



\bibitem{jde}
V. \textsc{Bagland} \& B. \textsc{Lods}, Existence of self-similar profile for a kinetic annihilation model, \textit{J. Differential Equations,}  \textbf{254}: 3023--3080, 2013.



\bibitem{BCL}
{\sc M. Bisi, J.~A. Ca\~{n}iz\'{o} \&  B. Lods,}  Uniqueness in the weakly inelastic regime of the equilibrium state to the Boltzmann equation driven by a particle  bath,  \textit{SIAM J. Math. Anal.} \textbf{43}:  2640--2674, 2011.


\bibitem{Ben-Naim}
 E. \textsc{Ben-Naim}, P. \textsc{Krapivsky}, F. \textsc{Leyvraz}, \& S. \textsc{Redner}, Kinetics of Ballistically Controlled Reactions, \textit{J. Chem. Phys}.\textbf{ 98}: 7284, 1994.


\bibitem{BoGaPa} A. V. \textsc{Bobylev}, I. M. \textsc{Gamba} \& V. \textsc{Panferov}, {Moment
inequalities and high-energy tails for the Boltzmann equations with
inelastic interactions}, \textit{J. Statist. Phys.} {\bf 116}: 1651--1682,
2004.

\bibitem{BD} \textsc{F. Bouchut \& L. Desvillettes}, A proof of the smoothing properties of the positive part of Boltzmann's kernel, \textit{Rev. Mat. Iberoamericana}, {\bf 14}: 47--61, 1998.


\bibitem{coppex04}
\textsc{F. Coppex, M. Droz \& E. Trizac}, Hydrodynamics of probabilistic ballistic annihilation,
\textit{Physical Review E} \textbf{72}, 061102, 2004.

\bibitem{coppex05}
\textsc{F. Coppex, M. Droz \& E. Trizac}, Maxwell and very-hard-particle models for probabilistic ballistic annihilation: Hydrodynamic description,
\textit{Physical Review E} \textbf{72}, 021105, 2005.





\bibitem{Maynar1}
\textsc{M. I. Garcia de Soria, P. Maynar,    G. Schehr,   A. Barrat, \&  E. Trizac,} Dynamics of Annihilation I : Linearized Boltzmann Equation and Hydrodynamics, \textit{Physical Review E} \textbf{77}, 051127, 2008.





\bibitem{Maynar2}
\textsc{P. Maynar,   M.I. Garcia de Soria,   G. Schehr,   A. Barrat, \&  E. Trizac,} Dynamics of Annihilation II : Fluctuations of Global Quantities
\textit{Physical Review E} \textbf{77}, 051128, 2008.



\bibitem{MiMouh06}
S. \textsc{Mischler} \& C. \textsc{Mouhot}, Cooling process for inelastic
Boltzmann equations for hard spheres. II. Self-similar solutions and tail
behavior. {\it J. Stat. Phys.} {\bf 124}: 703--746, 2006.


\bibitem{MiMo2}\textsc{ S. Mischler \&  C. Mouhot},
Stability, convergence to the steady state and elastic limit for the Boltzmann equation for diffusively excited granular media. \textit{Discrete Contin. Dyn. Syst.} \textbf{24} 159--185, 2009. 


\bibitem{MiMo3}\textsc{S. Mischler \& C. Mouhot}, Stability, convergence to self-similarity and elastic limit for the Boltzmann equation for inelastic hard-spheres.  \textit{Comm. Math. Phys.}  \textbf{288}:  431--502, 2009.


\bibitem{Mo}
C. \textsc{Mouhot}, Rate of convergence to equilibrium for the spatially homogeneous Boltzmann equation with hard potentials,
\textit{Comm. Math. Phys.}, \textbf{261}: 629--672, 2006. 


\bibitem{MouhVill04}
C. \textsc{Mouhot} \& C. \textsc{Villani}, Regularity theory for the spatially
homogeneous Boltzmann equation with cut-off. {\it Arch. Ration. Mech. Anal.},
{\bf 173}:  169--212, 2004.





\bibitem{Trizac}
E. \textsc{Trizac}, Kinetics and scaling in ballistic annihilation, \textit{Phys. Rev. Lett.} \textbf{88}, 160601, 2002.


\end{thebibliography}
\end{document}